\documentclass[10pt]{article}
\RequirePackage[left=28mm,right=28mm,top=20mm,bottom=20mm]{geometry}
\usepackage[utf8]{inputenc}
\usepackage[english]{babel}
\usepackage[]{amsmath}
\usepackage{amssymb}
\usepackage{amsthm}
\usepackage[shortlabels]{enumitem} 
\usepackage[square,sort,comma,numbers]{natbib}
\usepackage{afterpage}
\usepackage{graphicx} 
\usepackage{tikz-cd} 
\usepackage{stmaryrd} 
\usepackage{url}
\usepackage{hyperref}
\usepackage{multicol}
\usepackage{xcolor}

\usepackage{setspace}

\newcommand{\End}{\operatorname{End}}
\newcommand{\Hom}{\operatorname{Hom}}

\newcommand{\Ext}{\operatorname{Ext}}
\newcommand{\Tor}{\operatorname{Tor}}
\newcommand{\add}{\!\operatorname{add}}
\newcommand{\pdim}{\operatorname{pdim}}
\newcommand{\m}{\!\operatorname{-mod}} 
\newcommand{\M}{\!\operatorname{-Mod}} 
\newcommand{\proj}{\!\operatorname{-proj}}
\newcommand{\St}{\Delta}
\newcommand{\Cs}{\nabla}
\renewcommand{\L}{\Lambda}
\renewcommand{\l}{\lambda}
\newcommand{\pri}{\mathfrak{p}} 
\newcommand{\id}{\operatorname{id}}
\newcommand{\mi}{\mathfrak{m}} 
\newcommand{\Spec}{\operatorname{Spec}}
\newcommand{\rank}{\operatorname{rank}} 
\renewcommand{\top}{\operatorname{top}}
\newcommand{\Proj}{\!\operatorname{-Proj}}
\newcommand{\MaxSpec}{\operatorname{MaxSpec}} 
\newcommand{\Stsim}{\tilde{\St}}
\newcommand{\im}{\!\operatorname{im}} 
\newcommand{\coker}{\operatorname{coker}}
\newcommand{\sumSt}{\underset{\l\in\L}{\bigoplus}}

\newcommand{\Cssim}{\tilde{\Cs}}
\newcommand{\injdim}{\operatorname{idim}}
\newcommand{\domdim}{\operatorname{domdim}}
\newcommand{\rdomdim}{\!\operatorname{-domdim}}
\newcommand{\R}{\operatorname{R}}

\newcommand{\HN}{\operatorname{HNdim}}
\newcommand{\codomdim}{\operatorname{codomdim}}
\newcommand{\rcodomdim}{\!\operatorname{-codomdim}}
\newcommand{\inj}{\!\operatorname{-inj}} 
\newcommand{\cograde}{\operatorname{rcograde}} 
\newcommand{\sgn}{\operatorname{sgn}}

\newtheorem{numberingthm}{Theorem}[subsection] 
\theoremstyle{definition}
\newtheorem{Def}[numberingthm]{Definition}

\theoremstyle{plain}
\newtheorem{Prop}[numberingthm]{Proposition}
\newtheorem{Theorem}[numberingthm]{Theorem}

\newtheorem{Cor}[numberingthm]{Corollary}
\newtheorem{Lemma}[numberingthm]{Lemma}
\newtheorem{Remark}[numberingthm]{Remark}
\newtheorem*{thmintroduction}{Theorem}

\theoremstyle{remark}

\providecommand{\keywords}[1]
{\scriptsize
	\textbf{\textit{Keywords:}} #1 \normalsize \hfill
}
\providecommand{\msc}[1]
{\scriptsize
	\textbf{\textit{2020 Mathematics Subject Classification:}} #1 \normalsize \hfill
}

\title{On split quasi-hereditary covers and Ringel duality} 
\author{Tiago Cruz} 
\date{}

\newcommand{\Address}{{
		\bigskip
		\footnotesize
		
		TIAGO CRUZ,\par \textsc{Institute of Algebra and Number Theory}\par \textsc{University of Stuttgart,}\par \textsc{Pfaffenwaldring 57, 70569 Stuttgart, Germany,}\par\nopagebreak
		\textit{E-mail address}, T.~Cruz: \texttt{tiago.cruz@mathematik.uni-stuttgart.de}
		}}
		
		\allowdisplaybreaks
		
		\tikzcdset{scale cd/.style={every label/.append style={scale=#1},
				cells={nodes={scale=#1}}}}

\begin{document}

\maketitle

\begin{abstract}
In this paper, we develop two new homological invariants called relative dominant dimension with respect to a module and relative codominant dimension with respect to a module. These are used to establish precise connections between Ringel duality, split quasi-hereditary covers and double centraliser properties. 

These homological invariants are studied over Noetherian algebras which are finitely generated and projective as a module over the ground ring and they are shown to behave nicely under change of rings techniques. It turns out that relative codominant dimension with respect to a summand of a characteristic tilting module is a useful tool to construct quasi-hereditary covers of Noetherian algebras and measure their quality.
In particular, this homological invariant is used to construct split quasi-hereditary covers of quotients of Iwahori-Hecke algebras using Ringel duality of $q$-Schur algebras. Combining techniques of cover theory with relative dominant dimension theory we obtain a new proof for Ringel self-duality of the blocks of the Bernstein-Gelfand-Gelfand  category $\mathcal{O}$.
\end{abstract}
 

\section{Introduction}\label{Introduction}

Quasi-hereditary algebras are an important class of finite-dimensional associative algebras of finite global dimension occurring in areas like algebraic Lie theory, homological algebra and algebraic geometry.
Furthermore, all finite-dimensional algebras are centralizer subalgebras of quasi-hereditary algebras (see \cite{Dlab1989, zbMATH02070262}).
Hence, given a finite-dimensional algebra of infinite global dimension we can ask to resolve it by a quasi-hereditary algebra so that their representation theories are connected by a Schur functor with nice properties.
Such resolutions appear quite frequently in representation theory as split quasi-hereditary covers in the sense of Rouquier \cite{Rouquier2008}.

Well-known examples are Soergel’s Struktursatz (see \cite{zbMATH00005018}) and Schur--Weyl duality between Schur algebras $S_k(d, d)$ and symmetric groups $S_d$ (see for example \cite{zbMATH03708660}). The former says that certain subalgebras of coinvariant algebras are resolved by blocks of the BGG category $\mathcal{O}$ of a complex semi-simple Lie algebra together with Soergel’s combinatorial functor and the latter implies that the group algebra of the symmetric group on $d$ letters, $S_d$, is resolved by the Schur algebra $S_k(d, d)$ together with a Schur functor. In both cases, we are approximating a self-injective algebra with a (split) quasi-hereditary algebra. The quality of these resolutions was determined in \cite{zbMATH05278765, zbMATH05871076} and recently in \cite{p2paper} for the integral case using dominant dimension and relative dominant dimension as developed in \cite{CRUZ2022410}, respectively.

An important property of quasi-hereditary algebras is the existence of an important tilting module known as characteristic tilting module. The endomorphism algebra of such tilting module is again quasi-hereditary (see \cite{MR1128706}) and it is known as the Ringel dual.

In recent years, much interest has been devoted to develop methods and techniques to find if for a given quasi-hereditary algebra its module category is equivalent as highest weight category to the module category of its Ringel dual. Schur algebras $S_k(d, d)$ and blocks of the BGG category $\mathcal{O}$ do possess this sort of symmetry, that is, they are Ringel self-dual \cite{zbMATH00549737, zbMATH01267662}. Therefore, one might wonder if this phenomenon is somehow related with Schur--Weyl duality and Soergel's Struktursatz. In particular, \emph{what role does Ringel duality have in these two dualities?}

A trivial common feature in all these three dualities is the existence of a double centralizer property on a summand of a characteristic tilting module. In \cite{Koenig2001}, Schur--Weyl duality was proved using dominant dimension exploiting the previous fact and the group algebra of the symmetric group being self-injective.

On the other hand, Schur algebras $S(n, d)$ with $n<d$, in general, are not Ringel self-dual. Nonetheless, a version of Schur--Weyl duality still holds between $S(n, d)$ and $S_d$ on a module $M$. The object $M$ is a summand of the characteristic tilting module of $S(n, d)$ and a right $S_d$-module, although it is not necessarily faithful as $S_d$-module.  
This means that the double centralizer property occurs between $S(n, d)$ and a quotient of the group algebra of the symmetric group. In \cite{Koenig2001}, this double centralizer property was proved using the quasi-hereditary structure of $S(n, d)$ and a generalisation of dominant dimension, that we shall discuss below.
In addition, the quasi-hereditary structure of $S_k(n, d)$ still has connections with the representation theory of $S_d$. More precisely, Erdmann in  \cite{zbMATH00681964} showed that standard filtration multiplicities of summands of the characteristic tilting module are related with decomposition numbers of the symmetric group.

This situation raises the following question: \emph{Is Schur--Weyl duality arising from the existence of a quasi-hereditary cover that extends the connection between $S_d$ and $S_k(d, d)$ when $n\geq d$? If so, how is Ringel duality related with such a cover?}

The aim of the present paper is to give precise answers to these questions and develop new techniques on quasi-hereditary covers and Ringel self-duality continuing the approach developed in earlier papers \cite{CRUZ2022410, p2paper}.
This pursuit culminates in the creation of a general theory of relative dominant dimension, new techniques to construct split quasi-hereditary covers and a new perspective to Ringel self-duality of the blocks of the BGG category $\mathcal{O}$.

As discussed in \cite{CRUZ2022410}, the covers mentioned above admit integral versions and the developments of \cite{CRUZ2022410} motivate us to continue using an integral setup. So our focus lies in split quasi-hereditary covers instead of just quasi-hereditary covers, mainly because split quasi-hereditary covers are good objects in the integral setup being quite flexible under change of rings. In particular, every quasi-hereditary algebra over an algebraically closed field is a split quasi-hereditary algebra.

\paragraph*{Motivation to generalise dominant dimension}

A module $M$ (over a finite-dimensional algebra $A$) has dominant dimension at least $n$ if there exists an exact sequence $0\rightarrow M\rightarrow I_1\rightarrow \cdots\rightarrow I_n$ with $I_1, \ldots, I_n$ projective and injective modules over $A$. In \cite{zbMATH05871076}, it was observed that dominant dimension is a crucial tool to understand the quality of quasi-hereditary covers of self-injective algebras, and more importantly of quasi-hereditary covers of symmetric algebras having a simple preserving duality. The former fact is a consequence of Morita-Tachikawa correspondence, a theorem by Mueller \cite{zbMATH03248955}, and the fact that any generator is also a cogenerator over a self-injective algebra. Generator (resp. Cogenerator) means a module whose additive closure contains all finitely generated projective (resp. injective) modules.

Recall that $(A, P)$ is said to be a split quasi-hereditary cover of $B$ if $A$ is a split quasi-hereditary algebra, $P$ is a finitely generated projective $A$-module, $B$ is isomorphic to $\End_A(P)^{op}$ and the canonical map $A\rightarrow \End_B(\Hom_A(P, A))^{op}$ induced by the $A$-module structure on $\Hom_A(P, A)$ is an isomorphism of algebras. But, in general, $\Hom_A(P, A)$ is only a generator. So, $A$ might have dominant dimension equal to zero. So, \emph{how to evaluate the quality of split quasi-hereditary covers, in general?}
Furthermore, quotients of group algebras of the symmetric group are not in general self-injective, thus, a new generalisation of dominant dimension is required for our problem.
\begin{enumerate}[(1)]
	\item \emph{In this paper, we identify a generalisation of dominant dimension to be used as a tool to control the connection between the module category of an algebra $B$ with the module category of the endomorphism algebra of a generator (not necessarily cogenerator) over $B$. 
In particular, it is a tool to determine the quality of (split) quasi-hereditary covers.	
}
\end{enumerate}

Generalisations of dominant dimension have appeared several times in the literature.  In \cite{zbMATH03282730}, the concept of $U$-dominant dimension was introduced, where the additive closure of $U$ replaces the projective-injective modules.  When $U$ is a certain injective module, this invariant was used in \cite{zbMATH07398279} to characterise a generalisation of Auslander-Gorenstein algebras in terms of the existence of tilting-cotilting modules (see also \cite{zbMATH07241704}). 

In \cite{Koenig2001}, this concept was applied to study Schur-Weyl duality between $S_k(n, d)$ and $kS_d$ for a field $k$. That is, in \cite{Koenig2001} it was proved without using invariant theory that the canonical homomorphism, induced by the right action given by place permutation of the symmetric group on $d$-letters $S_d$ on $(k^n)^{\otimes d}$,
$kS_d\rightarrow \End_{S_k(n, d)}((k^n)^{\otimes d})^{op}$ is surjective for every natural numbers $n, d$. Here $S_k(n, d)$ denotes the Schur algebra $\End_{kS_d}((k^n)^{\otimes d}).$ The case $n\geq d$ follows from $(k^n)^{\otimes d}$ being faithful over the self-injective algebra $kS_d$ and so it is also a generator-cogenerator. Therefore, this case can be seen as a consequence of Morita-Tachikawa correspondence and the pair $(S_k(n, d), (k^n)^{\otimes d})$ being a (split) quasi-hereditary cover. In the case $n<d$, $(k^n)^{\otimes d}$ is no longer projective, in general, over the Schur algebra but still belongs to the additive closure of the characteristic tilting module. To obtain the assertion then they used the $(k^n)^{\otimes d}$-dominant dimension together with the quasi-hereditary structure of the Schur algebras to transfer the double centralizer property from the easier case $n\geq d$ to the case $n<d$.  This technique also works in the quantum case replacing the Schur algebras by $q$-Schur algebras. The problem in this case is the absence of a characterisation in terms of homology or cohomology like in the classical dominant dimension. Moreover, it is not strong enough to give us information if some cover is lurking around.  

In \citep{CRUZ2022410}, an integral version of dominant dimension was proposed to extend the theory of dominant dimension of finite-dimensional algebras to Noetherian algebras which are projective over the ground ring. These developments raise the following question:
\begin{enumerate}[(1)] \setcounter{enumi}{1}
	\item \emph{Is the above case $n<d$ a particular case of some quasi-hereditary cover? If so, does such cover admit an integral version?}
\end{enumerate}

\paragraph*{Contributions and main results}

We propose and investigate a generalisation of the relative dominant dimension introduced in \cite{CRUZ2022410} that we call relative dominant dimension of a module $M$ with respect to a module $Q$ over a projective Noetherian algebra (Noetherian algebra whose regular module is also projective over the ground ring). We denote it by $Q\rdomdim_{(A, R)} M$. The projective relative injective modules are replaced by the modules in the additive closure of $Q$ using again exact sequences which split over the ground ring and imposing an extra condition: the exact sequences considered should also remain exact under $\Hom_A(-, Q)$. Such condition trivially holds for the relative dominant dimension studied in \cite{CRUZ2022410}. This extra condition is also trivial when $U$ is injective, and so the concept discussed here coincides with $U$-dominant dimension in the cases that $U$ is an injective module. This invariant also generalises the concept of faithful dimension studied in \cite{zbMATH01218841}. 
This relative dominant dimension with respect to a module admits a relative version of a theorem by Mueller (see \citep[Lemma 3]{zbMATH03248955}) which characterises dominant dimension in terms of cohomology:

\begin{thmintroduction}[Theorem ~\ref{moduleMuellerpartone} and Theorem ~\ref{dualmoduleMuellerpartone}]
	Let $R$ be a commutative Noetherian ring. Let $A$ be a projective Noetherian $R$-algebra. Assume that $Q$ is a finitely generated left $A$-module which is projective over $R$  so that $DQ\otimes_A Q$ is projective over $R$. Denote by $B$ the endomorphism algebra $ \End_A(Q)^{op}$. For any finitely generated $A$-module $M$ which is projective over $R$, the following assertions are equivalent.
	\begin{enumerate}[(i)]
		\item $Q\rdomdim_{(A, R)} M\geq n\geq 2$;
		\item The map $\Hom_A(DQ, DM)\otimes_B DQ\rightarrow DM$, given by $f\otimes h\mapsto f(h)$, is an $A$-isomorphism and $\Tor_i^B(\Hom_A(DQ, DM), DQ)=0$ for all $1\leq i\leq n-2$.
	\end{enumerate}
\end{thmintroduction}
This result extends \citep[Proposition 2.2]{zbMATH01218841}, \citep[Corollary 2.16(1),(2)]{zbMATH06409569}   and  \citep[Theorem 5.2]{CRUZ2022410}.

In the same vein as in the generalisation studied in \citep[Section 6]{CRUZ2022410} for nice enough modules $Q, M$ computations of relative dominant dimension can be reduced to computations over finite-dimensional algebras over algebraically closed fields (see Theorem \ref{changeofringrelativedomdimrelativetomodule}). 

This dimension should not be studied alone and in fact it is useful to study also its dual version: the relative codominant dimension of a module with respect to another module. To understand (1), we study relative codominant dimension of a characteristic tilting module with respect to a summand of a characteristic tilting module. This leads us to one of our main results:

\begin{thmintroduction}[Theorem ~\ref{ringeldualitycovers} and Corollary ~\ref{codominantdimensioncontrollfinitedimensionalcase}]
	Let $R$ be a commutative Noetherian ring. Let $A$ be a split quasi-hereditary $R$-algebra with standard modules $\St(\l)$, $\l\in \L$, and with a characteristic tilting module $T$. Denote by $R(A)$ the Ringel dual of $A$: $\End_A(T)^{op}$. Assume that $Q$ is in the additive closure of $T$ and write $B$ to denote the endomorphism algebra $\End_A(Q)^{op}$.
	If $Q\rcodomdim_{(A, R)} T \geq 2$, then $(R(A), \Hom_A(T, Q))$ is a split quasi-hereditary cover of $B$ and the Schur functor \linebreak$F=\Hom_{R(A)}(\Hom_A(T, Q), -)\colon R(A)\m\rightarrow B\m$ induces isomorphisms 
	\begin{align*}
		\Ext_{R(A)}^j(M, N)\rightarrow \Ext_B^j(FM, FN), \quad \forall \  0\leq j\leq Q\rcodomdim_{(A, R)} T  -2,
	\end{align*}
	for all modules $M, N$ admitting a filtration by standard modules over the Ringel dual of $A$. The converse holds if $R$ is a field.
\end{thmintroduction}

Theorem \ref{ringeldualitycovers} clarifies that the relative dominant (as well as codominant) dimension of a characteristic tilting module with respect to some summand $Q$ of a characteristic tilting module measures how far $Q$ is from being a characteristic tilting module. Equivalently, it implies that the faithful dimension of $Q$ measures how far $Q$ is from being a characteristic tilting module.

Recall that in Rouquier's terminology if the Schur functor associated to some split quasi-hereditary cover $(A, P)$ is fully faithful on standard modules, then $(A, P)$ is called a 0-faithful (split quasi-hereditary) cover of the endomorphism algebra of $P$. Theorem ~\ref{ringeldualitycovers} and Corollary ~\ref{codominantdimensioncontrollfinitedimensionalcase} then imply that any $0$-faithful split quasi-hereditary cover of a finite-dimensional algebra can be detected using relative codominant dimension with respect to a summand of a characteristic tilting module and furthermore the quality of such a cover is completely controlled by this generalisation of codominant dimension. In the cases, that quasi-hereditary finite-dimensional algebra that constitutes the cover admits a simple preserving duality, the relative dominant dimension of a characteristic tilting module with respect to $Q$ coincides with the relative codominant dimension of a characteristic tilting module with respect to $Q$, where $Q$ is a summand of a characteristic tilting module.

We answer Question (2) by picking $A$ to be the $q$-Schur algebra $S_{R, q}(n, d)$ on Theorem \ref{ringeldualitycovers} and fixing $Q$ to be $(R^n)^{\otimes d}$. 

Technically, the relative dominant dimension with respect to $(R^n)^{\otimes d}$ is different from the one used in \cite{Koenig2001} but the approach taken in \cite{Koenig2001} also works and perhaps even better with the setup that we investigate here. In fact, using the Schur functor from the module category over a bigger $q$-Schur algebra $S_{k, q}(d, d)$ to the module category over $S_{k, q}(n, d)$, we can deduce that $(R^n)^{\otimes d}\rdomdim_{(S_{R, q}(n, d), R)} T$ is at least half of the relative dominant dimension $\domdim {(S_{R, q}(d, d), R)}$. The computation of the relative dominant dimension of $S_{R, q}(d, d)$ is due to \cite{zbMATH05871076, zbMATH07050778, CRUZ2022410}. In particular, the relative dominant dimension of $S_{R, q}(d, d)$ is at least two independently of $R$ and $d$.  Combining these developments with deformation techniques,  we obtain:

\begin{thmintroduction}[Theorem \ref{maintheoremforRingeldualSchur}]
	Let $R$ be a commutative Noetherian regular ring with invertible element $u\in R$ and $n, d$ be natural numbers. Put $q=u^{-2}$.  Let $T$ be a characteristic tilting module of $S_{R, q}(n,d)$. Denote by $R(S_{R, q}(n, d))$ the Ringel dual $\End_{S_{R, q}(n,d)}(T)^{op}$ of the $q$-Schur algebra $S_{R, q}(n, d)$ (there are no restrictions on the natural numbers $n$ and $d$).
	Then, 
	\begin{itemize}
		\item $(R(S_{R, q}(n, d)), \Hom_{S_{R, q}(n, d)}(T, (R^n)^{\otimes d}))$ is a split quasi-hereditary cover of $\End_{S_{R, q}(n, d)}(V^{\otimes d})^{op}$;
		\item $(R(S_{R, q}(n, d)), \Hom_{S_{R, q}(n, d)}(T, (R^n)^{\otimes d}))$ is an $((R^n)^{\otimes d}\rdomdim_{(S_{R, q}(n, d), R)} T-2)$-faithful (split \linebreak quasi-hereditary) cover of $\End_{S_{R, q}(n, d)}(V^{\otimes d})^{op}$ in the sense of Rouquier.
	\end{itemize}
\end{thmintroduction}

The existence of this split quasi-hereditary cover  clarifies why the quasi-hereditary structure of the Ringel dual of the Schur algebra can be used to study the decomposition numbers of the symmetric group \cite{zbMATH00681964}. 
This result also explains why \cite{Koenig2001} were successful in using tilting theory to establish Schur-Weyl duality (see Remark \ref{RemarkKSX}).

If $n\geq d$, the $q$-Schur algebra is Ringel self-dual (see \citep[Proposition 3.7]{zbMATH00549737}, and \citep[Proposition 4.1.4, Proposition 4.1.5]{MR1707336}) and so this split quasi-hereditary cover constructed in Theorem \ref{maintheoremforRingeldualSchur} is equivalent to the split quasi-hereditary cover $(S_{R, q}(n, d), (R^n)^{\otimes d})$ whose quality was completely determined in \citep[Subsections 7.1, 7.2]{p2paper}. 
In general, the usual strategy for $q=-1$ is not sufficient. In such a case, going integrally is crucial and we make use of deformation techniques (see Corollary \ref{deformationringeldualfunctor}). 

The tools that we study here on relative dominant dimension can be used to reprove the Ringel self-duality of the BGG category $\mathcal{O}$ (Theorem \ref{RingelselfdualityofO}).

In \cite{zbMATH01267662}, Soergel proved that the blocks of the BGG category $\mathcal{O}$ are Ringel self-dual by constructing the explicit functor giving Ringel self-duality. Unfortunately, such proof does not offer much information on which structural properties of $\mathcal{O}$ force its blocks to be Ringel self-dual. Later in \cite{MR1785327} a different proof of Ringel self-duality of the blocks of the BGG category $\mathcal{O}$ was presented using as main tool the Enright completion functor. 

In \citep[Subsection 7.3]{p2paper}, the author studied projective Noetherian algebras $A_{\mathcal{D}}$ that encode the representation theory of any block of the BGG category $\mathcal{O}$. The Ringel self-duality of blocks of the BGG category $\mathcal{O}$ is reproved by applying 
Theorem \ref{ringeldualitycovers} to the algebras $A_{\mathcal{D}}$ making use of integral versions of Soergel's Struktursatz.
 In particular, we can now regard Ringel self-duality of the blocks of the BGG category $\mathcal{O}$ as an instance of uniqueness of covers from Rouquier's cover theory.

\paragraph*{Organisation} This paper is structured as follows: Section \ref{Preeliminaries} sets up the notation, properties and results on covers, split quasi-hereditary algebras, relative dominant dimension over Noetherian algebras and approximation theory to use later. In Section \ref{Relative (co-)dominant dimension with respect to a module}, we give the definition of relative dominant (resp. codominant)  dimension with respect to a module over a projective Noetherian algebra (Definition \ref{relativedomdimrelativedef}). In Subsection \ref{Relative Mueller's characterization of relative dominant dimension with respect to a module}, we explore a characterisation of relative dominant (resp. codominant) dimension with respect to an $A$-module $Q$ in terms of homology over $\End_A(Q)^{op}$ (Theorems \ref{moduleMuellerpartone}-\ref{moduleMuellerparttwo}). As an application of this characterisation, we see how relative dominant dimension with respect to a module varies on exact sequences and long exact sequences in general.
In Subsection \ref{Change of rings on relative dominant dimension with respect to a module}, we study how computations of the relative dominant dimension $Q\rdomdim_{(A, R)} M$ can be reduced to computations over finite-dimensional algebras over algebraically closed fields under mild assumptions on $Q$ and $M$ (Theorem \ref{changeofringrelativedomdimrelativetomodule} and Lemma \ref{relativedominantdimensionalgebraicclosuremodule}). 
In Section \ref{The reduced grade with respect to a module}, we discuss the relation between relative dominant dimension with respect to a module with the concept of reduced cograde with respect to a module. In Section \ref{Fully faithfulness of }, we investigate relative codominant dimension with respect to a module as a tool to establish double centralizer properties (Lemmas \ref{fullyfaithfulnesstwo} and \ref{cocoversdef}), to discover (Theorem \ref{ringeldualitycovers} and Corollary \ref{deformationringeldualfunctor}) and to control (Corollary \ref{codominantdimensioncontrollfinitedimensionalcase}) the quality of split quasi-hereditary covers. 
In Subsection \ref{Ringel self-duality as an instance of uniqueness of covers}, we discuss how under special conditions dominant dimension can be used as a tool to study Ringel self-duality. In Section \ref{Wakamatsu tilting conjecture for quasi-hereditary algebras}, we clarify what the relative dominant dimension with respect to a summand of a characteristic tilting module measures. In Section \ref{Going from bigger covers to a smaller covers}, we explore how to obtain lower bounds to the quality of a split quasi-hereditary cover of a quotient algebra $B/J$ using the quality of split quasi-hereditary covers of $B$.
In Subsection \ref{Generalized Schur algebras in the sense of Donkin}, we construct a split quasi-hereditary cover of the quotient of the Iwahori-Hecke algebra involved in classical Schur-Weyl duality constituted by a Ringel dual of a $q$-Schur algebra using relative dominant dimension with respect to $(R^n)^{\otimes d}$ (Theorem \ref{maintheoremforRingeldualSchur}). In Subsubsection \ref{BGG category O}, we use cover theory and relative dominant dimension to reprove Ringel self-duality of the blocks of the BGG category $\mathcal{O}$ (Theorem \ref{RingelselfdualityofO}). In Subsubsection \ref{RdSchur algebras}, we see how these techniques can be used for Schur algebras in characteristic distinct from two.

\section{Preliminaries}\label{Preeliminaries}

Throughout this paper, we assume that $R$ is a Noetherian commutative ring with identity and $A$ is a projective Noetherian $R$-algebra, unless stated otherwise. Here, $A$ is called a \textbf{projective Noetherian} $R$-algebra if  $A$ is an $R$-algebra so that $A$ is finitely generated projective as $R$-module.
We call $A$ a \textbf{free Noetherian} $R$-algebra if $A$ is a Noetherian $R$-algebra so that $A$ is free of finite rank as $R$-module. 
The module category of left $A$-modules is denoted by $A\M$. We denote by $A\m$ the full subcategory of $A\M$ whose modules are finitely generated and by $A\Proj$ the subcategory of $A\M$ of projective modules. Given $M\in A\m$ we denote by $\add_A M$ (or just $\add M$) the full subcategory of $A\m$ whose modules are direct summands of a finite direct sum of copies of $M$. We write $A\proj$ to denote $\add A$. 
By $\End_A(M)$ we mean the endomorphism algebra of an $A$-module $M$. By $A^{op}$ we mean the opposite algebra of $A$. We denote by $D_R$ (or just $D$) the standard duality functor ${\Hom_R(-, R)\colon A\m\rightarrow A^{op}\m}$.  We say that $M\in A\m\cap R\proj$ is \textbf{$(A, R)$-injective} if $M\in \add DA$. By $(A, R)\inj\cap R\proj$ we mean the full subcategory of $A\m\cap R\proj$ whose modules are $(A, R)$-injective. By an \textbf{$(A, R)$-exact sequence} we mean an exact sequence of $A$-modules which splits as sequence of $R$-modules. By an \textbf{$(A, R)$-monomorphism} we mean an homomorphism $f\in \Hom_A(M, N)$ that fits into an $(A, R)$-exact sequence of the form $0\rightarrow M\xrightarrow{f} N$. 
By a \textbf{generator}  we mean a module $M\in A\m$ satisfying $A\in \add M$. 
Given $M\in A\m$ we denote by $\pdim_A M$ (resp. $\injdim_A M$) the projective (resp. injective) dimension of $M$. 

\paragraph{Change of rings}
We  denote by $\MaxSpec(R)$ the set of maximal ideals of $R$ and by $\Spec R$ the set of prime ideals of $R$. By $\dim R$ we mean the Krull dimension of $R$. 
We denote by $R_\pri$ the localisation of $R$ at the prime ideal $\pri$, and by $M_\pri$ the localisation of $M$ at $\pri$ for every $M\in A\m$. In particular, $M_\pri\in A_\pri\m$. 
We say that a local commutative Noetherian ring is \textbf{regular} if it has finite global dimension. In such a case, the global dimension coincides with the Krull dimension. We say that a commutative Noetherian ring is \textbf{regular} if for every $\pri\in \Spec R$, $R_\pri$ is regular.
We denote by $R(\mi)$ the residue field $R/\mi\simeq R_\mi/\mi_\mi$. For each $M\in A\m$, by $M(\mi)$ we mean the finite-dimensional module $R(\mi)\otimes_R M$ over $A(\mi)=R(\mi)\otimes_R A$. 
By $R^\times$ we denote the set of invertible elements of $R$. We write $D_{(\mi)}$ to abbreviate $D_{R(\mi)}$ for every $\mi\in \MaxSpec R$.

\paragraph{The functor $F_Q$} Let $Q\in A\m\cap R\proj$ satisfying $DQ\otimes_A Q\in R\proj$. By $F_Q$ (or just $F$ when no confusion arises) we mean the functor $\Hom_A(Q, -)\colon A\m\rightarrow B\m$, where $B$ is the endomorphism algebra $\End_A(Q)^{op}$. In particular,  $B\in R\proj$ and $B\m$ is an abelian category. Hence, $B$ is a projective Noetherian $R$-algebra.
Given two maps $\alpha\in \Hom_A(X, Y)$, $\beta\in \Hom_A(Z, W)$, we say that $f$ and $g$ are \textbf{equivalent} if there are $A$-isomorphisms $f\colon X\rightarrow Z$, $g\colon Y\rightarrow  W$ satisfying $g\circ \alpha=\beta\circ f$.
 By $\mathbb{I}_Q$ (or just $\mathbb{I}$ when no confusion arises) we mean the left adjoint of $F$, $Q\otimes_B -\colon B\m\rightarrow A\m$.
We denote by $\upsilon$ the unit $\id_{B\m}\rightarrow F\mathbb{I}$ and $\chi$ the counit $\mathbb{I}F\rightarrow \id_{A\m}$. Thus, for any $N\in B\m$, $\upsilon_N$ is the $B$-homomorphism $\upsilon_N\colon N\rightarrow \Hom_A(Q, Q\otimes_B N)$, given by $\upsilon_N(n)(q)=q\otimes n$, $n\in N, q\in Q$. For any $M\in A\m$, $\chi_M$ is the $A$-homomorphism $Q\otimes_B \Hom_A(Q, M)\rightarrow M$, given by $\chi_M(q\otimes g)=g(q)$, $g\in FM,  q\in Q$. By projectivization, the restriction of $F$ to $\add Q$ gives an equivalence between $\add Q$ and $B\proj$. Further, for every $X, Y\in A\m$ and every $M, N\in B\m$, $\upsilon_{M\oplus N}$ is equivalent to $\upsilon_M\oplus \upsilon_N$ and $\chi_{X\oplus Y}$ is equivalent to $\chi_X\oplus \chi_Y$, respectively. 	We shall write $\chi^r$ and $\upsilon^r$ for the counit and unit, respectively, of the adjunction $-\otimes_B DQ \dashv \Hom_A(DQ, -)$. Given a left (resp. right) exact functor $H$ between two module categories, we denote by $\R^iH$ (resp. $\operatorname{L}_i H$) the $i$-th right (resp. left) derived functor of $H$ for $i\in \mathbb{N}$.

\paragraph{Filtrations}	Recall that for a given set (possibly infinite) of modules $\Theta$ in $A\m\cap R\proj$, $\mathcal{F}(\Theta)$ denotes the full subcategory of $A\m\cap R\proj$ whose modules admit a finite filtration by the modules in $\Theta$. Given a set of modules $\Theta$ in $A\m\cap R\proj$, we denote by $\tilde{\Theta}$ the set of modules \mbox{$\{\theta\otimes_R X_\theta\colon \theta\in \Theta, X_\theta\in R\proj \}$.} The following lemma allows us to identify the set \mbox{$F_Q\tilde{\Theta}:=\{F_QX\colon X\in \tilde{\Theta} \}$}	 with the set $\widetilde{F_Q\Theta}$, where $F_Q\Theta:=\{F_Q\theta\colon \theta\in \Theta \}$.
	
\begin{Lemma}\label{tensorprojcommutingonHom}
	Let $M, N\in A\m$ and $U\in R\proj$. Then, the $R$-homomorphism \linebreak\mbox{$\varsigma_{M, N, U}\colon \Hom_A(M, N)\otimes_R U\rightarrow \Hom_A(M, N\otimes_R U)$,} given by $ \ g\otimes u\mapsto g(-)\otimes u$ is an $R$-isomorphism.
\end{Lemma}
\begin{proof}
	Given that $M\in A\m$,  $\Hom_A(M, -)$ preserves finite direct sums. Hence, $\varsigma_{M, N, R^t}$ is an isomorphism for every $t\in \mathbb{N}$. 
	Since for all modules $U_1, U_2\in R\m$, $\varsigma_{M, N, U_1}\oplus \varsigma_{M, N, U_2}$ is an isomorphism if and only if $\varsigma_{M, N, U_1\oplus U_2}$ is, the result follows.
\end{proof}

\subsection{Basics on approximations}

	\begin{Def}\label{aproximations}
	Let $T\in A\m$. An $A$-homomorphism $M\rightarrow N$ is called a \textbf{left $\add T$-approximation} of $M$ provided that $N$ belongs to $\add T$ and the induced homomorphism $\Hom_A(N, X)\rightarrow \Hom_A(M, X)$ is surjective for every $X\in \add T$.
	A map $f\in \Hom_A(Y, M)$ is called a \textbf{right $\add T$-approximation} of $M$ if $Y\in \add T$ and $\Hom_A(X, f)$ is surjective for every $X\in \add T$.
\end{Def}

\begin{Lemma}\label{rightapprbyonemodule}
	Let $N, T\in A\m\cap R\proj$ and $M\in \add T$. Then, $f\in \Hom_A(M, N)$ is a right $\add T$-approximation of $N$ if and only if the map $\Hom_A(T, f)\colon \Hom_A(T, M)\rightarrow \Hom_A(T, N)$ is surjective.
\end{Lemma}
\begin{proof}
	Since $\Hom_A(-, N)$ commutes with finite direct sums for every $N\in A\m$, $\Hom_A(T_1\oplus T_2, f)$ is surjective if and only if $\Hom_A(T_1\oplus T_2, f)$ is surjective.
\end{proof}

\begin{Lemma}\label{gofromrighttoleftappro}
	Let $M, T\in A\m\cap R\proj$ and $N\in \add T$. An $A$-homomorphism $f\colon M\rightarrow N$ is a left $\add T$-approximation of $M$ if and only if $Df\colon DN\rightarrow DM$ is a right $\add DT$-approximation of $DM$.
\end{Lemma}
\begin{proof}
	The maps $\Hom_A(T{,} f)$ and $\Hom_A(Df {,} DT)$ are equivalent (see for example \citep[Proposition 2.2]{CRUZ2022410}), and so, the claim follows.
\end{proof}


Let $M, T\in A\m\cap R\proj$.
It is easy to check that an $(A, R)$-exact sequence $X_t\xrightarrow{\alpha_t}\cdots \rightarrow X_1\xrightarrow{\alpha_1}X_0\xrightarrow{\alpha_0}M \rightarrow 0$ remains exact under $\Hom_A(T, -)$ with $X_i\in \add T$ if and only if for every $i=1, \ldots, t$, the induced maps $X_i\twoheadrightarrow \im \alpha_i$ and $\alpha_0$ are right $\add T$-approximations.
Dually, an $(A, R)$-exact sequence $0\rightarrow M\xrightarrow{\alpha_0}X_0\xrightarrow{\alpha_1}X_1\rightarrow\cdots\rightarrow X_t$ remains exact under $\Hom_A(-, T)$ with $X_i\in \add T$ if and only if the $(A, R)$-monomorphisms $\im \alpha_{i+1}\hookrightarrow X_{i+1}$ and $\alpha_0$ are left $\add T$-approximations with $i=0, \ldots, t-1$.

\subsection{Split quasi-hereditary algebras}

Quasi-hereditary algebras were introduced by Cline, Parshall and Scott in \citep{MR961165}, and were defined in terms of the existence of a certain idempotent ideal chain reflecting structural properties of the algebra like the finiteness of the global dimension. In \cite{CLINE1990126}, the concept of quasi-hereditary algebra was generalised to Noetherian algebras. Among them, are the split quasi-hereditary algebras which possess nicer properties with respect to change of ground rings. In particular, every quasi-hereditary algebra over an algebraically closed field is split quasi-hereditary. A module theoretical approach to split quasi-hereditary algebras over commutative Noetherian rings was considered in \cite{Rouquier2008}. In \cite{zbMATH01527053}, a comodule theoretical approach was developed to split quasi-hereditary algebras over commutative Noetherian rings.

\begin{Def}\label{splithwc}
	Given a projective Noetherian $R$-algebra $A$ and a collection of finitely generated left $A$-modules $\{\St(\l)\colon \l\in \L\}$ indexed by a  poset $\L$, we say that $(A, \{\Delta(\lambda)_{\lambda\in \Lambda}\})$ is a \textbf{split quasi-hereditary $R$-algebra} if the following conditions hold:
	\begin{enumerate}[(i)]
				\item The modules $\St(\l)\in A\m$ are projective over $R$.
		\item Given $\l, \mu\in \L$, if $\Hom_A(\St(\l), \St(\mu))\neq 0$, then $\l\leq\mu$.
		\item $\End_A(\St(\l))\simeq R$, for all $\l\in\L$.
		\item Given $\l\in\L$, there is $P(\l)\in A\proj$ and an exact sequence $0\rightarrow C(\l)\rightarrow P(\l)\rightarrow \St(\l)\rightarrow 0$ such that $C(\l)$ has a finite filtration by modules of the form $\St(\mu)\otimes_R U_\mu$ with $U_\mu\in R\proj$ and $\mu>\l$. 
		\item $P=\sumSt P(\l)$ is a progenerator for $A\m$. 
	\end{enumerate}
\end{Def}
Under these conditions, we also say that $(A\m, \{\Delta(\lambda)_{\lambda\in \Lambda}\})$ is a \textbf{split highest weight category}. We use the terms split quasi-hereditary algebra and split highest weight category interchangeably. The modules $\St(\l)$ are known as \textbf{standard modules}. Much of the structure of a split quasi-hereditary algebra is controlled by the subcategory $\mathcal{F}(\Stsim)$, where $\Stsim_A$ or just $\Stsim$ (when there is no confusion on the ambient algebra) denotes the set $\{\St(\l)\otimes_R U_\l\colon \l\in \L, U_\l\in R\proj \}$. This subcategory contains all projective finitely generated $A$-modules, it is closed under extensions, closed under kernels of epimorphisms, and closed under direct summands. Hence, $\mathcal{F}(\Stsim)$ can be viewed as the full subcategory of $A\m\cap R\proj$ whose modules admit a finite filtration by direct summands of direct sum of copies of standard modules.

Given a split quasi-hereditary algebra $(C, \{\Delta_C(\lambda)_{\lambda\in \Omega}\})$, a functor $G\colon A\m\rightarrow C\m$ is called a \textbf{split highest weight equivalence of categories}  if $F$ is an equivalence of categories and there exists a bijection of posets $\Phi\colon \L\rightarrow \Omega$ such that for all $\l\in \L$,  $G\St_A(\l)\simeq \St_C(\Phi(\l))\otimes_R U_\l$ for some invertible $R$-module $U_\l$. In such a case, we say that $A$ and $C$ are \textbf{Morita equivalent as split quasi-hereditary algebras} and that $A\m$ and $C\m$ are \textbf{equivalent as split highest weight categories}.

\paragraph{Split heredity chains}

An alternative way to define split quasi-hereditary algebras is via split heredity chains.
An ideal $J$ is called \textbf{split heredity} of $A$ if $A/J\in R\proj$, $J\in A\proj$, $J^2=J$ and the $R$-algebra $\End_A(J)^{op}$ is Morita equivalent to $R$. A chain of ideals $0=J_{t+1}\subset J_t\subset \cdots\subset J_1=A$ is called \textbf{split heredity} if $J_i/J_{i+1}$ is a split heredity ideal in $A/J_{i+1}$ for $1\leq i\leq t$. The algebra $A$ is split quasi-hereditary if it admits a split heredity chain. Given an increasing bijection $\L\rightarrow \{1, \ldots, t\}$, $\l\mapsto i_\l$, the standard modules and the split heredity chain are related by the following identification: $\im (\tau_i) \simeq J_i/J_{i+1}$, where $\tau_i$ is the map $\St_i\otimes_R \Hom_{A/J_{i+1}}(\St_i, A/J_{i+1})\rightarrow A/J_{i+1}$ given by $\tau_i(l\otimes f)= f(l)$. For more details on this equivalence, we refer to \citep[3.3]{cruz2021cellular} and \citep[Theorem 4.16]{Rouquier2008}.

\subsubsection{Costandard modules and characteristic tilting modules}

A split quasi-hereditary algebra $(A, \{\Delta(\lambda)_{\lambda\in \Lambda}\})$ also comes equipped with a set of modules ${\{\Cs(\l)\colon \l\in \L\}}$ known as \textbf{costandard modules}. These modules satisfy the following properties.
\begin{Prop}	Let $(A, \{\Delta(\lambda)_{\lambda\in \Lambda}\})$ be a split quasi-hereditary algebra. \label{qhproperties}
	The following assertions hold.
	\begin{enumerate}[(i)]
		\item $(A^{op}, \{D\Cs(\lambda)_{\lambda\in \Lambda}\})$ is a split quasi-hereditary algebra;
		\item $\mathcal{F}(\Cssim)=\{X\in A\m\cap R\proj\colon \Ext_A^1(M, X)=0, \ \forall M\in \mathcal{F}(\Stsim) \}$, where $\Cssim $ denotes the set \linebreak${\{\Cs(\l)\otimes_R U_\l\colon \l\in \L, U_\l\in R\proj \}}$;
		\item For any $\mu\neq \l\in \L$, $\Hom_A(\St(\l), \Cs(\l))\simeq R$ and $\Hom_A(\St(\mu), \Cs(\l))=0$;
		\item The choice of costandard modules satisfying the previous assertions is unique up to isomorphism;
		\item $\mathcal{F}(\Stsim)=\{X\in A\m\cap R\proj\colon \Ext_A^1(X, N)=0, \ \forall N\in \mathcal{F}(\Cssim) \}$;
		\item For any $M\in \mathcal{F}(\Stsim)$, the functor $-\otimes_A M\colon \mathcal{F}(D\Cssim)\rightarrow R\proj$ is well defined and exact;
		\item For any $N\in \mathcal{F}(\Cssim)$, the functor $DN\otimes_A -\colon \mathcal{F}(\Stsim)\rightarrow R\proj$ is well defined and exact;
		\item For any $M\in \mathcal{F}(\Stsim)$ and $N\in \mathcal{F}(\Cssim)$, it holds $\Hom_A(M, N)\in R\proj$.
	\end{enumerate}
\end{Prop}
\begin{proof}
	For (i), (ii), (iii), (iv), (v) we refer to \citep[Proposition 3.1, Theorem 4.1]{appendix} and \citep[Proposition 4.19, Lemma 4.21]{Rouquier2008}.  For (vi), (vii), (viii) we refer to \citep[Proposition 4.3., Corollary 4.4.]{appendix}.
\end{proof}

The subcategories $\mathcal{F}(\Stsim)$ and $\mathcal{F}(\Cssim)$ are determined and determine a characteristic tilting module. We call $T$ a \textbf{characteristic tilting module} if it is a tilting module satisfying $\add T=\mathcal{F}(\Stsim)\cap \mathcal{F}(\Cssim)$. By a tilting $A$-module we mean a module of finite projective dimension, $\Ext_A^{i>0}(T, T)=0$ and the regular module has a finite coresolution by modules in the additive closure of $T$. Let $T$ be a characteristic tilting module of $A$. So, we have $\add T=\add \bigoplus_{\l\in \L} T(\l)$ so that each $T(\l)$ with $\l\in \L$ fits into exact sequences of the form
\begin{align}
	0\rightarrow \St(\l)\rightarrow T(\l)\rightarrow X(\l)\rightarrow 0, \quad X(\l)\in \mathcal{F}(\Stsim_{\mu<\l}) \label{eq120b}\\
	0\rightarrow Y(\l)\rightarrow T(\l)\rightarrow \Cs(\l)\rightarrow 0, \quad Y(\l)\in \mathcal{F}(\Cssim_{\mu<\l}).\label{eq121b}
\end{align} 
By $\Stsim_{\mu<\l}$ we mean the set $\{\St(\mu)\otimes_R U_\mu\colon \mu\in \L, \ \mu<\l, U_\mu\in R\proj \}$. Analogously, the set $\Cssim_{\mu<\l}$ is defined.
A fundamental difference between the classical case is that, in general, we cannot choose $T(\l)$ to be indecomposable modules. But two distinct characteristic tilting modules have the same additive closure. These can be chosen to be indecomposable when the ground ring is a local commutative Noetherian ring. The philosophical reason is that any split quasi-hereditary algebra over a local commutative Noetherian ring is semi-perfect (\citep[Theorem 3.4.1]{cruz2021cellular}). Let $T=\bigoplus_{\l\in \L} T(\l)$ be a characteristic tilting module. It follows, by construction, that a module $M\in A\m\cap R\proj$ belongs precisely to $\mathcal{F}(\Stsim)$ if and only if there exists a finite coresolution of $M$ by modules in $\add T=\mathcal{F}(\Stsim)\cap \mathcal{F}(\Cssim)$. In the same way, 
a module $M\in A\m\cap R\proj$ belongs precisely to $\mathcal{F}(\Cssim)$ if and only if there exists a finite resolution of $M$ by modules in $\add T$. For details on these statements, we refer to \citep{appendix} and \citep[Section 3, Appendix A and B]{cruz2021cellular}.

\paragraph{Change of rings} An important feature of split quasi-hereditary algebras is that they behave quite well under change of rings. This manifests itself in the subcategories $\mathcal{F}(\Stsim)$ and $\mathcal{F}(\Cssim)$ as follows:

\begin{Prop}\label{standardscotiltingsreductiontofields}
 Let $(A, \{\Delta(\lambda)_{\lambda\in \Lambda}\})$ be a split quasi-hereditary $R$-algebra. Let \mbox{$M\in A\m$.}  Let $Q$ be a commutative $R$-algebra and Noetherian ring. Then, the following assertions hold.\begin{enumerate}[(a)]
 	\item $(Q\otimes_R A, \{Q\otimes_R\Delta(\lambda)_{\lambda\in \Lambda}\})$  is a split quasi-hereditary algebra over $Q$. The costandard modules of $Q\otimes_R A$ are the form $Q\otimes_R \Cs(\l)$, $\l\in \L$.
		\item $M\in \mathcal{F}(\Stsim)$ if and only if $M(\mi)\in \mathcal{F}(\St(\mi))$ for all maximal ideals $\mi$ of $R$ and $M\in R\proj$.
		\item $M\in \mathcal{F}(\Cssim)$ if and only if $M(\mi)\in \mathcal{F}(\Cs(\mi))$ for all maximal ideals $\mi$ of $R$ and $M\in R\proj$.
		\item Let $T$ be a characteristic tilting module. $M\in \add T$ if and only if $M(\mi)\in \add T(\mi)$ for all maximal ideals $\mi$ of $R$ and $M\in R\proj$.
		\item Let $M\in \mathcal{F}(\Stsim)$ and let $N\in \mathcal{F}(\Cssim)$. Then, $Q\otimes_R \Hom_A(M, N)\simeq \Hom_{Q\otimes_RA}(Q\otimes_R M, Q\otimes_R N).$
	\end{enumerate}  
\end{Prop}
\begin{proof}
	For (a), we refer to \citep[Proposition 4.14]{Rouquier2008}, \citep[Theorem 3.1.1]{cruz2021cellular}, and \citep[Proposition 5.9]{appendix}.
	For (b), (c) and (d), we refer to 	\citep[Proposition 4.30]{Rouquier2008} and \citep[Proposition 5.7]{appendix}. For (e), we refer to \citep[Corollary 5.6]{appendix}.
\end{proof}

\subsubsection{Ringel duality}

Applying the methods of \cite{zbMATH00010169} to Artinian quasi-hereditary algebras, Ringel in \cite{MR1128706}, discovered that endomorphism algebras of characteristic tilting modules admit a quasi-hereditary structure. As seen in \citep[Section 7]{appendix}, split quasi-hereditary algebras over commutative Noetherian rings also come in pairs. The \textbf{Ringel dual of a split quasi-hereditary $R$-algebra} $(A, \{\Delta(\lambda)_{\lambda\in \Lambda}\})$, is, up to Morita equivalence, the endomorphism algebra $R(A):=\End_{A}(T)^{op}$ of a characteristic tilting module $T$ of $A$. 
The standard modules of $R(A)$ are $\St_{R(A)}(\l)=\Hom_A(T, \Cs(\l))$ with $\l\in \L^{op}$, where $\L^{op}$ is the opposite poset of $\L$. The \textbf{Ringel dual functor} $\Hom_A(T, -)\colon A\m\rightarrow R(A)\m$ restricts to an exact equivalence $\mathcal{F}(\Cssim)\rightarrow \mathcal{F}(\Stsim_{R(A)})$, it sends costandard modules to standard modules, (partial) tilting modules to projective modules and modules in $\add DA$ to tilting modules.

A split quasi-hereditary $R$-algebra is called \textbf{Ringel self-dual} if there exists an exact equivalence between $\mathcal{F}(\Stsim)$ and $\mathcal{F}(\Cssim)$, that is, if $A$ and $R(A)$ are Morita equivalent as split quasi-hereditary algebras. For split quasi-hereditary algebras over local commutative Noetherian rings, it is enough to test Ringel self-duality after applying extension of scalars from the local ground ring to its residue field.
 For more details, we refer to \cite{appendix}.

\subsubsection{The quasi-hereditary structure of $eAe$} It is also possible to construct new split quasi-hereditary algebras from bigger ones without using necessarily quotients by split heredity ideals. 

\begin{Theorem}\label{eAeqh}
	Let $(A, \{\Delta(\lambda)_{\lambda\in \Lambda}\})$ be a split quasi-hereditary algebra over a field $k$. Let $S(\l)$ be the top of the standard module $\St(\l)$ for $\l\in \L$. Then the following assertions hold.
	\begin{enumerate}[(i)]
		\item Let $e$ be an idempotent of $A$ and define ${\Lambda'=\{\lambda\in\Lambda\colon eS(\lambda)\neq 0 \}.}$ Then, $\{eS(\l)\colon \lambda\in \Lambda' \}$ is a full set of simple modules in $eAe\m$. 
		\item  Assume that there exists an idempotent $e$ satisfying the following 
		\begin{align}
			eS(\lambda)=0 \Longleftrightarrow \lambda\leq \mu \text{ for some } \mu\in \Gamma, \ \text{for some fixed subset }  \Gamma \subset \L. \label{eqfc5}
		\end{align}	Set ${\Lambda'=\{\lambda\in\Lambda\colon eS(\lambda)\neq 0 \}.}$  Then, the following holds true.
	\begin{enumerate}[(I)]
		\item 	$(eAe, \{e\St(\l)_{\l\in \L'} \})$ is a split quasi-hereditary algebra. The costandard modules of $eAe$ are of the form $\{e\Cs(\l)\colon \l\in \L'\}$. Moreover, $e\St(\l)=e\Cs(\l)=0$ for $\l\in \L\setminus \L'$.
		\item The (Schur) functor $\Hom_A(Ae, -)\colon A\m\rightarrow eAe\m$ preserves (partial) tilting modules. Moreover, the partial tilting indecomposable modules of $eAe$ are exactly $\{eT(\l)\colon \l\in  \L' \}$ and $eT(\l)=0$ for any $\l\in \L\setminus \L'$.
		\item Let $M\in \mathcal{F}(\St)$ and $N\in \mathcal{F}(\Cs)$. Then, the Schur functor $\Hom_A(Ae, -)\colon A\m\rightarrow eAe\m$ induces a surjective map $\Hom_A(M, N)\rightarrow \Hom_{eAe}(eM, eN)$.
	\end{enumerate}
	\end{enumerate}
\end{Theorem}
\begin{proof}
	Statement (i) actually holds for finite-dimensional algebras, in general. We refer to \citep[Theorem 6.2g]{zbMATH05080041}. For (I), see \citep[Proposition A3.11]{MR1707336}.  For (III), see for example \citep[1.7]{zbMATH00681964} or \citep[Lemma A3.12]{MR1707336}. For (II), see \citep[Lemma A4.5]{MR1707336}.
\end{proof}

Given a split quasi-hereditary algebra  $(A, \{\Delta(\lambda)_{\lambda\in \Lambda}\})$, we call an idempotent $e$ of $A$ \textbf{cosaturated} if it satisfies the property (\ref{eqfc5}).

\subsection{Covers}
A module $M\in A\m$ is said to have a \textbf{double centralizer property} if the canonical homomorphism of $R$-algebras $A\rightarrow \End_{\End_A(M)^{op}}(M)$ is an isomorphism.

The concept of cover was introduced in \cite{Rouquier2008} to evaluate the quality  of a module category with an approximation, in some sense, by a highest weight category. Let $P\in A\proj$. We say that the pair $(A, P)$ is a \textbf{cover} of $\End_A(P)^{op}$  if there is a double centralizer property on $\Hom_A(P, A)$ (as a right $A$-module).
 Equivalently, $(A, P)$ is a cover of $\End_A(P)^{op}$ if and only if, the \textbf{Schur functor},  ${F_P=\Hom_A(P, -)\colon A\m\rightarrow \End_A(P)^{op}\m}$ is fully faithful on $A\proj$.
Fix $B=\End_A(P)^{op}$. Since ${\Hom_A(P, -)\colon A\m\rightarrow B\m}$ is isomorphic to the functor ${\Hom_A(P, A)\otimes_A -\colon A\m\rightarrow B\m}$, it follows by Tensor-Hom adjunction that ${\Hom_{B}(\Hom_A(P, A), -)\colon B\m\rightarrow A\m}$ is right adjoint to the Schur functor ${\Hom_A(P, -)\colon A\m\rightarrow B\m}$. 
We denote by $\eta\colon \id_{A\m}\rightarrow G\circ F$ the unit of this adjunction. Observe that $\Hom_A(P, A)$ is a generator as a $B$-module and since $P\in A\proj$ the functor $\Hom_{B}(\Hom_A(P, A), -)\colon B\m\rightarrow A\m$ is fully faithful. 

If $(A, \{\Delta(\lambda)_{\lambda\in \Lambda}\})$ is a split quasi-hereditary algebra and $(A, P)$ is a cover of $B$, then we say that $(A, P)$ is a \textbf{split quasi-hereditary cover} of $B$.
Let $\mathcal{A}$ be a resolving subcategory of $A\m$ and let $i$ be a non-negative integer. Denote by $F_P$ the functor $\Hom_A(P, -)\colon A\m\rightarrow B\m$ and by $G_P$ its right adjoint.
We say that the pair $(A, P)$ is an \textbf{$i$-$\mathcal{A}$ cover} of $B$ if the Schur functor $F_P$ induces isomorphisms 
	\begin{align*}
	\Ext_A^j(M, N)\rightarrow \Ext_B^j(F_PM, F_PN), \quad \forall M, N\in \mathcal{A}, \ 0\leq j\leq i.
\end{align*} 	We say that $(A, P)$ is an \textbf{$(-1)$-$\mathcal{A}$ cover} of $B$ if $(A, P)$ is a cover of $B$ and the restriction of $F_P$ to $\mathcal{A}$ is faithful.
In Rouquier's terminology, when $(A, \{\Delta(\lambda)_{\lambda\in \Lambda}\})$ is a split quasi-hereditary algebra an $i$-$\mathcal{F}(\Stsim)$ cover is known as an $i$-faithful cover. 
The optimal value of $i$ making $(A, P)$ an $i$-$\mathcal{A}$ of $B$ is called the \textbf{Hemmer-Nakano dimension} of $\mathcal{A}$ (with respect to $P$). We refer to \cite{p2paper} for details.
We recall some useful facts about these objects.

\begin{Prop}\label{zeroAcover}
	The following assertions hold.
	\begin{enumerate}[(a)]
		\item $(A, P)$ is a $(-1)$-$\mathcal{A}$  cover of $B$ if and only if $(A, P)$ is a cover of $B$ and $\eta_M$ is a monomorphism for all $M\in \mathcal{A}$;
		\item $(A, P)$ is a $0$-$\mathcal{A}$ cover of $B$ if and only if  $\eta_M$ is an isomorphism for all $M\in \mathcal{A}$;
		\item $(A, P)$ is an $i$-$\mathcal{A}$ cover of $B$ for some $i\in \mathbb{N}$ if and only if $\R^jG_P(F_PM)=0$ and $\eta_M$ is an isomorphism   for all $M\in \mathcal{A}$ and all $1\leq j\leq i$.
	\end{enumerate}
\end{Prop}
\begin{proof}
	See \citep[Proposition 3.0.3, 3.0.4, Lemma 2.3.6.]{p2paper}.
\end{proof}
\begin{Remark}\label{casezerodeltacover}
	Since any module in $\mathcal{F}(\Stsim)$ admits a finite coresolution by summands of a characteristic tilting module, it is sufficient to check Proposition \ref{zeroAcover} (b) only for a characteristic tilting module to evaluate whether a split quasi-hereditary cover is a $0$-$\mathcal{F}(\Stsim)$ cover (see for example \citep[Proposition 4.40]{Rouquier2008}).
\end{Remark}

Among the deformation results studied in \citep{p2paper} we recall the following two.

\begin{Prop}\label{faithfulcoverresiduefield}
	Let $R$ be a regular (commutative Noetherian) ring.    Let $(A, \{\Delta(\lambda)_{\lambda\in \Lambda}\})$ be a split quasi-hereditary algebra, $P\in A\m\cap R\proj$ and $i\in \mathbb{N}\cup \{-1, 0\}$. If  $(A(\mi), P(\mi))$ is an $i$-$\mathcal{F}(\St(\mi))$ cover of $B(\mi)$ for every maximal ideal $\mi$ of $R$, then $(A, P)$ is an $i$-$\mathcal{F}(\Stsim)$ cover of $B$.
\end{Prop}
\begin{proof}
	See \citep[Propositions 5.0.5, 5.0.6]{p2paper}.
\end{proof}

\begin{Theorem}\label{improvingcoverwithspectrum}
	Let $R$ be a local commutative regular Noetherian ring with quotient field $K$. Let $(A, \{\Delta(\lambda)_{\lambda\in \Lambda}\})$ be a split quasi-hereditary algebra and let $(A, P)$ be a $0$-$\mathcal{F}(\Stsim)$ cover of $B$. Let $i\geq 0$.
	Assume that the following conditions hold:
	\begin{enumerate}[(i)]
	\item $(K\otimes_R A, K\otimes_R P)$ is an $(i+1)$-$\mathcal{F}(K\otimes_R \St)$ cover of $K\otimes_R B$;
	\item For each prime ideal $\mathfrak{p}$ of height one, $(R/\mathfrak{p}\otimes_R A, R/\mathfrak{p}\otimes_R P)$ is an $i$-$\mathcal{F}(R/\pri\otimes_R \Stsim)$ cover of $R/\mathfrak{p}\otimes_R B$.
	\end{enumerate}
	Then, $(A, P)$ is an $i+1$-$\mathcal{F}(\Stsim)$ cover of $B$.
\end{Theorem}
\begin{proof}
	See \citep[Theorem 5.1.1]{p2paper}.
\end{proof}

\subsubsection{Uniqueness of covers}

We recall the concept of equivalent covers discussed in \citep[Subsection 4.3]{p2paper} generalising \citep[4.2.3]{Rouquier2008}.

	Let $A, A', B, B'$ be projective Noetherian $R$-algebras and $\mathcal{A}$ and $\mathcal{A}'$ be resolving subcategories of $A\m\cap R\proj$ and $A'\m\cap R\proj$, respectively.
	Assume that $(A, P)$ is a $0$-$\mathcal{A}$ cover of $B$ and $(A', P')$ is a $0$-$\mathcal{A}'$ cover of $B'$. We say that the $\mathcal{A}$-cover $(A, P)$ is \textbf{equivalent} to the $\mathcal{A}'$-cover $(A', P')$   if there is an equivalence of categories $H\colon A\m\rightarrow A'\m$, which restricts to an exact equivalence $\mathcal{A}\rightarrow \mathcal{A}'$, and an equivalence of categories $L\colon B\m\rightarrow B'\m$ satisfying $L\circ \Hom_A(P, -)=\Hom_{A'}(P', -)\circ H$.

In \citep[Section 4]{p2paper}, the author proved for example that the number of simple $B$-modules is an upper bound to the quality of a split quasi-hereditary cover, if the Schur functor associated to the cover is not fully faithful. 
The importance of $1$-$\mathcal{F}(\Stsim)$ covers stems from the following characterization.

\begin{Prop}
	Let $(A, \{\Delta(\lambda)_{\lambda\in \Lambda}\})$ be a split quasi-hereditary algebra and let $(A, P)$ be a $0$-$\mathcal{F}(\Stsim)$ cover of $B$. \label{onefaithfulcovers} Then, the following assertions are equivalent.
	\begin{enumerate}[(a)]
		\item $(A, P)$ is a $1$-$\mathcal{F}(\Stsim)$ cover of $B$;
		\item The functor $F_P$ restricts to an exact equivalence of categories $\mathcal{F}(\Stsim)\rightarrow\mathcal{F}(F_P\Stsim)$ with inverse the exact functor $\Hom_B(F_PA, -)_{|_{\mathcal{F}(F_P\Stsim)}}\colon \mathcal{F}(F_P\Stsim)\rightarrow \mathcal{F}(\Stsim)$.
	\end{enumerate}
\end{Prop}
\begin{proof}
	See for example \citep[Proposition 4.41]{Rouquier2008}.
\end{proof}
The following result based on \citep[Corollary 4.46]{Rouquier2008} states that split quasi-hereditary covers with large enough quality are essentially unique, in some sense.

\begin{Cor}\label{equivalenceofoneuniqueness}
	Let $(A, \{\Delta(\lambda)_{\lambda\in \Lambda}\})$, $(A', \{\Delta'(\lambda)_{\lambda\in \Lambda'}\})$ be two split quasi-hereditary algebras and let $(A, P)$, $(A', P')$ be $1$-$\mathcal{F}(\Stsim)$, $1$-$\mathcal{F}(\Stsim')$ covers of $B$, respectively.
	Assume that there exists an exact equivalence $L\colon B\m\rightarrow B\m$ which restricts to an exact equivalence $
		\mathcal{F}(F_P\Stsim)\rightarrow  \mathcal{F}(F'_{P'}\Stsim'). \label{fceq148}
$ Then, the $\mathcal{A}$-cover $(A, P)$ is equivalent to the $\mathcal{A}'$-cover $(A', P')$ of $B$.
\end{Cor}
\begin{proof}
We refer to \cite[Corollary 4.3.6]{p2paper}.
\end{proof}

\subsection{Relative dominant dimension over Noetherian algebras}

In \cite{zbMATH05871076}, it was established that dominant dimension of certain modules over finite-dimensional algebras control the quality of  many $A\proj$-covers and $\mathcal{F}(\St)$-covers of a finite-dimensional algebra $B$. These covers can could be controlled in such a way are formed by endomorphism algebras of generator-cogenerators over $B$.

In \cite{CRUZ2022410}, the author generalised the concept of dominant dimension to projective Noetherian algebras. In \citep{p2paper}, the author showed that this invariant does not fully control the quality of a cover in the integral setup, since there are more factors that influence the cover. But, this invariant can be exploited together with change of rings to compute the quality of a cover which is formed by an integral version of the endomorphism algebra of a generator-cogenerator. Such integral versions first appear in \citep[Theorem 4.1]{CRUZ2022410} which we call relative Morita-Tachikawa correspondence.

	\begin{Def}\label{relativedominantdef}
	Let $M\in A\m$. We say that $M$ has \textbf{relative dominant dimension at least $t\in \mathbb{N}$}  if there exists an $(A, R)$-exact sequence of 
	finitely generated left $A$-modules
	$
		0\rightarrow M\rightarrow I_1\rightarrow\cdots \rightarrow I_t$
	 where all $I_i$ are projective and $(A, R)$-injective left $A$-modules. If $M$ admits no such $(A, R)$-exact sequence, then we say that $M$ has relative dominant dimension zero. Otherwise, the relative dominant dimension of $M$ is $\sup\{ t\in \mathbb{N}\colon 0\rightarrow M\rightarrow I_1\rightarrow \cdots \rightarrow I_t \text{ is $(A, R)$-exact}, \ I_1, \ldots, I_t\in A\proj\cap \add DA \}\subset \mathbb{N}\cup \{0, +\infty\}$.
	We denote by $\domdim_{(A, R)} M$ the relative dominant dimension of $M$. We just write $\domdim_A M$ when $R$ is a field.
\end{Def} We denote by $\domdim{(A, R)}$ the relative dominant dimension of the regular module $A$. 
This concept behaves well under extension of scalars. In particular, if $\domdim_{(A, R)} A\geq 1$, then $\domdim_{(A, R)} M=\inf\{\domdim_{A(\mi)}M(\mi)\colon \mi\in \MaxSpec R \}$ for any $M\in A\m\cap R\proj$ (see \citep[Theorem 6.13]{CRUZ2022410}).

Given $M\in A\m$ we say that $M$ is an \textbf{$(A, R)$-injective-strongly faithful  module} if $M$ is $(A, R)$-injective and there exists an $(A, R)$-monomorphism $A\hookrightarrow X$ for some $X\in \add M$.
By a \textbf{RQF3-algebra} we mean a triple $(A, P, V)$ formed by a projective Noetherian $R$-algebra $A$, a projective $(A, R)$-injective-strongly faithful left $A$-module $P$ and a projective $(A, R)$-injective-strongly faithful right $A$-module $V$.

We recall the following properties relating relative dominant dimension with cover theory.

\begin{Prop}
	Let $(A, P, V)$ be a RQF3-algebra over a commutative Noetherian ring $R$. \label{domdimtoolcover}
	The following assertions hold.
	\begin{enumerate}[(a)]
		\item If $\domdim{(A, R)} \geq 2$, then $(A, \Hom_{A^{op}}(V, A))$ is a cover of $\End_A(V)$.
		\item Let $(A, \{\Delta(\lambda)_{\lambda\in \Lambda}\})$ be a split quasi-hereditary algebra. Assume that $T$ is a characteristic tilting module of $A$. Then, $(A, \Hom_{A^{op}}(V, A))$ is an $(\domdim_{(A, R)} T-2)$-$\mathcal{F}(\Stsim)$ cover of $\End_A(V)$. Moreover, the Hemmer--Nakano dimension of $\mathcal{F}(\Stsim)$ (with respect to $\Hom_{A^{op}}(V, A)$) is less than or equal to $\domdim_{(A, R)} T+\dim R -2$.
	\end{enumerate}
\end{Prop}
\begin{proof}
	For (a), see \citep[Proposition 2.4.4.]{p2paper}. For (b), see \citep[Theorem 6.0.1, Theorem 6.2.1]{p2paper}.
\end{proof}

A triple $(A, P, DP)$ is said to be a \textbf{relative Morita $R$-algebra} if $(A, P, DP)$  is an RQF3-algebra so that $\domdim {(A, R)} \geq 2$ and $\add DP\otimes_A DA=\add DP$. In particular,  $$\add \Hom_{A^{op}}(DP, A)={\add D(DP\otimes_ADA)}=\add DDP=\add P.$$

By a \textbf{relative gendo-symmetric $R$-algebra} we mean a pair  $(A, P)$ so that $P$ is a left projective $(A, R)$-injective-strongly faithful module satisfying $P\simeq DA \otimes_A P$ as $(A, \End_A(P)^{op})$-bimodules and $\domdim{(A, R)} \geq 2$. In both cases, when $R$ is a field, we just say that $A$ is a Morita (resp. gendo-symmetric) algebra.

\section{Relative (co-)dominant dimension with respect to a module}\label{Relative (co-)dominant dimension with respect to a module}

In this part, we study a new generalisation of relative dominant dimension over projective Noetherian algebras and how it can be used to obtain information about the functor $F_Q$, for some $Q\in A\m\cap R\proj$, where we do not assume $Q$ to be necessarily projective.

	\begin{Def}\label{relativedomdimrelativedef}
	Let $T, X\in A\m\cap R\proj$. 
	If $X$ does not admit a left $\add T$-approximation which is an $(A, R)$-monomorphism, then we say that 	\textbf{relative dominant dimension of $X$ with respect to $T$} is zero.
	Otherwise, the
	\textbf{relative dominant dimension of $X$ with respect to $T$}, denoted by $T\rdomdim_{(A,
		R)} X$, is the supremum of  all $n\in \mathbb{N}\subset \mathbb{N}\cup \{0, +\infty\}$ such that there exists an $(A,
	R)$-exact sequence $
		0\rightarrow X\rightarrow T_1\rightarrow \cdots \rightarrow T_n
$ which remains exact under $\Hom_A(-, T)$ with all $T_i\in \add T$. 
\end{Def}

By convention, the empty direct sum is the zero module. So, the existence of a  finite relative $\add T$-coresolutions  implies that $T\rdomdim_{(A, R)} X$ is infinite.
In the same way, we can define the relative dominant dimension of a right module with respect to a right module $Q$. We write $Q\rdomdim{(A, R)}$ instead of $Q\rdomdim_{(A, R)} A$. We just write $T\rdomdim_A X$ if $R$ is a field.

Definition \ref{relativedomdimrelativedef} generalizes the concept of relative dominant dimension introduced in \citep{CRUZ2022410}.
More precisely, we have the following.

\begin{Prop}\label{relativedomdimrelativecoincd}
	Let $R$ be a commutative Noetherian ring. Let $A$ be a projective Noetherian $R$-algebra with
	$\domdim{(A, R)} \geq 1$ with projective $(A,
	R)$-injective-strongly faithful left $A$-module $P$. Then, \begin{align}
		P\rdomdim_{(A, R)} X=\domdim_{(A, R)} X, \quad X\in A\m.
	\end{align}
\end{Prop}
\begin{proof}
See \citep[Proposition 3.15]{CRUZ2022410}.
\end{proof}

	In order to avoid changing from left to right modules systematically in the coming sections, we can introduce the relative codominant dimension with respect to a module.

\begin{Def}
	 Let $Q, X\in A\m\cap R\proj$. 
	If $X$ does not admit a surjective right $\add Q$-approximation, then we say that 	\textbf{relative codominant dimension of $X$ with respect to $Q$} is zero.
	Otherwise, the
	\textbf{relative codominant dimension of $X$ with respect to $Q$}, denoted by $Q\rcodomdim_{(A,
		R)} X$, is the supremum of all $n\in \mathbb{N}\subset \mathbb{N}\cup \{0, +\infty\}$ such that there exists an $(A,
	R)$-exact sequence $
		Q_n\rightarrow\cdots\rightarrow Q_1 \rightarrow X\rightarrow 0
$ which remains exact under $\Hom_A(Q, -)$ with all $Q_i\in \add Q$. If some $Q_i=0$, then we say that
	\mbox{$Q\rcodomdim_{(A, R)} X$} is infinite. 
\end{Def} 
In particular, $DQ\rdomdim_{(A, R)}DM = Q\rcodomdim_{(A, R)} M$ whenever $Q, M\in A\m\cap R\proj$.
As we will see later, these two invariants coincide in our main cases of interest.

\subsection{Relative Mueller's characterisation of relative dominant dimension with respect to a module}\label{Relative Mueller's characterization of relative dominant dimension with respect to a module}

In this section, we study a version of Mueller's theorem for the relative dominant dimension with respect to a module. The arguments for these results are inspired by \citep[Proposition 2.1]{zbMATH00423524}.

	\begin{Theorem}\label{moduleMuellerpartone}
	Let $R$ be a commutative Noetherian ring. Let $A$ be a projective Noetherian $R$-algebra. Assume that $Q\in A\m\cap R\proj$ satisfying in addition that $\Hom_A(Q, Q)\in R\proj$. Denote by $B$ the endomorphism algebra $ \End_A(Q)^{op}$. For $M\in A\m\cap R\proj$, the following assertions hold.
	\begin{enumerate}[(i)]
		\item The counit $\chi_M\colon Q\otimes_B \Hom_A(Q, M)\rightarrow M$ is surjective if and only if $DQ\rdomdim_{(A, R)}DM\geq 1$.
		\item $\chi_M\colon Q\otimes_B \Hom_A(Q, M)\rightarrow M$ is an isomorphism if and only if \mbox{$DQ\rdomdim_{(A, R)}DM\geq 2$.}
	\end{enumerate}
\end{Theorem}
\begin{proof}
	Assume that $\chi_M$ is surjective. Since $\Hom_A(Q, M)\in B\m $ there exists $X\in \add Q$ and a surjective map $\Hom_A(Q, X)\twoheadrightarrow \Hom_A(Q, M)$, say $g$. The functor $Q\otimes_B -$ is right exact, so $Q\otimes_B g$ is surjective as well. Define $f:=\chi_M\circ Q\otimes_B g \circ \chi^{-1}_X$. The map $f$ is surjective and we claim that $\Hom_A(Q, f)=g$. To see that, observe first that $q\otimes h=\chi_X^{-1}\chi_X(q\otimes h)=\chi_X^{-1}(h(q))$ for every $q\in Q$ and $h\in \Hom_A(Q, X)$. Now, we can see that for every $h\in \Hom_A(Q, X)$, $q\in Q$,
	\begin{align}
		\Hom_A(Q, f)(h)(q)=f\circ h (q)=\chi_M \circ Q\otimes_B g (q\otimes h)=g(h)(q).
	\end{align} Applying $D$ yields the $(A, R)$-monomorphism $DM\rightarrow DX$ which remains exact under $\Hom_A(-, DQ)$. Thus,  \mbox{$DQ\rdomdim_{(A, R)} DM\geq 1$}. Conversely, assume that $DQ\rdomdim_{(A, R)} DM\geq 1$. So, there exists $X\in \add DQ$ and an $(A, R)$-monomorphism $f\colon DM\rightarrow X$ which is also a left $\add DQ$-approximation. Since $\chi$ is a natural transformation between $Q\otimes_B \Hom_A(Q, -)$ between $\id_{B\m}$ the map $Df\circ \chi_{DX}={\chi_{DDM}\circ Q\otimes_B \Hom_A(Q, Df)}$ is surjective. In particular, $\chi_{DDM}$ is surjective. As $DDM\simeq M$, $\chi_M$ is surjective and (i) follows.
	
	Now, assume that $DQ\rdomdim_{(A, R)} DM\geq 2$. Then, there exists an $(A, R)$-exact sequence $0\rightarrow DM\xrightarrow{f_0} X_0\xrightarrow{f_1} X_1$, with $X_0, X_1\in \add DQ$, which remains exact under $\Hom_A(-, DQ)$. As $Q\otimes_B -$ is right exact, the following diagram is commutative with exact rows
	\begin{equation}
		\begin{tikzcd}
			Q\otimes_B \Hom_A(Q, DX_1) \arrow[r, "Q\otimes_B\Hom_A(Q{,} Df_1)", outer sep=0.75ex] \arrow[d, "\chi_{DX_1}", "\simeq"'] &[2em] Q\otimes_B \Hom_A(Q, DX_0)\arrow[r, twoheadrightarrow, "Q\otimes_B\Hom_A(Q{,} Df_0)", outer sep=0.75ex ] \arrow[d, "\chi_{DX_0}", "\simeq"']  &[2em] Q\otimes_B \Hom_A(Q, DDM) \arrow[d, "\chi_{DDM}"]\\
			DX_1 \arrow[r, "Df_1"] & DX_0 \arrow[r, "Df_0", twoheadrightarrow] & DDM 
		\end{tikzcd}.
	\end{equation} By diagram chasing, $\chi_{DDM}$ is an isomorphism. Since $DDM\simeq M$, $\chi_M$ is an isomorphism. Conversely, assume that $\chi_M$ is an isomorphism. $B$ is a Noetherian $R$-algebra, so we can consider a projective $B$-presentation for $\Hom_A(Q, M)$ of the form
	\begin{align}
		\Hom_A(Q, Q^m)\xrightarrow{g_1} \Hom_A(Q, Q^n)\xrightarrow{g_0} \Hom_A(Q, M)\rightarrow 0,
	\end{align} for some integers $m, n$. Since $\Hom_A(Q, -)_{|_{\add T}}$ is full and faithful there exists $f_1\in \Hom_A(Q^m, Q^n)$ such that $\Hom_A(Q, f_1)=g_1$. Fix $f_0=\chi_M\circ Q\otimes_B g_0 \circ \chi_{Q^n}^{-1}$. We have seen previously, that $\Hom_A(Q, f_0)=g_0$. 
	So, the diagram 
	\begin{equation}
		\begin{tikzcd}
			Q\otimes_B \Hom_A(Q, Q^m) \arrow[r, "Q\otimes_B g_1", outer sep=0.75ex, swap] \arrow[d, "\chi_{Q^m}", "\simeq"']&[1.3em] Q\otimes_B \Hom_A(Q, Q^n) \arrow[r, "Q\otimes_B g_0", twoheadrightarrow, outer sep=0.75ex, swap] \arrow[d, "\chi_{Q^n}", "\simeq"'] &[1.3em] Q\otimes_B \Hom_A(Q, M) \arrow[d, "\chi_M"]\\
			Q^m\arrow[r, "f_1"] & Q^n \arrow[r, "f_0", twoheadrightarrow] & M
		\end{tikzcd}
	\end{equation}
	is commutative. Since the vertical maps are isomorphisms and the upper row is exact it follows that the bottom row is exact and by construction it remains exact under $\Hom_A(Q, -)$. As $M\in R\proj$, it is, also, $(A, R)$-exact. By applying the standard duality $D$ we obtain that $DQ\rdomdim_{(A, R)} DM\geq 2$.
\end{proof}

	\begin{Remark}\label{counitchianddelta}
	Note that, for each $M\in A\m\cap R\proj$, the map $\chi_M$ is equivalent to the map $\delta_{DM}$ studied in  \citep[Lemma 3.21]{CRUZ2022410} when $Q=P$ is a projective $(A, R)$-injective-strongly faithful module.
\end{Remark}

	Similarly, we can write the dual version of Theorem \ref{moduleMuellerpartone}.

\begin{Theorem}\label{dualmoduleMuellerpartone}
	Let $R$ be a commutative Noetherian ring. Let $A$ be projective Noetherian $R$-algebra. Assume that $Q\in A\m\cap R\proj$ satisfying in addition that $\Hom_A(Q, Q)\in R\proj$. Denote by $B$ the endomorphism algebra $ \End_A(Q)^{op}$. For $M\in A\m\cap R\proj$, the following assertions hold.\begin{enumerate}[(i)]
		\item $\chi^r_{DM}\colon \Hom_A(DQ, DM)\otimes_B DQ\rightarrow DM$ is surjective if and only if $Q\rdomdim_{(A, R)} M\geq 1$.
		\item  $\chi^r_{DM}\colon \Hom_A(DQ, DM)\otimes_B DQ\rightarrow DM$ is an isomorphism if and only if $Q\rdomdim_{(A, R)} M\geq 2$.
	\end{enumerate}
\end{Theorem}

In fact, Theorem \ref{moduleMuellerpartone} characterizes the lower cases of relative codominant dimension with respect to a module.

	Now, the second part of the Mueller version for relative (co-)dominant dimension with respect to a module is as follows:

\begin{Theorem}\label{moduleMuellerparttwo}
	Let $R$ be a commutative Noetherian ring. Let $A$ be a projective Noetherian $R$-algebra. Assume that $Q\in A\m\cap R\proj$ satisfying in addition that $\Hom_A(Q, Q)\in R\proj$. Denote by $B$ the endomorphism algebra $ \End_A(Q)^{op}$. For $M\in A\m\cap R\proj$, the following assertions hold.
	\begin{enumerate}[(i)]
		\item $Q\rcodomdim_{(A, R)} M\geq n\geq 2$ if and only if $\chi_M\colon Q\otimes_B \Hom_A(Q, M)\rightarrow M$ is an isomorphism of left $A$-modules and $\Tor_i^B(Q, \Hom_A(Q, M))=0$, $1\leq i\leq n-2$.
		\item $Q\rdomdim_{(A, R)} M\geq n\geq 2$ if and only if $\chi^r_{DM}\colon \Hom_A(DQ, DM)\otimes_B DQ\rightarrow DM$ is an isomorphism and $\Tor_i^B(\Hom_A(DQ, DM), DQ)=\Tor_i^B(\Hom_A(M, Q), DQ)=0$, $1\leq i\leq n-2$.
	\end{enumerate}
\end{Theorem}
\begin{proof}
	We shall prove (i). The statement (ii) is analogous to (i).
	Assume that $DQ\rdomdim_{(A, R)} DM\geq n\geq 2$. By Theorem \ref{moduleMuellerpartone}, $\chi_M$ is an isomorphism.  By definition, there exists an $(A, R)$-exact sequence
	\begin{align}
		0\rightarrow DM\rightarrow X_0\rightarrow X_1\rightarrow \cdots \rightarrow X_{n-1}\label{eqcocovers24}
	\end{align}which remains exact under $\Hom_A(-, DQ)$ with $X_i\in \add DQ$, $i=0, \ldots, n-1$. In particular, \linebreak${\Hom_A(X_{n-1}, DQ)\rightarrow \cdots\rightarrow \Hom_A(X_0, DQ)\rightarrow \Hom_A(DM, DQ)\rightarrow 0}$ is exact and can be continued to a left projective $B$-resolution of $\Hom_A(Q, M)$. Consider the following commutative diagram
	\begin{equation}
		\begin{tikzcd}
			Q\otimes_B \Hom_A(Q, DX_{n-1}) \arrow[r] \arrow[d, "\chi_{DX_{n-1}}", "\simeq"']& \cdots \arrow[r] & Q\otimes_B \Hom_A(Q, DX_0) \arrow[r, twoheadrightarrow] \arrow[d, "\chi_{DX_{0}}", "\simeq"']& Q\otimes_B \Hom_A(Q, M) \arrow[d, "\chi_{DDM}", "\simeq"']\\
			DX_{n-1}\arrow[r] & \cdots \arrow[r] & DX_0\arrow[r, twoheadrightarrow] & DDM
		\end{tikzcd}.
	\end{equation} Observe that the bottom row is exact since the exact sequence (\ref{eqcocovers24}) is $(A, R)$-exact. Since all vertical maps are isomorphisms, it follows that the upper row is exact. Thus, $\Tor_i^B(Q, \Hom_A(Q, M))=0$, $1\leq i\leq n-2$.
	
	Conversely, assume that $\chi_M$ is an isomorphism and $\Tor_i^B(Q, \Hom_A(Q, M))=0$ for $1\leq i\leq n-2$. Let $\Hom_A(Q, X_{n-1}) \xrightarrow{g_{n-1}} \cdots \rightarrow \Hom_A(Q, X_0)\xrightarrow{g_0} \Hom_A(Q, M)\rightarrow 0$ be a truncated projective $B$-resolution of  $\Hom_A(Q, M)$ and $X_i\in \add_A Q$. Furthermore, $\Hom_A(Q, -)_{|_{\add Q}}$ is full and faithful, so each map $g_i$ can be written as $\Hom_A(Q, f_i)$ including $g_0$ since $\chi_M$ is an isomorphism, where $f_i\in \Hom_A(X_i, X_{i-1})$ and $f_0\in \Hom_A(X_0, M).$ So, we have a commutative diagram
	\begin{equation}
		\begin{tikzcd}
			Q\otimes_B \Hom_A(Q, X_{n-1})\arrow[d, "\chi_{X_{n-1}}", "\simeq"', swap]\arrow[r, "Q\otimes_B \Hom_A(Q {,} f_{n-1})", outer sep=0.75ex, swap] & \cdots \arrow[r] & Q\otimes_B \Hom_A(Q, X_0) \arrow[r, "Q\otimes g_0", twoheadrightarrow,  outer sep=0.75ex, swap] \arrow[d, "\chi_{X_{0}}", "\simeq"']& Q\otimes_B \Hom_A(Q, M) \arrow[d, "\chi_{M}", "\simeq"'] \\
			X_{n-1}\arrow[r, "f_{n-1}"] & \cdots \arrow[r] & X_0 \arrow["f_0", r, twoheadrightarrow]& M 
		\end{tikzcd}.
	\end{equation}By assumption, $\Tor_i^B(Q, \Hom_A(Q, M))=0$, $1\leq i\leq n-2$. So, the upper row is exact. By the exactness and the vertical maps being isomorphisms the bottom row becomes exact. Since $M\in R\proj$ it is also $(A, R)$-exact and so it remains $(A, R)$-exact under $D$. By construction, the bottom row remains exact under $\Hom_A(Q, -)$, thus $DQ\rdomdim_{(A, R)} DM\geq n\geq 2$.
\end{proof}

\begin{Remark}
	The condition $\Hom_A(Q, Q)\in R\proj$ is considered to enforce that $B$ is a projective Noetherian $R$-algebra same as $A$. Having only the criteria in Theorems \ref{moduleMuellerparttwo} and \ref{moduleMuellerpartone}, we can observe that the arguments carry over if $B$ is only a Noetherian $R$-algebra.
\end{Remark}

\subsubsection{Identities on relative dominant and codominant dimension}

	An immediate consequence of Theorems \ref{moduleMuellerparttwo} and \ref{moduleMuellerpartone} is the following.

\begin{Cor}\label{rightandleftrelativedomidimension}
	Let $R$ be a commutative Noetherian ring. Let $A$ be a projective Noetherian $R$-algebra. Assume that $Q\in A\m\cap R\proj$ satisfying in addition that $\Hom_A(Q, Q)\in R\proj$. Then,
	$$DQ\rdomdim{(A, R)} = Q\rcodomdim_{(A, R)} DA= Q\rdomdim {(A,R)}.$$
\end{Cor}
\begin{proof}
	By Tensor-Hom adjunction, there are isomorphisms $\Hom_A(Q, DA)\simeq DQ$, given by $f\mapsto (f(-)(1_A))$, and $Q\simeq \Hom_A(DQ, DA)$, given by $q\mapsto (f\mapsto (a\mapsto f(aq)))$. We shall denote the first by $\psi$ and the second isomorphism by $\omega$. So, $\psi$ is a left $B$-isomorphism while $\omega$ is a right $B$-isomorphism. Moreover, $\chi^r_{DA}\circ\omega\otimes_B \psi =\chi_{DA}$. In fact, for $a\in A$, $q\in Q$, $g\in \Hom_A(Q, DA)$,
	\begin{align*}
		\chi_{DA}^r\circ \omega\otimes_B \psi (q\otimes g)(a)=\omega(q)(\psi(g))(a)=\psi(g)(aq)&=g(aq)(1_A)=(ag(q))(1_A)\\&=g(q)(a)=\chi_{DA}(q\otimes g)(a).
	\end{align*}
	By Theorems \ref{moduleMuellerpartone} and \ref{dualmoduleMuellerpartone}, $DQ\rdomdim{(A, R)} \geq i$ if and only if $Q\rdomdim{(A, R)} \geq i$ for $i=1, 2$. Finally, by Theorem \ref{moduleMuellerparttwo}, $DQ\rdomdim{(A, R)}=DQ\rdomdim_{(A, R)} DDA_A\geq n\geq 2$ if and only if $\chi_{DA}$ is an isomorphism and $0=\Tor_i^B(Q, \Hom_A(Q, DA))=\Tor_i^B(Q, DQ)=\Tor_i^B(\Hom_A(DQ, DA),DQ)$, $1\leq i\leq n-2$ if and only if $Q\rdomdim{(A, R)} \geq n\geq 2$.
\end{proof}

 It follows, by Theorem \ref{moduleMuellerparttwo} and Corollary \ref{rightandleftrelativedomidimension}, that the relative dominant dimension of the regular module with respect to a module $Q$ over a finite-dimensional algebra coincides with the \textbf{faithful dimension of $Q$} introduced in \cite{zbMATH01218841}. 

There is a version of Corollary \ref{rightandleftrelativedomidimension} for (partial) tilting modules, if the split quasi-hereditary algebra admits a duality functor on $A\m\cap R\proj$ interchanging $\St(\l)$ with $\Cs(\l)$  (or a simple preserving duality if the ground ring is a field).  Here, by a \textbf{duality functor} on $A\m\cap R\proj$ we mean an exact, involutive and contravariant autoequivalence of categories $A\m\cap R\proj\rightarrow A\m\cap R\proj$.
 
 \begin{Prop}\label{dominantandcodominantoftiltingcomparison}
 	Let $(A, \{\Delta(\lambda)_{\lambda\in \Lambda}\})$ be a split quasi-hereditary $R$-algebra with a characteristic tilting module $T$. Let $V\in \add_A T$ and assume that ${}^\natural (-)\colon A\m\cap R\proj\rightarrow A\m\cap R\proj$ is a duality satisfying ${}^\natural\St(\l)=\Cs(\l)$ and ${}^\natural V\simeq V$ for all $\l\in \L$. Then,
 	$V\rdomdim_{(A, R)} T=V\rcodomdim_{(A, R)} T$.
 \end{Prop}
 \begin{proof}
 	Assume that $V\rdomdim_{(A, R)} T\geq n\geq 1$. By definition, there exists an $(A, R)$-exact sequence 
 	\begin{align}
 		0\rightarrow T\xrightarrow{\alpha_0} V_0\xrightarrow{\alpha_1} V_1\rightarrow \cdots \rightarrow V_{n-1}\label{eqcocover27}
 	\end{align} which remains exact under $\Hom_A(-, V)$, with $V_i\in \add_A V$. Applying the duality ${}^\natural$ we obtain the exact sequence
 	\begin{align}
 		\delta:{}^\natural V_{n-1} \xrightarrow{{}^\natural\alpha_{n-1}} \cdots \rightarrow {}^\natural V_1 \xrightarrow{{}^\natural\alpha_1} {}^\natural V_0 \xrightarrow{{}^\natural\alpha_0} {}^\natural T\rightarrow 0.\label{eqcocover28}
 	\end{align}Applying ${}^\natural(-)$ to the exact sequences defining $T(\l)$ we obtain that ${}^\natural T$ is also a characteristic tilting module and so $\add T=\add {}^\natural T$. In particular, (\ref{eqcocover28}) is $(A, R)$-exact. It remains to show that (\ref{eqcocover28}) remains exact under $\Hom_A(V, -)\simeq \Hom_A({}^\natural V, -)$. To show that consider, for each $i$, the factorization of $\alpha_{i+1}$ through its image $\alpha_{i+1}=\nu_i\circ \pi_i$. Hence, $\alpha_0$ and $\nu_i$, $i=0, \ldots, n-1$, are left $\add V$-approximations. By the exactness of the contravariant functor ${}^\natural$, ${}^\natural(\ker \alpha_{i+1})\simeq \coker({}^\natural\alpha_{i+1})$ and ${}^\natural \im \alpha_{i+1}= \im ({}^\natural \alpha_{i+1})$ for all $i$. Moreover, for every homomorphism $f\in \Hom_A(N, L)$ the maps $\Hom_A({}^\natural V, {}^\natural f)$ and $\Hom_A(f, V)$ are related by the commutative diagram
 	\begin{equation}
 		\begin{tikzcd}
 			\Hom_A({}^\natural V, {}^\natural L) \arrow[r, "\Hom_A({}^\natural V {,} {}^\natural f )", outer sep=0.75ex, swap] & \Hom_A({}^\natural V, {}^\natural N)\\
 			\Hom_A(L, V)\arrow[u, "\simeq"] \arrow[r, "\Hom_A(f{,} V)"] & \Hom_A(N, V)\arrow[u, "\simeq"]
 		\end{tikzcd}.
 	\end{equation}Hence, for each $i$, ${}^\natural\nu_i$ is an surjective right $\add V$-approximation and ${}^\natural \alpha_{i+1} ={}^\natural \pi_i\circ {}^\natural \nu_i$. The same is true for ${}^\natural \alpha_0$. By Lemma \ref{gofromrighttoleftappro}, the complex $\Hom_A(V, \delta)$ is exact. Hence, $V\rcodomdim_{(A, R)} T\geq n\geq 1$. Conversely, $V\rcodomdim_{(A, R)} T=DV\rdomdim_{(A, R)} DT\geq DV\rcodomdim_{(A, R)} DT = V\rdomdim_{(A, R)} T$.
 \end{proof}

	\subsubsection{Behaviour of relative dominant dimension on long exact sequences}
Using Theorem \ref{moduleMuellerparttwo} is now clear to understand how the relative dominant dimension with respect to a module behaves in short exact sequences.

\begin{Lemma}\label{relativedominantdimensiononses}
	Let $R$ be a commutative Noetherian ring. Let $A$ be a projective Noetherian $R$-algebra. Assume that $Q\in A\m\cap R\proj$ satisfying in addition that $\Hom_A(Q, Q)\in R\proj$. Let $M\in R\proj$ and consider the following
	$(A, R)$-exact \begin{align}
		0\rightarrow M_1\rightarrow M\rightarrow M_2\rightarrow 0
	\end{align}which remains exact under $\Hom_A(-, Q)$. Let $n=Q\rdomdim_{(A, R)} M$ and $n_i=Q\rdomdim_{(A, R)} M_i$. Then, the
	following holds.
 \begin{multicols}{2}
 		\begin{enumerate}[(a)]
 		\itemsep0em 
 		\item $n\geq \min\{n_1, n_2\}$.
 		\item If $n_1<n$, then $n_2=n_1-1$.
 		\item \begin{enumerate}[(i)]
 			\item $n_1=n\implies n_2\geq n-1$.
 			\item $n_1=n+1\implies n_2\geq n$.
 			\item $n_1\geq n+2\implies n_2=n$.
 		\end{enumerate}
 		\item $n<n_2\implies n_1=n$.
 		\vfill
 		\item \begin{enumerate}[(i)] 
 			\item $n=n_2\implies n_1\geq n_2$.
 			\item $n=n_2+1\implies n_1\geq n_2+1$.
 			\item $n\geq n_2+2\implies n_1=n_2+1$.
 		\end{enumerate}
 	\end{enumerate}
 \end{multicols}
\end{Lemma}
\begin{proof}
	By assumption, $0\rightarrow \Hom_A(DQ, DM_2)\rightarrow \Hom_A(DQ, DM)\rightarrow \Hom_A(DQ, DM_1)\rightarrow 0$ is exact. The remaining of the proof is exactly analogous to Lemma 5.12 of \citep{CRUZ2022410}.
\end{proof}

\begin{Cor}
	Assume that $Q\in A\m\cap R\proj$ satisfying in addition that $\Hom_A(Q, Q)\in R\proj$. Let $M_i\in A\m\cap R\proj$, $i\in I$, for some finite set $I$. Then,
	\begin{align*}
		Q\rdomdim_{(A, R)} \left( \bigoplus_{i\in I} M_i \right) = \inf \{Q\rdomdim_{(A, R)}M_i\colon \ i\in I \}.
	\end{align*}
\end{Cor}
\begin{proof}
	The argument given in \citep[Corollary 5.10]{CRUZ2022410} works in this setup replacing the use of \citep[Theorem 5.2, Proposition 3.23]{CRUZ2022410} by Theorems \ref{moduleMuellerpartone} and \ref{moduleMuellerparttwo}). 
\end{proof}

Usually, proving that a certain exact sequence remains exact under a certain $\Hom$ functor might be difficult.  Sometimes, we can assert that the existence of an $(A, R)$-exact sequence implies the existence of another whose morphisms can be factored through $\add Q$-approximations for some module $Q$.

\begin{Lemma}\label{continuinganapproximationresolution}Let $R$ be a commutative Noetherian ring. Let $A$ be a projective Noetherian $R$-algebra and $Q, \ M\in A\m\cap R\proj$.
	Assume that $Q\rdomdim_{(A, R)} M\geq n\geq 1$ where the $(A, R)$-exact sequence
	\begin{align}
		0\rightarrow M\rightarrow X_1\rightarrow \cdots \rightarrow X_{n-1}, \label{eqcocover30}
	\end{align}where $X_i\in \add Q$, which remains exact under $\Hom_A(-, Q)$, can be continued to an
	$(A, R)$-exact sequence
	\begin{align}
		0\rightarrow M\rightarrow X_1\rightarrow \cdots \rightarrow X_{n-1}\rightarrow Y
	\end{align}where $Y\in \add Q$. Then, $Q\rdomdim_{(A, R)} M\geq n+1$.
\end{Lemma}
\begin{proof}Apply $D$ to the exact sequence (\ref{eqcocover30}).
	Denote by $\alpha_i$ the maps $DX_i\rightarrow DX_{i-1}$, where we fix $X_{-1}:= DM$. Also the map $DY\rightarrow DX_{n-1}$ which we denote by $h$ admits a factorization through $\ker \alpha_{n-1}$, say $\nu\circ \pi$. Since $B$ is a Noetherian $R$-algebra there exists $Z\in \add DQ$ such that  there exists a surjective map \mbox{$g\colon\Hom_A(DQ, Z)\rightarrow \Hom_A(DQ, \ker \alpha_{n-1})$.} Further, by projectization, the map $\Hom_A(DQ, \nu)\circ g$ is equal to $\Hom_A(DQ, f)$ for some $f\in \Hom_A(Z, DX_{n-1})$. By construction, the exact sequence
	$Z\xrightarrow{f} DX_{n-1}\rightarrow \cdots\rightarrow DX_1\rightarrow DM\rightarrow 0$ remains exact under $\Hom_A(DQ, -)$ and if exact it is $(A, R)$-exact. The remaining of the proof is a routine check that $\ker \alpha_{n-1} =\im f$.
	First, observe that $\Hom_A(DQ, \alpha_{n-1}\circ f)=\Hom_A(DQ, \alpha_{n-1})\circ \Hom_A(DQ, \nu)\circ g=0$. Thus, $\alpha_{n-1}\circ f\chi_Z^r=\chi^r_{DX_{n-2}}\circ \Hom_A(DQ, \alpha_{n-1}\circ f)=0$. So, $\alpha_{n-1}\circ f=0$. By definition of kernel, there exists $s\in \Hom_A(Z, \ker \alpha_{n-1})$ such that $f=\nu\circ s$. Since $\Hom_A(DQ, -)$ is left exact, $g=\Hom_A(DQ, s)$. So, $s$ is a right $\add DQ$-approximation of $\ker \alpha_{n-1}$. In particular, there exists $h_1\in \Hom_A(DY, Z)$ such that $\pi=s\circ h_1$. Consequently, $s$ is surjective, as well. This concludes the proof.
\end{proof}

\begin{Remark}\label{RemarkKSX}
	Observe that Theorem 2.8 of \citep{Koenig2001} and  Theorem 2.15 of \cite{Koenig2001} are particular cases of Lemma \ref{continuinganapproximationresolution} (when $n=1$) and Theorem \ref{dualmoduleMuellerpartone}.
\end{Remark}

Recall that ${}^\perp Q=\{M\in A\m\cap R\proj| \Ext_A^{i>0}(M, Q)=0 \}$ is a resolving subcategory of \linebreak$A\m\cap R\proj$.
In contrast to Lemma \ref{continuinganapproximationresolution}, if we know the last map in an exact sequence and its cokernel, then we can deduce the value of relative dominant dimension with respect to a module using that exact sequence.

\begin{Prop}\label{cokernelrolereldomimodule}
	Let $R$ be a commutative Noetherian ring. Let $A$ be a projective Noetherian $R$-algebra and $Q\in A\m\cap R\proj$ so that $\Hom_A(Q, Q)\in R\proj$ and $\Ext_A^{i>0}(Q, Q)=0$.
	Suppose that $M\in {}^{\perp}Q$. An exact sequence
	\begin{align}
		0\rightarrow M\rightarrow Q_1\rightarrow \cdots \rightarrow Q_n\label{eqcocovers32}
	\end{align}yields $Q\rdomdim_{(A, R)} M\geq n$ if and only if  $Q_i\in \add Q$  and the cokernel of $Q_{n-1}\rightarrow Q_n$ belongs to ${}^\perp Q$.
\end{Prop}
\begin{proof}
	Assume that $Q\rdomdim_{(A, R)} M\geq n$. By definition, $Q_i\in \add Q$ and (\ref{eqcocovers32}) is $(A, R)$-exact. Hence, the cokernel of $Q_{n-1}\rightarrow Q_n$ belongs to $A\m\cap R\proj$. Denote by $X_i$ the cokernel of $Q_{i-1}\rightarrow Q_{i}$ and fix $Q_{0}=M$. Combining the conditions of $\Ext_A^{i>0}(Q_i, Q)=0$,  $\Hom_A(-, Q)$ being exact on (\ref{eqcocovers32}) and $M\in {}^\perp Q$,  it follows by induction on $i$ that $X_i\in {}^\perp Q$. 
	
	Conversely, assume that $Q_i\in \add Q$ and the cokernel of $Q_{n-1}\rightarrow Q_n$ belongs to ${}^\perp Q$ which we denote again by $X_n$. So, $X_n\in R\proj$ and (\ref{eqcocovers32}) is $(A, R)$-exact. It follows that $\Ext_A^1(X_i, Q)\simeq \Ext_A^{n-i+1}(X_n, Q)=0$. This means that (\ref{eqcocovers32}) remains exact under $\Hom_A(-, Q)$. So, the result follows.
\end{proof}

We note the following application of Lemma \ref{relativedominantdimensiononses} useful in examples.

\begin{Cor}\label{dominantdimensiononlongexactsequences}
	Let $R$ be a commutative Noetherian ring. Let $A$ be a projective Noetherian $R$-algebra. Assume that $Q\in A\m\cap R\proj$ satisfying in addition that $\Hom_A(Q, Q)\in R\proj$ and $\Ext_A^{i>0}(Q, Q)=0$. Let $M\in A\m\cap R\proj$ and consider the following $(A, R)$-exact sequence $0\rightarrow M\rightarrow Q_1\rightarrow \cdots\rightarrow Q_t\rightarrow X\rightarrow 0.$ If $\Ext_A^i(X, Q)=0$ for $1\leq i\leq t$, then $Q\rdomdim_{(A, R)} M=t+Q\rdomdim_{(A, R)} X$.
\end{Cor}
\begin{proof}Let $C_i$ be the image of the maps $Q_i\rightarrow Q_{i+1}$, $i=1, \ldots, t-1$.
	Since $Q\in Q^{\perp}$, it follows that $\Ext_A^1(C_i, Q)\simeq \Ext_A^{t-i+1}(X, Q)=0$. So, the exact sequences $0\rightarrow C_i\rightarrow Q_{i+1}\rightarrow C_{i+1}\rightarrow 0$ are exact under $\Hom_A(-, Q)$ (also with $C_0=M$ and $C_t=X$). By Lemma \ref{relativedominantdimensiononses} and induction on \mbox{$i$, the result follows.}
\end{proof}

\paragraph{Projective dimension as an upper bound for relative dominant dimension} We should point out that under certain conditions the relative dominant dimension of the regular module with respect to a module $Q$ is bounded above (if finite) by the projective dimension of $Q$.

\begin{Prop}
	Let $R$ be a commutative Noetherian ring. Let $A$ be a projective Noetherian $R$-algebra. Assume that $Q\in A\m\cap R\proj$ satisfying the following:
	\begin{enumerate}
		\item $\Hom_A(Q, Q)\in R\proj$;
		\item The projective dimensions $\pdim_A Q$ and $\pdim_B Q$ are finite;
		\item $\Ext_A^{i>0}(Q, Q)=0$.
	\end{enumerate} 
	If $Q\rdomdim{(A, R)}> \max\{\pdim_B Q, 2\}$, then $Q$ is a  tilting $A$-module.
\end{Prop}
\begin{proof}Fix $n=\pdim_B Q$.
	By assumption, $DQ\rcodomdim_{(A^{op}, R)} DA\geq n+1$, and so $\Tor_i^B(Q, DQ)=0$ $i=1, \ldots, n-1$ and $Q\otimes_B DQ\simeq DA$ by 	Theorem \ref{moduleMuellerparttwo}. Therefore, applying $-\otimes_B DQ$ to a projective resolution of $Q$ over $B$ yields an exact sequence
	$0\rightarrow X_n\rightarrow \cdots \rightarrow X_0\rightarrow Q\otimes_B DQ\rightarrow 0$ with $X_i\in \add DQ$. By applying $D$ to this exact sequence, we obtain that $Q$ is a tilting $A$-module.
\end{proof}

	\subsection{Change of rings on relative dominant dimension with respect to a module}\label{Change of rings on relative dominant dimension with respect to a module}

We will now see that, as the usual relative dominant dimension, relative dominant dimension with respect to a module behaves well under change of rings techniques. As usual, the following results also hold for right $A$-modules and consequently with codominant dimension in place of dominant dimension. For brevity, we consider only the left versions. In the following, we say that a module $Q\in A\m\cap R\proj$ \textbf{has a base change property} if $\Hom_A(Q, Q)\in R\proj$ and for every commutative $R$-algebra Noetherian ring $S$, the canonical map $S\otimes_R \Hom_A(Q, Q)\rightarrow \Hom_{S\otimes_R A}(S\otimes_R Q, S\otimes_R Q)$ is an isomorphism.

\begin{Lemma}\label{changeofringreldomdimone}
	Let $R$ be a commutative Noetherian ring. Let $A$ be a projective Noetherian $R$-algebra. Assume that $Q\in A\m\cap R\proj$ has a base change property. Denote by $B$ the endomorphism algebra $ \End_A(Q)^{op}$. Assume that $M\in A\m\cap R\proj$, satisfies the following two conditions:
	\begin{enumerate}
		\item $\Hom_A(M, Q)\in R\proj$;
		\item The canonical map $R(\mi)\otimes_R \Hom_A(M, Q)\rightarrow \Hom_{A(\mi)}(M(\mi), Q(\mi))$ is an isomorphism for every maximal ideal $\mi$ of $R$.
	\end{enumerate} Then, the following assertions are equivalent.
	\begin{enumerate}[(a)]
		\item $Q\rdomdim_{(A, R)} M\geq 1$;
		\item $S\otimes_R Q\rdomdim_{(S\otimes_R A, S)} S\otimes_R M\geq 1$ for every commutative $R$-algebra and Noetherian ring $S$;
		\item $Q_\mi\rdomdim_{(A_\mi, R_\mi)} M_\mi\geq 1$ for every maximal ideal $\mi$ of $R$;
		\item $Q(\mi)\rdomdim_{A(\mi)} M(\mi)\geq 1$ for every maximal ideal $\mi$ of $R$.
	\end{enumerate}
\end{Lemma}
\begin{proof}Let $S$ be a commutative $R$-algebra.
	Denote by $D_S$ the standard duality with respect to $S$, $\Hom_S(-, S)$. The result follows from the following commutative diagram:
	\begin{equation}
		\begin{tikzcd}
			S\otimes_R \Hom_A(DQ, DM)\otimes_B DQ \arrow[r, "S\otimes_R \chi^r_{DM}"] \arrow[d, "\theta_{S, M}", "\simeq"'] & S\otimes_R DM \arrow[dd, "\simeq"] \\
			S\otimes_R \Hom_A(DQ, DM)\otimes_{S\otimes_R B} S\otimes_R DQ \arrow[d, "\varphi_S"] & \\
			\Hom_{S\otimes_R A}(D_S (S\otimes_R Q), D_S(S\otimes_R M))\otimes_{S\otimes_R B}D_S(S\otimes_R Q) \arrow[r, "\chi^r_{D_SS\otimes_R M}"] & D_S(S\otimes_R M)
		\end{tikzcd}\label{eqcocover34}
	\end{equation}
	where the map $\theta_{S, M}$ is the isomorphism given in Proposition 2.3 of \cite{CRUZ2022410} while $\varphi_S$ is the tensor product of the canonical map given by extension of scalars on $\Hom$ (which is not claimed at the moment to be an isomorphism) with the one providing the isomorphism $S\otimes_R DQ\simeq D_SS\otimes_RQ$. 
	
	The implications $(b)\Rightarrow (c)\Rightarrow (d)$ are immediate. Assume that $(a)$ holds. Then, $\chi_{DM}^r$ is surjective. By the commutative diagram, $\chi_{D_SS\otimes_R M}^r$ is surjective, and so $(b)$ follows. Assume that $(d)$ holds. By condition 2, $\varphi_{R(\mi)}$ must be an isomorphism for every $\mi\in \MaxSpec R$. Thus, by the diagram, $\chi^r_{DM}(\mi)$ is surjective for every maximal ideal $\mi$ of $R$. By Nakayama's Lemma, $\chi^r_{DM}$ is surjective and $(a)$ holds.
\end{proof}

\begin{Lemma}
	Let $R$ be a commutative Noetherian ring. Let $A$ be a projective Noetherian $R$-algebra. Let $Q\in A\m\cap R\proj$ satisfying $\Hom_A(Q, Q)\in R\proj$. Denote by $B$ the endomorphism algebra $ \End_A(Q)^{op}$. For $M\in A\m\cap R\proj$, the following assertions are equivalent.
	\begin{enumerate}[(a)]
		\item $Q\rdomdim_{(A, R)} M\geq n\geq 1$;
		\item $Q\rdomdim_{(S\otimes_R A, S)}  S\otimes_R M\geq n\geq 1$ for every flat commutative $R$-algebra which is a Noetherian ring;
		\item $Q\rdomdim_{(A_\mi, R_\mi)} M_\mi\geq n\geq 1$ for every maximal ideal $\mi$ of $R$.
	\end{enumerate}
\end{Lemma}
\begin{proof}
	By the flatness of $S$, the vertical maps of the commutative diagram (\ref{eqcocover34}) are isomorphisms. So, by Lemma \ref{changeofringreldomdimone}, the implication $(a)\Rightarrow (b)$ is clear for $n=1, 2$. Again, since $S$ is flat and $B$ is finitely generated projective over $R$, $S\otimes_R -$ commutes with $\Tor$ functors over $B$. Therefore, $(b)$ follows by Theorem \ref{moduleMuellerparttwo}. Analogously, we obtain $(c)\Rightarrow (a)$.
\end{proof}

It is no surprise that relative dominant dimension with respect to a module remains stable under extension of scalars to the algebraic closure. For the sake of completeness, we give the result.

\begin{Lemma}\label{relativedominantdimensionalgebraicclosuremodule}
	Let $k$ be a field with algebraic closure $\overline{k}$. Let $A$ be a finite-dimensional $k$-algebra and assume that $Q\in A\m$. Then, $\overline{k}\otimes_k Q\rdomdim_{\overline{k}\otimes_k A}\overline{k}\otimes_k M=Q\rdomdim_A M$.
\end{Lemma}
\begin{proof}
	Of course, $\overline{k}$ is free over $k$. Therefore, $\Tor_i^B(\Hom_A(DQ, DM), DQ)=0$ if and only if \linebreak\mbox{$\Tor_i^{\overline{k}\otimes_k B}(\Hom_{\overline{k}\otimes_k A}(\overline{k}\otimes_k DQ, \overline{k}\otimes_k DM),  \overline{k}\otimes_k DQ)=0$.} By the same reason, $\chi^r_{DM}$ is surjective (or bijective) if and only if $\chi^r_{\overline{k}\otimes_k DM}$ is surjective (or bijective).
\end{proof}

\begin{Lemma}\label{changeringpartonemaximal}
Let $Q\in A\m\cap R\proj$ having a base change property. Denote by $B$ the endomorphism algebra $ \End_A(Q)^{op}$. Let $M\in A\m\cap R\proj$ satisfying $\Hom_A(M, Q)\in R\proj$. Assume that $S$ is a commutative $R$-algebra and a Noetherian ring such that the canonical map \linebreak${S\otimes_R \Hom_A(M, Q)\rightarrow \Hom_{S\otimes_R A}(S\otimes_R M, S\otimes_R Q)}$ is an isomorphism.
 
  Then, $Q\rdomdim_{(A, R)} M\leq S\otimes_R Q\rdomdim_{(S\otimes_R A, S)}S\otimes_R M$.
\end{Lemma}
\begin{proof}
	Assume that $Q\rdomdim_{(A, R)} M\geq n\geq 1$. Then, there exists an $(A, R)$-exact sequence $0\rightarrow M\rightarrow X_1\rightarrow \cdots\rightarrow X_n$ which remains exact under $\Hom_A(-, Q)$, where $X_i\in \add_A Q$. The functor $S\otimes_R -$ preserves $R$-split exact sequences. Hence, $0\rightarrow S\otimes_R M\rightarrow S\otimes_R X_1\rightarrow \cdots\rightarrow S\otimes_R X_n$ is $(S\otimes_R A, S)$-exact and $S\otimes_R X_i\in \add_{S\otimes_R A}S\otimes_RQ$. By assumption, \begin{align}
		\Hom_A(X_n, Q)\rightarrow \Hom_A(X_{n-1}, Q)\rightarrow \cdots\rightarrow \Hom_A(X_1, Q)\rightarrow \Hom_A(M, Q)\rightarrow 0\label{eqcocover35}
	\end{align}is exact. Since $\Hom_A(M, Q)\in R\proj$ (\ref{eqcocover35}) splits over $R$. Thus, (\ref{eqcocover35}) remains exact under $S\otimes_R -$. Using the commutative diagram
	\begin{equation*}
		\begin{tikzcd}[scale cd=0.96, column sep=small]
			S\otimes_R \Hom_A(X_n, Q)\arrow[r] \arrow[d, "\simeq"] & S\otimes_R \Hom_A(X_{n-1}, Q) \arrow[r] \arrow[d, "\simeq"] & \cdots \arrow[r] & S\otimes_R \Hom_A(M, Q)\arrow[d, "\simeq"] \\
			\Hom_{S\otimes_R A}(S\otimes_R X_n, S\otimes_R Q)\arrow[r] & \Hom_{S\otimes_R A}(S\otimes_R X_{n-1}, S\otimes_R Q)\arrow[r] & \cdots \arrow[r]& \Hom_{S\otimes_R A}(S\otimes_R M, S\otimes_R Q)
		\end{tikzcd}
	\end{equation*} it follows that the bottom row is exact. Hence, $S\otimes_RQ\rdomdim_{(A, R)} S\otimes_R M\geq n$.
\end{proof}

Finally, we reach the most important result in this section.

\begin{Theorem}\label{changeofringrelativedomdimrelativetomodule}
	Let $R$ be a commutative Noetherian ring. Let $A$ be a projective Noetherian $R$-algebra. Assume that $Q\in A\m\cap R\proj$ has a base change property. Denote by $B$ the endomorphism algebra $ \End_A(Q)^{op}$. Assume that $M\in A\m\cap R\proj$, satisfies the following two conditions
	\begin{enumerate}
		\item $\Hom_A(M, Q)\in R\proj$;
		\item The canonical map $R(\mi)\otimes_R \Hom_A(M, Q)\rightarrow \Hom_{A(\mi)}(M(\mi), Q(\mi))$ is an isomorphism for every maximal ideal $\mi$ of $R$.
	\end{enumerate} 
	Then, $Q\rdomdim_{(A, R)} M=\inf \{Q(\mi)\rdomdim_{A(\mi)} M(\mi)\colon \mi\in \MaxSpec(R) \}$.
\end{Theorem}
\begin{proof}
	By Lemma \ref{changeringpartonemaximal}, $Q(\mi)\rdomdim_{A(\mi)} M(\mi)\geq Q\rdomdim_{(A, R)} M$ for every maximal ideal $\mi$ of $R$. Conversely,
	assume that $Q(\mi)\rdomdim_{A(\mi)} M(\mi)\geq n$ for every maximal ideal $\mi$ of $R$.  We want to show that $Q\rdomdim_{(A, R)} M\geq n$. If $n=0$, then there is nothing to show. Using the commutative diagram (\ref{eqcocover34}) we obtain that if $n\geq 1$ (resp. $n\geq 2$), then $\chi^r_{DM}(\mi)$ is surjective (resp. bijective) for every maximal ideal $\mi$ of $R$. By Nakayama's Lemma $\chi^r_{DM}$ is surjective and since $DM\in R\proj$, $\chi^r_{DM}$ is bijective in case $n\geq 2$. So, the inequality holds for $n=1, 2$. Assume now that $n\geq 3$. In particular, $Q\rdomdim_{(A, R)} M\geq 2$, and therefore $\Hom_A(DQ, DM)\otimes_B DQ\in R\proj$. By assumption, $\Tor_i^{B(\mi)}(\Hom_{A(\mi)}(M(\mi), Q(\mi)), D_{(\mi)}Q(\mi))=0$ for $1\leq i\leq n-2$ for every maximal ideal $\mi$ of $R$. Let $Q^{\bullet}$ be a deleted $B$-projective resolution of $DQ$. So the chain complex $P^{\bullet}=\Hom_A(M, Q)\otimes_B Q^{\bullet}$ is a projective complex over $R$ since $\Hom_A(M, Q)\in R\proj$. Consider the K\"unneth spectral sequence for chain complexes (see for example \citep[ Theorem 5.6.4]{Weibel2003})
	\begin{align}
		E^2_{i, j}=\Tor_j^R(H_i(\Hom_A(M, Q)\otimes_B Q^{\bullet}), R(\mi))\Rightarrow H_{i+j}(\Hom_A(M, Q)\otimes_B Q^{\bullet}(\mi)).
	\end{align}
	Since $\Hom_A(M, Q)\otimes_B DQ\in R\proj$, $\Hom_A(M, Q)\otimes_B Q^{\bullet}(\mi)$ becomes a deleted projective $B(\mi)$-resolution of $DQ(\mi)$. 
	
	We shall proceed by induction on $1\leq i\leq n-2$ to show that $\Tor_i^B(\Hom_A(M, Q), DQ)=0$. Observe that $\Tor_1^B(\Hom_A(M, Q), DQ)\otimes_R R(\mi)=0$ for every $\mi\in \MaxSpec R$ (see for example \citep[Lemma A.3]{CRUZ2022410}). Hence, $\Tor_1^B(\Hom_A(M, Q), DQ)=0$. Assume now that $\Tor_i^B(\Hom_A(M, Q), DQ)=0$ for all $1\leq i\leq l$ with $1\leq l\leq n-2$ for some $l$. Then, $E_{i, j}^2=0$, for $1\leq i\leq l$, $j\geq 0$ and $E_{0, j}^2=0$, $j>0$. It follows that $\Tor_{l+1}^B(\Hom_A(M, Q), DQ)(\mi)=0$ for every $\mi\in \MaxSpec R$ (see \citep[Lemma A.4]{CRUZ2022410}). Therefore, $\Tor_i^B(\Hom_A(M, Q), DQ)=0$ for $1\leq i\leq n-2$. By 
	Theorem \ref{moduleMuellerparttwo}, the result follows.\end{proof}

\begin{Remark}
	The condition $DQ\otimes_A M\in R\proj$ implies both conditions required in Theorem \ref{changeofringrelativedomdimrelativetomodule} and $DQ\otimes_A Q\in R\proj$ implies that $Q$ has a base change property (see for example the first part of the proof of \citep[Theorem 6.14.]{CRUZ2022410}.)
\end{Remark}

\begin{Remark}
	We can, of course, drop the condition ${}^\natural V\simeq V$ in Proposition \ref{dominantandcodominantoftiltingcomparison} by doing first the field case and then consider the integral case using Theorem \ref{changeofringrelativedomdimrelativetomodule}.
\end{Remark}

Combining Theorem \ref{changeofringrelativedomdimrelativetomodule} with Lemma  \ref{relativedominantdimensionalgebraicclosuremodule}, we obtain that, in most applications,  the computations of relative dominant dimension with respect to a module over a commutative ring can be reduced to computations of relative dominant dimension with respect to a module in the setup of finite-dimensional algebras over algebraically closed fields.

Considering the condition $DQ\otimes_A M\in R\proj$ may seem unnatural but we should refer once again that projective modules, or more generally characteristic tilting modules of split quasi-hereditary algebras do satisfy such a condition. The following result explains why we should expect that there are many modules with such a condition (see also \citep[1.5.2(e), (f)]{zbMATH00971625}).

\begin{Lemma}
 Let $Q\in A\m\cap R\proj$. If $\Ext_{A(\mi)}^1(Q(\mi), Q(\mi))=0$ for every maximal ideal $\mi$ of $R$, then $DQ\otimes_A Q\in R\proj$.
\end{Lemma}
\begin{proof}
	For each maximal ideal $\mi$ of $R$, \begin{align}
		\Tor_1^{A(\mi)}(DQ(\mi), Q(\mi))=\Hom_{R(\mi)}(\Ext_{A(\mi)}^1(Q(\mi), Q(\mi)), R(\mi))=0.
	\end{align}
	Let $Q^{\bullet}$ be a deleted projective $A$-resolution of $Q$. Since $Q\in R\proj$, $Q^{\bullet}(\mi)$ is a deleted projective $A(\mi)$-resolution of $Q(\mi)$. Consider the K\"unneth Spectral sequence with $P=DQ\otimes_A Q^{\bullet}$, (see for example  \citep[ Theorem 5.6.4]{Weibel2003})
	\begin{multline}
		E^2_{i,j}=\Tor_i^R(\Tor_j^A(DQ, Q), R(\mi)) = \Tor_i^R(H_j(DQ\otimes_A Q^{\bullet}), R(\mi)) \\ \Rightarrow H_{i+j}(DQ\otimes_A Q^{\bullet}\otimes_R R(\mi))=\Tor_{i+j}^{A(\mi)}(DQ(\mi), Q(\mi)).
	\end{multline} Then, $E_{1, 0}^2=\Tor_1^R(DQ\otimes_A Q, R(\mi))=0$ for every maximal ideal $\mi$ of $R$ (see for example \citep[Lemma A.3]{CRUZ2022410}). Hence, $DQ\otimes_A Q\in R\proj$.
\end{proof}

	\section{The reduced grade with respect to a module}\label{The reduced grade with respect to a module}

In \citep{zbMATH06409569}, Gao and Koenig  compared the Auslander-Bridger grade with dominant dimension. We will now see that the same method also works for relative dominant dimension over any ring with respect to a module once we replace the $\Ext$ in the notion of grade by $\Tor$, motivating the terminology: cograde. 
There is, however, another modification that needs to be considered. We are not interested in the case of grade being zero, and so we will instead study (the dual notion of) reduced grade (see for example \citep{MR1076061}). 

\begin{Def}
	Let $R$ be a Noetherian commutative ring and $A$ a projective Noetherian $R$-algebra. Let $X\in A^{op}\m\cap R\proj$ and $M\in A\m\cap R\proj$.
	The reduced cograde of $X$ with respect to $M$, written as 	$\cograde_M X$, is defined as the value
	\[
	\cograde_M X=\inf\{i>0| \Tor_i^A(X, M)\neq 0 \}. \]
	Analougously, we can define the reduced cograde of a right module with respect to a left module.
\end{Def}
The following is based on Theorem 2.3 of \citep{zbMATH06409569}.

\begin{Theorem}
	Let $R$ be a Noetherian commutative ring and $A$ a projective Noetherian $R$-algebra. Assume that $Q\in A\m\cap R\proj$ satisfying in addition that $\Hom_A(Q, Q)\in R\proj$. Denote by $B$ the endomorphism algebra $\End_A(Q)^{op}$. For any $Y\in A\m\cap R\proj$ with an exact sequence
	\begin{align}
		Q_1\xrightarrow{f} Q_0\rightarrow Y\rightarrow 0,\label{eqcocover62}
	\end{align} define $X=\coker \Hom_A(f, Q)\in B^{op}\m$. 
Then,
	\begin{align*}
		Q\rdomdim_{(A, R)} Y\geq n\geq 1 \Longleftrightarrow  \cograde_{DQ} X\geq n+1.
	\end{align*}
\end{Theorem}
\begin{proof}
	Applying the functor $\Hom_A(-, Q)$ yields the exact sequence
	\begin{align}
		0\rightarrow \Hom_A(Y, Q)\rightarrow \Hom_A(Q_0, Q)\xrightarrow{\Hom_A(f, Q)} \Hom_A(Q_1, Q)\twoheadrightarrow X.	\label{eq50}\end{align}
	Denote by $C$ the kernel of $\Hom_A(Q_1, Q)\rightarrow X$ which is the same as the image of $\Hom_A(f, Q)$. Since $\Hom_A(Q_1, Q)$ is projective as right $B$-module, applying $-\otimes_B DQ$ yields the exact sequence
	\begin{align}
		0\rightarrow \Tor_1^B(X, DQ)\rightarrow C\otimes_B DQ\xrightarrow{\iota} \Hom_A(Q_1, Q)\otimes_B DQ\rightarrow X\otimes_B DQ\rightarrow 0\label{eqcocover64}
	\end{align}and $\Tor_{i+1}^B(X, DQ)=\Tor_i^B(C, DQ)$, $i\geq 1$. Since $Y\in R\proj$ we can consider the $(A, R)$-exact sequence $DY\hookrightarrow DQ_0\twoheadrightarrow D\im f$. Applying $-\otimes_B DQ$ to the first part of (\ref{eq50}) we obtain the commutative diagram
	\begin{equation}
		\begin{tikzcd}[scale cd=0.95]
			\Tor_1^B(C, DQ)\arrow[r, hookrightarrow] &\Hom_A(DQ, DY)\otimes_B DQ\arrow[r]\arrow[d, "\chi_{DY}^r"]& \Hom_A(DQ, DQ_0)\otimes_B DQ\arrow[d, "\chi_{DQ_0}^r", "\simeq"']\arrow[r, "\pi_2", twoheadrightarrow] & C\otimes_B DQ\\
			0\arrow[r] & DY\arrow[r] &DQ_0\arrow[r, twoheadrightarrow]& D\im f
		\end{tikzcd}\label{eqcocover65}
	\end{equation}and $\Tor_{i+1}^B(C, DQ)\simeq \Tor_i^B(\Hom_A(DQ, DY), DQ)$ for all $i\geq 1$. Thus, $\Tor_i^B(\Hom_A(DQ, DY), DQ)\simeq \Tor_{i+2}^B(X, DQ)$, for all $i\geq 1$. By the commutativity of the diagram (\ref{eqcocover65}) we can complete the diagram with a map \mbox{$g\colon C\otimes_B DQ\rightarrow D\im f$.} By Snake Lemma, there exists an exact sequence $\ker \chi_{DY}^r\rightarrow 0\rightarrow \ker g\rightarrow \coker \chi_{DY}^r\rightarrow 0\rightarrow \coker g\rightarrow 0$. Also, by the diagram (\ref{eqcocover65}) we obtain that $\ker \chi_{DY}^r\simeq \Tor_1^B(C, DQ)$. Therefore, $\ker \chi_{DY}^r\simeq \Tor_2^B(X, DQ)$. So, it remains to show that $\coker \chi_{DY}^r=0$ if and only if $\Tor_1^B(X, DQ)=0$. For that, consider the diagram
	\begin{equation}
		\begin{tikzcd}
			DQ_0\arrow[r, "\pi_1", twoheadrightarrow] \arrow[dr, shorten >=25pt, shorten <=25pt, "\diamond", phantom]& D\im f\arrow[r, "k", hookrightarrow] \arrow[dr, shorten >=25pt, shorten <=25pt, "\star", phantom]& DQ_1\\
			\Hom_A(DQ, DQ_0)\otimes_B DQ\arrow[r, "\pi_2", twoheadrightarrow] \arrow[u, "\chi_{DQ_0}^r"]& C\otimes_B DQ\arrow[r, "\iota"] \arrow[u, "g"]& \Hom_A(DQ, DQ_1)\otimes_B DQ\arrow[u, "\chi_{DQ_1}^r", "\simeq"'] 
		\end{tikzcd}.
	\end{equation}
	By construction of $g$, $\diamond$ is a commutative diagram. Since $C$ is isomorphic to the image of $\Hom_A(DQ, Df)$, $\Hom_A(DQ, Df)\otimes_B DQ$ factors through $C\otimes_B DQ$. More precisely, $\iota\circ \pi_2=\Hom_A(DQ, Df)\otimes_B DQ$. Observe that $k\circ \pi_1=Df$. Hence, the external diagram is commutative. Therefore,
	\begin{align}
		k\circ g\circ \pi_2=k\circ \pi_1\circ \chi_{DQ_0}^r=Df\chi_{DQ_0}^r=\chi_{DQ_1}^r\circ \Hom_A(DQ, Df)\otimes_B DQ=\chi_{DQ_1}^r\circ \iota\circ \pi_2.
	\end{align}By the surjectivity of $\pi_2$, the diagram $\star$ is commutative.  Now if $\coker \chi_{DY}^r=\ker g=0$, then the diagram $\star$ implies that $\iota$ is injective. By (\ref{eqcocover64}), $\Tor_1^B(X, DQ)=0$. Conversely, suppose that $\Tor_1^B(X, DQ)=0$. Then, $\iota$ is injective and $k\circ g=\chi_{DQ_1}^r\circ \iota$ is injective. Thus, $g$ is injective and $\chi_{DY}^r$ is surjective.
\end{proof}

\section{Fully faithfulness of $\Hom_A(Q, -)$ on subcategories of $A\m$}\label{Fully faithfulness of }

We will now use Theorems \ref{moduleMuellerpartone} and \ref{moduleMuellerparttwo} to establish relative codominant dimension with respect to $Q$ as a measure for the strength of the connection between $A$ and $\End_A(Q)^{op}$ via the functor \linebreak$\Hom_A(Q, -)\colon A\m\rightarrow \End_A(Q)^{op}\m$.

\subsection{$Q\rcodomdim {(A, R)} \geq 2$ as a tool for double centralizer properties}

	\begin{Lemma}\label{fullyfaithfulnesstwo}
	Let $\mathcal{X}$ be the full subcategory of $A\m\cap R\proj$ whose modules $X$ satisfy \linebreak $Q\rcodomdim_{(A, R)} X\geq 2$. Then, $\Hom_A(Q, -)$ is fully faithful on $\mathcal{X}$.  
\end{Lemma}
\begin{proof}
	By Theorem \ref{moduleMuellerpartone}, $\chi_{X}$ is an isomorphism for every $X\in \mathcal{X}$. Fix $F_Q=\Hom_A(Q, -)$ and $\mathbb{I}$ its left adjoint. Then, if $F_Qf=0$ for some $f\in \Hom_{\mathcal{X}}(M, N)$, we obtain $f\circ \chi_M=\chi_N\circ \mathbb{I}F_Qf=0$. Hence, in such a case, $f=0$. So, $F_Q$ is faithful. To show fullness, let $g\in \Hom_B(F_QM, F_QN)$ with $M,N\in \mathcal{X}$. Fixing $h=\chi_N \circ \mathbb{I} g\circ \chi_M^{-1}$ we get $F_Qh=g$. 
\end{proof}

The next result says that fully faithfulness of $F_Q$ on $\add DA$ is related with a double centralizer property.
	Such a result is exactly a version of \citep[Proposition 4.33]{Rouquier2008} for general double centralizer properties. This result is known in the literature for Artinian algebras (see for example \mbox{\citep[Corollary 2.4]{zbMATH00423524}}).

\begin{Lemma}\label{cocoversdef}
	Let $A$ be a projective Noetherian $R$-algebra. Let $Q\in A\m\cap R\proj$ satisfying \linebreak$\Hom_A(Q, Q)\in R\proj$ and denote by $B$ the endomorphism algebra $\End_A(Q)^{op}$. The following assertions are equivalent.
	\begin{enumerate}[(i)]
		\item The canonical map of algebras $A\rightarrow \End_B(Q)^{op}$, given by $a\mapsto (q\mapsto aq)$, is an isomorphism.
		\item $D\chi_{X}$ is an isomorphism of right $A$-modules for all $X\in (A, R)\inj\cap R\proj$.
		\item The restriction of $F_Q$ to $\add DA$ is full and faithful. 
	\end{enumerate}
\end{Lemma}
\begin{proof}
	Denote by $\psi$ the canonical map of algebras $A\rightarrow \End_B(Q)^{op}$ and denote by $\omega_X$ the natural transformation between the identity functor and the double dual $DD$ for $X\in A\m$.  The equivalence $(i)\Leftrightarrow (ii)$ follows from the following commutative diagram
	\begin{equation*}
		\begin{tikzcd}
			A \arrow[d, "\psi", swap] \arrow[rrr, "\omega_A", "\simeq"'] & & & DDA \arrow[d, "D\chi_{DA}"]\\
			\End_B(Q) \arrow[r, "\Hom_A(Q {,} \omega_Q)", "\simeq"', outer sep=0.75ex, swap]	& \Hom_B(Q, DDQ)  \arrow[r, "\Hom_B(Q {,} \sigma_Q)", "\simeq"', outer sep=0.75ex, swap] &  \Hom_B(Q, D\Hom_A(Q, DA)) &  D(Q\otimes_B \Hom_A(Q, DA)) \arrow[l, "\theta", "\simeq"']
		\end{tikzcd}
	\end{equation*}
with horizontal maps being isomorphisms, where $\sigma_Q\colon DDQ\rightarrow D\Hom_A(Q, DA)$ is given by \linebreak${h\mapsto (f\mapsto h(f(-)(1_A)))}$, and $\theta$ is the isomorphism given by Tensor-Hom adjunction.

For each $M, N\in A\m$, denote by $F_{M, N}$ the canonical map, induced by the functor $F_Q$, \linebreak $\Hom_A(M, N)\rightarrow \Hom_B(F_QM, F_QN)$, given by $f\mapsto F_Qf=\Hom_A(Q, f)$.
We can fit $\psi$ into the following commutative diagram:
	\begin{equation}
		\begin{tikzcd}
			A \arrow[r, "\simeq"', "\zeta", swap] \arrow[d, "\psi"]& \Hom_A(A, A) \arrow[r, "\psi_{A {,} A}", "\simeq"', swap] & \Hom_A(DA, DA) \arrow[d, "F_{DA, DA}"] \\
			\Hom_B(Q, Q)\arrow[r, "\psi_{Q {,} Q}", "\simeq"'] & \Hom_B(DQ, DQ) \arrow[r, "\kappa", "\simeq"'] & \Hom_B(FDA, FDA)
		\end{tikzcd}.\label{eqcocovers5}
	\end{equation}Here, $\psi_{A, A}$ and $\psi_{Q, Q}$ are isomorphisms (see for example \citep[Proposition 2.2]{CRUZ2022410}), $\zeta$ is given by $\zeta(a)(b)=ab$, $a, b\in A$ and $\kappa=s^{-1}\circ - \circ s$, where $s$ is the isomorphism $\Hom_A(Q, DA)\rightarrow DQ$ given by Tensor-Hom adjunction.
Taking into account that $F_{M\oplus X, N}$ is equivalent to $F_{M, X}\oplus F_{X, N}$ and $F_{M, X\oplus N}$ is equivalent to $F_{M, X}\oplus F_{M, N}$ for every $M, X, N\in A\m$ it follows by (\ref{eqcocovers5}) that $(iii)\Leftrightarrow (i)$.
\end{proof}


\begin{Remark}
	Let $(A, P, V)$ be a RQF3-algebra over a commutative Noetherian ring $R$. If \linebreak $\domdim{(A, R)}\geq 2$, then the Schur functor $\Hom_A(P, -)\colon A\m\rightarrow \End_A(P)^{op}\m$ is fully faithful on $(A, R)\inj \cap R\proj$ (see \citep[Theorem 4.1]{CRUZ2022410} with Lemma \ref{cocoversdef}).
\end{Remark}

\begin{Remark}
	Let $A$ be a finite-dimensional algebra so that $\domdim A\geq 2$. Let $P$ be a faithful projective-injective $A$-module. The Schur functor $\Hom_A(P, -)\colon A\m\rightarrow \End_A(P)^{op}\m$  is fully faithful on $A\proj$ and on $A\inj$ if and only if $A$ is a Morita algebra (see \cite{cruz2021characterisation}). 
\end{Remark}

Following the work developed in \citep[2.2]{zbMATH05344019}, we will see in the following lemma that every split quasi-hereditary algebra has a (partial) tilting module  with a double centralizer property. In the worst case scenario, this (partial) tilting module coincides with the characteristic tilting module.

\begin{Lemma}\label{tiltingmodulescocovers}
	Let $(A, \{\Delta(\lambda)_{\lambda\in \Lambda}\})$ be a split quasi-hereditary algebra and let $T$ be a characteristic tilting module of $A$. Then, there is an exact sequence $0\rightarrow A \rightarrow M\rightarrow X\rightarrow 0$ where $M$ is (partial) tilting and $X\in \mathcal{F}(\Stsim_{A})$.
	Moreover, there exists a (partial) tilting module $Q\in A\m\cap R\proj$ such that $DM\in \add Q$ and \linebreak $Q\rdomdim (A, R)\geq 2$. In particular, there exists a double centralizer property on $Q$.
\end{Lemma}
\begin{proof}
	Denote by $R(A)$ the Ringel dual $\End_A(T)^{op}$. Let $P\twoheadrightarrow T$ be a right projective presentation of $T$ over $R(A)$. Then, $P\in \mathcal{F}(\Stsim_{R(A)^{op}})$. Note that $T\simeq \Hom_A(A, T)\simeq \Hom_{A^{op}}(DT, DA)\in \mathcal{F}(\Stsim_{R(A)^{op}})$. Since $\mathcal{F}(\Stsim_{R(A)^{op}})$ is resolving, so the kernel of $P\rightarrow T$ belongs to $\mathcal{F}(\Stsim_{R(A)^{op}})$. Since $\Hom_{A^{op}}(DT, -)$ gives an exact equivalence between $\mathcal{F}(\Cssim_{A^{op}})$ and $\mathcal{F}(\Stsim_{R(A)^{op}})$ there exists an exact sequence of right $A$-modules $
		0\rightarrow K\rightarrow M' \rightarrow DA\rightarrow 0,
$ where $M'$ is a (partial) tilting module and $K\in \mathcal{F}(\Cssim_{A^{op}})$. Applying $D$ we obtain the desired exact sequence. By \citep[Proposition 4.26]{Rouquier2008}, since $DK\in \mathcal{F}(\Stsim_A)$ there exists an exact sequence $0\rightarrow DK\rightarrow M''\rightarrow K''\rightarrow 0$, where $M'$ is a (partial) tilting module and $K'\in \mathcal{F}(\Stsim_A)$. Put $Q=DM'\oplus M''$. Hence, $Q$ is (partial) tilting module and the $(A, R)$-exact sequence $0\rightarrow A\rightarrow DM'\rightarrow M''$ remains exact under $\Hom_A(Q, -)$. This means that $Q\rdomdim (A, R)\geq 2$.  By Corollary \ref{rightandleftrelativedomidimension} and Lemma \ref{cocoversdef}, there exists a double centralizer property on $Q$ between $A$ and $\End_A(Q)^{op}$.
\end{proof} 

\subsection{$Q\rcodomdim_{(A, R)} M$ controlling the behaviour of $\Hom_A(Q, -)$}

To address how the relative dominant dimension with respect to a module $Q$ can be used as a tool to deduce what extension groups are preserved by the functor $\Hom_A(Q, -)$ the following result is crucial.

	\begin{Lemma}\label{sesforcocovers}
	Let $M\in A\m$. Suppose that $\Tor_i^B(Q, F_QM)=\operatorname{L}_i \mathbb{I}(F_QM)=0$ for $1\leq i\leq q$. For any $X\in Q^{\perp}:=\{Y\in A\m| \Ext_A^{i>0}(Q, Y)=0 \}$, there are isomorphisms $\Ext_A^i(\mathbb{I}F_QM, X)\simeq \Ext_B^i(F_QM, F_QX),$ $0\leq i\leq q$, and an exact sequence
	\begin{multline}
		0\rightarrow \Ext_A^{q+1}(\mathbb{I}F_QM, X)\rightarrow \Ext_B^{q+1}(F_QM, F_QX)\rightarrow \Hom_A(\Tor_{q+1}^B(Q, F_QM), X) \\\rightarrow \Ext_A^{q+2}(\mathbb{I}F_QM, X)\rightarrow \Ext_B^{q+2}(F_QM, F_QX).
	\end{multline}
\end{Lemma}
\begin{proof}
	Let $X\in A\m$ such that $\Ext_A^{i>0}(Q, X)=0$. Fix $i=0$. Then, by Tensor-Hom adjunction,
	\begin{align}
		\Hom_A(\mathbb{I}F_QM, X)\simeq \Hom_B(F_QM, F_QX).
	\end{align}To obtain the result for higher values we will use Theorem 10.49 of \cite{Rotman2009a}. So, fix $f=\Hom_A(-, X)$ and $g=Q\otimes_B -$. $f$ is a contravariant left exact and $g$ is covariant. We note that $gP$ is $f$-acyclic for any $P\in B\proj$. In fact, 
	$\R^{j>0}f(gP)=\Ext_A^{j>0}(gP, X)=0$, since $gP=Q\otimes_B P\in \add_A Q$. So, for each $a\in B\m$, there is a spectral sequence \begin{align}
		E_2^{i, j}=(\R^if)(L_j g)(a)\Rightarrow \R^{i+j}(f\circ g)(a).
	\end{align}
	By Tensor-Hom adjunction $f\circ g(N)=\Hom_A(Q\otimes_B N, X)\simeq \Hom_B(N, \Hom_A(Q, X))=\Hom_B(-, F_QX)(N)$, for every $N\in B\m$. Hence, we can rewrite the previous spectral sequence into
	\begin{align}
		E_2^{i, j}=\Ext_A^i(\Tor_j^B(Q, a), X)\Rightarrow \Ext_B^{i+j}(a, F_QX).
	\end{align}For each $M\in A\m$, fix $a=F_QM$. By assumption, $\Tor_i^B(Q, F_QM)=0$ for $1\leq i\leq q$. Hence, $E_2^{i, j}=0$, $1\leq i\leq q$. By the dual of Lemma A.4 of \citep{CRUZ2022410}, the result follows.
\end{proof}

\subsection{Ringel duality in cover theory} \label{Ringel duality}

As we saw, fully faithfulness of $\Hom_A(Q, -)$ on relative injectives is related to the existence of a double centralizer property on $Q$. 
In this subsection, we explore the meaning of fully faithfulness of a functor $F_Q$ on $\mathcal{F}(\Cssim)$ for relative projective objects $Q$ in $\mathcal{F}(\Cssim)$. This leads us to one of our main results, connecting Ringel duality with Rouquier's cover theory.


\begin{Theorem}\label{ringeldualitycovers}
	Let $R$ be a commutative Noetherian ring. Let $(A, \{\Delta(\lambda)_{\lambda\in \Lambda}\})$ be a split quasi-hereditary $R$-algebra with a characteristic tilting module $T$. Denote by $R(A)$ the Ringel dual $\End_A(T)^{op}$ (of $A$).
	Assume that $Q\in \add T$ is a (partial) tilting module of $A$ and denote by $B$ the endomorphism algebra of $Q$. Then, the following assertions hold.
	\begin{enumerate}[(a)]
		\item  $Q\rcodomdim_{(A, R)} \mathcal{F}(\Cssim)= Q\rcodomdim_{(A, R)} T = Q\rcodomdim_{(A, R)} \displaystyle\bigoplus_{\l\in \L}\Cs(\l)$.
		\item If $Q\rcodomdim_{(A, R)} T\geq n\geq 2$, then the functor $F_Q$ induces isomorphisms 
		\begin{align*}
			\Ext_A^j(M, N)\rightarrow \Ext_B^j(F_QM, F_QN), \quad \forall M, N\in \mathcal{F}(\Cssim), \ 0\leq j\leq n-2.
		\end{align*}
		\item If $Q\rcodomdim_{(A, R)} T \geq 3$, then the functor $F_Q$ induces an exact equivalence $\mathcal{F}(\Cssim)\rightarrow \mathcal{F}(F\Cssim)$.
		\item If $Q\rcodomdim_{(A, R)} T \geq n\geq 2$, then $(R(A), \Hom_A(T, Q))$ is an $(n-2)$-$\mathcal{F}(\Stsim_{R(A)})$ split quasi-hereditary cover of $\End_A(Q)^{op}$.
	\end{enumerate}
\end{Theorem}
\begin{proof}
	The proof of $(a)$ is analogous to \citep[Theorem 6.2.1]{p2paper}. Thanks to $F_Q$ and $D$ being exact on short exact sequences of modules belonging to $\mathcal{F}(\Cssim)$ we obtain that we can apply Lemma \ref{relativedominantdimensiononses} to the filtrations by costandard modules. Further, for every $X\in R\proj$ so that $R^t\simeq X\oplus Y$ for some $t\in \mathbb{N}$ \begin{align*}
		Q\rcodomdim_{(A, R)} \Cs(\l)&=Q\rcodomdim_{(A, R)} \Cs(\l)^t\\&=\inf \{Q\rcodomdim_{(A, R)} \Cs(\l)\otimes_R X, Q\rcodomdim_{(A, R)} \Cs(\l)\otimes_R Y \}.
	\end{align*}Therefore, $Q\rcodomdim_{(A, R)} \displaystyle\bigoplus_{\l\in \L}\Cs(\l)= Q\rcodomdim_{(A, R)} \mathcal{F}(\Cssim)$. Now using the exact sequences (\ref{eq121b}) together with Lemma \ref{relativedominantdimensiononses} and the reasoning of Theorem \citep[Theorem 6.2.1]{p2paper}, assertion (a) follows.
	
	By Proposition \ref{qhproperties}, $DQ\otimes_A Q\in R\proj$.
 Any $(A,R)$-injective module being projective over $R$ belongs to $Q^\perp$ (see for example \citep[Proposition 2.10]{CRUZ2022410}). By Lemma \ref{fullyfaithfulnesstwo} and Theorem \ref{moduleMuellerpartone}, $F_Q$ is fully faithful on $\mathcal{F}(\Stsim)$. By Theorem \ref{moduleMuellerparttwo}, $\Tor_i^B(Q, F_QM)=0$, $1\leq i\leq n-2$ for every $M\in \mathcal{F}(\Cssim)$. By Lemma \ref{sesforcocovers}, $\Ext_B^i(F_QM, F_QX)\simeq \Ext_A^i(\mathbb{I}F_QM, X)$ for $0\leq i\leq n-2$, where $M, X\in \mathcal{F}(\Cssim)$. Since $\chi_M$ is an isomorphism for every $M\in \mathcal{F}(\Cssim)$, (b) follows.
	
	By the exactness of $F_Q$ on $\mathcal{F}(\Cssim)$ and according to Lemma \ref{tensorprojcommutingonHom}, $F_Q(\Cs(\l)\otimes_R X)\simeq F_Q\Cs(\l)\otimes_R X$ for every $\l\in \L$ and $X\in R\proj$, and the restriction of the functor $F_Q$ on $\mathcal{F}(\Cssim)$ has image in $\mathcal{F}(F_Q\Cssim)$. By (b), it is enough to prove that for each module $M$ in $\mathcal{F}(F_Q\Cssim)$, there exists $N\in \mathcal{F}(\Cssim)$ so that $F_QN\simeq M$. By (b), the functor $\mathbb{I}=Q\otimes_B -$ is exact on short exact sequences of modules belonging to $\mathcal{F}(F_Q\Cssim)$. Thanks to $Q\otimes_B F_Q\Cs(\l)\otimes_R X\simeq \Cs(\l)\otimes_R X$ for every $X\in R\proj$ and $\l\in \L$, we obtain that $\mathbb{I}$ sends $\mathcal{F}(F_Q\Cssim)$ to $\mathcal{F}(\Cssim)$. So, (c) follows.
	
	Assume now that $Q\rcodomdim_{(A, R)} T\geq n\geq 2$. Fix $B=\End_A(Q)^{op}$. By Ringel duality (see for instance \citep[Lemma 7.1]{appendix}), for each $M\in \mathcal{F}(\Cssim)$, we have \begin{align}
		\Hom_{R(A)}(\Hom_A(T, Q), \Hom_A(T, M))\simeq \Hom_A(Q, M) \label{eq5503}
	\end{align}as $(B, R(A))$-bimodules. In particular, $\Hom_{R(A)}(\Hom_A(T, Q), R(A))\simeq \Hom_A(Q, T)$ as $(B, R(A))$-bimodules. By (c), $F_Q$ is fully faithful on $\mathcal{F}(\Cssim)$. Hence, $
		\End_B(\Hom_A(Q, T))^{op}\simeq \End_A(T)^{op}
$ and $\End_{R(A)}(\Hom_A(T, Q))^{op}\simeq \End_A(Q)^{op}$. So, $(R(A), \Hom_A(T, Q))$ is a split quasi-hereditary cover of $B$. Now, by (b) and \citep[Lemma 7.1]{appendix}, for each $M\in \mathcal{F}(\Cssim)$,
	\begin{align}
		\Ext_B^i(\Hom_A(Q, T), \Hom_{R(A)}(\Hom_A(T, Q), \Hom_A(T, M)))&\simeq \Ext_B^i(\Hom_A(Q, T), \Hom_A(Q, M))\\\simeq &\Ext_B^i(F_QT, F_QM), \quad 0\leq i\leq n-2, \label{eq5506}
	\end{align}which vanishes if $1\geq i\geq n-2$. By \citep[Lemma 7.1]{appendix}, and Proposition \ref{zeroAcover}, we conclude the proof.
\end{proof}

Since every projective module is the image of a (partial) tilting under the Ringel dual functor, every quasi-hereditary cover can be recovered/discovered using this approach. More precisely, every split quasi-hereditary algebra $A$ is Morita equivalent to the Ringel dual of its Ringel dual $R(R(A))$ and every projective over $R(R(A))$ can be written as $\Hom_{R(A)}(T_{R(A)}, Q)$ for some $Q\in \add T_{R(A)}$, where $T_{R(A)}$ is a characteristic tilting module of $R(A)$. Hence, every split quasi-hereditary cover can be written in the form $(R(A), \Hom_A(T, Q))$ for some split quasi-hereditary algebra $A$, $T$ a characteristic tilting module and $Q\in \add T$.

\begin{Remark}
	$T\rcodomdim_{(A, R)} \mathcal{F}(\Cssim)=T\rcodomdim_{(A, R)} T=+\infty$ for a characteristic tilting module $T$. Of course, the Ringel dual is an infinite cover of itself.
\end{Remark}

\begin{Remark}\label{exactequivalences}
	The cover constructed in Theorem \ref{ringeldualitycovers} makes the following diagram commutative
	\begin{equation}
		\begin{tikzcd}
			\mathcal{F}(\Cssim_A)\arrow[rr, "\Hom_A(T{,} -)"] \arrow[dr, "\Hom_A(Q{,} -)", swap] & & \mathcal{F}(\Stsim_{R(A)}) \arrow[dl, "\Hom_{R(A)}(\Hom_A(T{,} Q){,} -)"]\\
			& \mathcal{F}(F\Cssim_A) &
		\end{tikzcd}.
	\end{equation}
\end{Remark}

\begin{Cor}\label{codominantdimensioncontrollfinitedimensionalcase}
	Let $k$ be a field. Let $(A, \{\Delta(\lambda)_{\lambda\in \Lambda}\})$ be a split quasi-hereditary $k$-algebra with a characteristic tilting module $T$. Denote by $R(A)$ the Ringel dual $\End_A(T)^{op}$ (of $A$).
	Assume that $Q\in \add T$ is a (partial) tilting module of $A$ and $n\geq 2$ is a natural number.  Then, ${Q\rcodomdim_{(A, R)} T \geq n\geq 2}$ if and only if $(R(A), \Hom_A(T, Q))$ is an $(n-2)$-$\mathcal{F}(\Stsim_{R(A)})$ split quasi-hereditary cover of $\End_A(Q)^{op}$.
\end{Cor}
\begin{proof}
	It follows by Theorem \ref{ringeldualitycovers}, equations (\ref{eq5503}) and (\ref{eq5506}) together with Theorem \ref{moduleMuellerparttwo} and the commutative diagram
\begin{equation}
		\begin{tikzcd}
		DT \arrow[r] \arrow[d, "\simeq"] & \Hom_{\End_{A}(Q)^{op}}(\Hom_A(Q, T), DQ) \arrow[d, "\simeq"] \\
		\Hom_A(T, DA)\arrow[r] & \Hom_{\End_A(Q)^{op}}(\Hom_A(Q, T), \Hom_A(Q, DA))  
	\end{tikzcd}. \tag*{\qedhere} 
\end{equation}
\end{proof}

Therefore, the previous results say that the quality of faithful split quasi-hereditary covers of finite-dimensional algebras is controlled by the relative codominant dimension of characteristic tilting modules with respect to (partial) tilting modules.

	\subsubsection{An analogue of Lemma \ref{cocoversdef} for Ringel duality}

\begin{Lemma}\label{analogueofdpclemmaringel}
	Let $R$ be a commutative Noetherian ring. Let $(A, \{\Delta(\lambda)_{\lambda\in \Lambda}\})$ be a split quasi-hereditary $R$-algebra with a characteristic tilting module $T$. Denote by $R(A)$ the Ringel dual $\End_A(T)^{op}$ (of $A$).
	Assume that $Q\in \add T$ is a (partial) tilting module of $A$ and fix $B=\End_A(Q)^{op}$.
	Then, the following assertions hold.
	\begin{enumerate}[(a)]
		\item If $D\chi_T\colon DT\rightarrow \Hom_B(\Hom_A(Q, T), DQ)$ is an isomorphism, then $(R(A), \Hom_A(T, Q))$ is a split quasi-hereditary cover of $B$.
		\item If $D\chi^r_{DT}\colon DDT\rightarrow \Hom_B(\Hom_A(DQ, DT), DDQ )$ is an isomorphism, then $\Hom_A(T, Q)$ satisfies a double centralizer property between $R(A)$ and $B$.
	\end{enumerate}
\end{Lemma}
\begin{proof}
	By projectivization, $\Hom_A(T, Q)\in R(A)\proj$ and $\End_{R(A)}(\Hom_A(T, Q))^{op}\simeq \End_A(Q)^{op}=B$. By $(a)$ and Proposition \ref{qhproperties}, we have as $(R(A), R(A))$-bimodules 
	\begin{align}
		R(A)&=\Hom_A(T, T)\simeq \Hom_{A^{op}}(DT, DT)\simeq \Hom_{A^{op}}(DT, \Hom_B(\Hom_A(Q, T), DQ))\\&\simeq \Hom_{A^{op}}(DT, \Hom_{B^{op}}(Q, D\Hom_{A}(Q, T)) )
		\simeq \Hom_{B^{op}}(DT\otimes_A Q, D\Hom_{A}(Q, T))\\
		&\simeq \Hom_B(\Hom_{A}(Q, T), \Hom_A(Q, T)).
	\end{align} Since $F_QR(A)=\Hom_{R(A)}(\Hom_A(T, Q), \Hom_A(T, T))\simeq \Hom_A(Q, T)$ assertion (a) follows.
	
	Now using the isomorphism $\chi_{DT}^r$ and Proposition \ref{qhproperties} we obtain
	\begin{align*}
		R(A)=\Hom_A(T, T)&\simeq \Hom_A(T, \Hom_{B^{op}}(\Hom_A(T, Q), Q))\simeq \Hom_A(T, \Hom_B(DQ, D\Hom_A(T, Q)))\\
		&\simeq \Hom_B(DQ\otimes_A T, D\Hom_A(T, Q))\simeq \Hom_B(\Hom_A(T, Q), \Hom_A(T, Q)). \tag*{\qedhere}
	\end{align*}
\end{proof}

We did not yet address the case of $Q\rcodomdim_{(A, R)} T=1$.  For this case, we can recover the Ringel dual being a cover using deformation theory.

\begin{Cor}\label{deformationringeldualfunctor}
	Let $R$ be a commutative regular Noetherian domain with quotient field $K$. Let \linebreak$(A, \{\Delta(\lambda)_{\lambda\in \Lambda}\})$ be a split quasi-hereditary $R$-algebra with a characteristic tilting module $T$. Let $R(A)$ be the Ringel dual of $A$, $\End_A(T)^{op}$.
	Assume that $Q\in \add T$ is a  (partial) tilting module of $A$ so that $Q\rcodomdim_{(A, R)} T\geq 1$ and $K\otimes_R Q\rcodomdim_{(K\otimes_R A)} K\otimes_R T\geq 2$. Then, $(R(A), \Hom_A(T, Q))$ is a split quasi-hereditary cover of $\End_A(Q)^{op}$. Moreover, $(R(A), \Hom_A(T, Q))$  is a $0$-$\mathcal{F}(\Stsim)$ split quasi-hereditary cover of $\End_A(Q)^{op}$. 
\end{Cor}
\begin{proof}
	If $Q\rcodomdim_{(A, R)} T\geq 2$, then this is nothing more than Theorem \ref{ringeldualitycovers}. Assume that \linebreak $Q\rdomdim_{(A, R)} T=1$. By Theorem \ref{moduleMuellerpartone}, $\chi_T$ is surjective. In view of Lemma \ref{analogueofdpclemmaringel}, it is enough to prove that $D\chi_T$ is an isomorphism. Since $T\in R\proj$, $\chi_T$ is an $(A, R)$-epimorphism, and therefore $D\chi_T$ is an $(A, R)$-monomorphism. By assumption, $K\otimes_R Q\rcodomdim_{K\otimes_R A} K\otimes_R T\geq2$. Hence, thanks to the flatness of $K$, $K\otimes_R D\chi_T$ is an isomorphism.

	Denote by $X$ the cokernel of $D\chi_T$. Since $D\chi_T$ is split over $R$, $X\in \add_R \Hom_B(\Hom_A(Q, T), DQ)$. As we saw, $K\otimes_R X=0$. In particular, $X$ is a torsion $R$-module. We cannot deduce right away that $\Hom_B(\Hom_A(Q, T), DQ)$ is projective over $R$ but we can embed $\Hom_B(\Hom_A(Q, T), DQ)$ into $\Hom_R(\Hom_A(Q, T), DQ)$ which is projective over $R$ due to both $\Hom_A(Q, T)$ and $DQ$ being projective over $R$. So, $\Hom_B(\Hom_A(Q, T), DQ)$ is a torsion free $R$-module. On the other hand, the localisation $\Hom_B(\Hom_A(Q, T), DQ)_\pri$ is projective over $R_\pri$ for every prime ideal $\pri$ of height one, and so $K\otimes_R X=0$ yields that $X_\pri=0$ for every prime ideal $\pri$ of height one.
	 Applying Proposition 3.4 of \citep{zbMATH03151673} to $D\chi_T$ we obtain that $X_\pri$ must be zero, and consequently, $D\chi_T$ is an isomorphism. Denote by $F_{\Hom_A(T, Q)}$ the Schur functor and $G_{\Hom_A(T, Q)}$ its right adjoint of this cover. Observe that $\Hom_A(T, DA)$ is a characteristic tilting module of $R(A)$. Since $D\chi_T$ is a monomorphism and \begin{align}
		\Hom_B(\Hom_A(Q, T), DQ)&\simeq \Hom_B(\Hom_A(Q, T), \Hom_A(Q, DA))\\&\simeq \Hom_B(\Hom_{R(A)}(F_TQ, F_TT), \Hom_{R(A)}(F_TQ, F_TDA))\nonumber\\&\simeq G_{\Hom_A(T, Q)}F_{\Hom_A(T, Q)} \Hom_A(T, DA),
	\end{align}the claim follows by Remark \ref{casezerodeltacover}.
\end{proof}

\subsection{Ringel self-duality as an instance of uniqueness of covers}\label{Ringel self-duality as an instance of uniqueness of covers}
We will now see how Ringel self-duality can be related with uniqueness of covers.

\begin{Cor}\label{Ringelselfdualcodomdimtwo}
	Let $(A, P, DP)$ be a relative Morita $R$-algebra. Let $(A, \{\Delta(\lambda)_{\lambda\in \Lambda}\})$ be a split quasi-hereditary algebra, Assume that
	$\domdim_{(A, R)} T, \codomdim_{(A, R)} T\geq 3$ for a characteristic tilting module $T$. Then,
	there exists an exact equivalence $\mathcal{F}(F_P\Stsim_A)\rightarrow \mathcal{F}(F_P\Cssim_A)$ if and only if
	$A$ is Ringel self-dual.
\end{Cor}
\begin{proof}
	By Theorems \ref{ringeldualitycovers} and Proposition \ref{domdimtoolcover}, $(A, P)$ is a 1-faithful split quasi-hereditary cover of $\End_A(P)^{op}$ and $(R(A), \Hom_A(T, P))$ is a 1-faithful split quasi-hereditary cover of $\End_A(P)^{op}$. As illustrated in Remark \ref{exactequivalences}, $F$ restricts to exact equivalences $\mathcal{F}(\Cssim_A)\rightarrow \mathcal{F}(F_P\Cssim_A)$ and $\mathcal{F}(\Stsim_A)\rightarrow \mathcal{F}(F_P\Stsim_A)$. Therefore, there exists an exact equivalence between $\mathcal{F}(F_P\Stsim_A)$ and $\mathcal{F}(F_P\Cssim_A)$ if and only if there exists an exact equivalence between $\mathcal{F}(\Stsim_A)$ and $\mathcal{F}(\Cssim_A)$. Hence $A$ is Ringel self-dual.
\end{proof}

This is an indication that the phenomenon of Ringel self-duality behaves better the larger the dominant dimension of the characteristic tilting module.   
As before, for deformations we can weaken the conditions on the dominant and codominant dimension of the characteristic tilting module.

\begin{Cor}\label{Ringelselfdualcodomdim}
	Let $R$ be an integral regular domain with quotient field $K$. Let $(A, \{\Delta(\lambda)_{\lambda\in \Lambda}\})$ be a split quasi-hereditary $R$-algebra. Let $(A, P, V)$ be a relative Morita $R$-algebra. Fix $B=\End_A(P)^{op}$. Assume the following conditions hold.
	\begin{enumerate}[(i)]
		\item $(K\otimes_R A, K\otimes_R P)$ is a 1-faithful split quasi-hereditary cover of $B$;
		\item  $K\otimes_R P\rcodomdim K\otimes_R T\geq 3$ for a characteristic tilting module $T$;
		\item $\domdim_{(A, R)} T$, $\codomdim_{(A, R)} T\geq 2$ for a characteristic tilting module $T$;
		\item There exists an exact equivalence $\mathcal{F}(F_P\Stsim_A)\rightarrow \mathcal{F}(F_P\Cssim_A)$.
	\end{enumerate} 
	Then, $A$ is Ringel self-dual.
\end{Cor}
\begin{proof}
	Observe that $\domdim_{A(\mi)} T(\mi)\geq 2$ and $\domdim_{A^{op}(\mi)} DT(\mi)=\codomdim_{A(\mi)} T(\mi)\geq 2$ for every maximal ideal $\mi$ of $R$. By Theorem \ref{ringeldualitycovers} and Corollary \ref{codominantdimensioncontrollfinitedimensionalcase}, $$(R(K\otimes_R A), \Hom_{K\otimes_R A}(K\otimes_R T, K\otimes_R P))=(K\otimes_R R(A), K\otimes_R \Hom_A(T, P))$$ is an 1-faithful split quasi-hereditary cover of $K\otimes_R B$, and $(R(A(\mi)), \Hom_A(T, P)(\mi))$ is a $0$-faithful split quasi-hereditary cover of $B(\mi)$ for every maximal ideal $\mi$ of $R$. By Proposition \ref{domdimtoolcover}(b), $(A(\mi), P(\mi))$ is a $0$-faithful split quasi-hereditary cover of $B(\mi)$ for every $\mi\in \MaxSpec R$. By Theorem \ref{improvingcoverwithspectrum} and Proposition \ref{faithfulcoverresiduefield}, $(R(A), \Hom_A(T, P))$ and $(A, P)$ are 1-faithful split quasi-hereditary covers of $B$. By Remark \ref{exactequivalences} and Proposition  \ref{onefaithfulcovers}, there exists an exact equivalence, \begin{align}
		\mathcal{F}(\Stsim_A)\rightarrow \mathcal{F}(F_P\Stsim_A)\xrightarrow{(iv)} \mathcal{F}(F_P\Cssim_A)\rightarrow \mathcal{F}(\Cssim_A).
	\end{align}
	Therefore, $A$ is Ringel self-dual.
\end{proof}

	\section{Wakamatsu tilting conjecture for quasi-hereditary algebras} \label{Wakamatsu tilting conjecture for quasi-hereditary algebras}

 Theorem \ref{moduleMuellerparttwo} is the main advantage of Definition \ref{relativedomdimrelativedef} compared to \citep[Definition 2.1]{Koenig2001} giving a meaning to what this relative dominant measures. Another point of view to be referred to is the Wakamatsu tilting conjecture (see \cite{zbMATH04053850}). In this context, the Wakamatsu tilting conjecture says that if $Q$ has finite projective $A$-dimension and it admits no self-extensions in any degree, then $Q\rdomdim {(A,R)}$ measures how far $Q$ is from being a tilting module. In particular, for split quasi-hereditary algebras this amounts to saying that for a module $Q$ in the additive closure of a characteristic tilting module, $Q\rdomdim {(A, R)}$ measures how far $Q$ is from being a characteristic tilting module of $A$. This is indeed the case and it follows as an application of Theorem \ref{ringeldualitycovers}.

\begin{Theorem}Let $R$ be a Noetherian commutative ring and
	$(A, \{\Delta(\lambda)_{\lambda\in \Lambda}\})$ be a split quasi-hereditary $R$-algebra. Assume that $T$ is a characteristic tilting module and $Q\in \add_A T$ is a partial tilting module.
	
	If $Q\rdomdim{(A, R)}=+\infty$, then $Q$ is a characteristic tilting module of $A$.  
\end{Theorem}
\begin{proof}Consider first that $R$ is a field.
	By assumption, $DQ\rcodomdim_{(A^{op}, R)} DA=+\infty$. $\Hom_{A^{op}}(DQ, -)$ is exact on $\mathcal{F}(\Cssim_{A^{op}})$, and so, we obtain by Lemma \ref{relativedominantdimensiononses} that $DQ\rcodomdim_{(A^{op}, R)} \mathcal{F}(\Cssim_{A^{op}})=+\infty$. By Theorem \ref{ringeldualitycovers}, $(\End_{A^{op}}(DT)^{op}, \Hom_{A^{op}}(DT, DQ))$ is an $+\infty$ faithful split quasi-hereditary cover of $\End_{A^{op}}(DQ)^{op}$. By \citep[Corollary 4.2.2]{p2paper}, $\End_{A^{op}}(DT)^{op}$ is Morita equivalent to $\End_{A^{op}}(DQ)^{op}$. In particular, by projectization, $DT$ and $DQ$  have the same number of indecomposable modules. Therefore, $\add_{A^{op}} DQ=\add_{A^{op}} DT$, and so $Q$ is a characteristic tilting module. Assume now that $R$ is an arbitrary Noetherian commutative ring. If $Q\rdomdim_{(A, R)} =+\infty$, then we have $Q(\mi)\rdomdim_{A(\mi)}=+\infty$ for every maximal ideal $\mi$ of $R$. Hence, $Q(\mi)$ is a characteristic tilting module for every maximal ideal $\mi$ of $R$, that is, $\add Q(\mi)=\mathcal{F}(\St(\mi))\cap \mathcal{F}(\Cs(\mi))$. By Proposition \ref{standardscotiltingsreductiontofields}, the result follows.
\end{proof}

\subsection{Measuring $Q\rdomdim_{(A, R)} T$ using projective resolutions}
We can also try to compute $Q\rdomdim_{(A, R)} T$ using projective resolutions over a Ringel dual of $A$. Moreover, this perspective leads us to reformulate that $Q\rdomdim_{(A, R)} T$ measures how much the projective module $\Hom_A(T, Q)$ controls projective resolutions of characteristic tilting modules over the Ringel dual.

\begin{Prop}
	Let $(A, \{\Delta(\lambda)_{\lambda\in \Lambda}\})$ be a split quasi-hereditary $R$-algebra with a characteristic tilting module $T$. Denote by $R(A)$ the Ringel dual $\End_A(T)^{op}$ (of $A$).
	Suppose that $Q\in \add_A T$ is a partial tilting module. Then, $\Hom_A(T, Q)\rcodomdim_{(R(A), R)} DT=Q\rdomdim_{(A, R)} T$.
\end{Prop}
\begin{proof}
	Observe that $\End_{R(A)}(\Hom_A(T, Q))^{op}\simeq \End_A(Q)^{op}$ since $Q\in \add T$. Denote by $B$ this endomorphism algebra and  by $H$ the Ringel dual functor $\Hom_A(T, -)\colon A\m\rightarrow R(A)\m$. Observe also that the following diagram is commutative:
	\begin{equation*}
		\begin{tikzcd}[column sep=20pt,row sep=10pt]
			\Hom_A(T, Q)\otimes_B \Hom_A(Q, DA) \arrow[r, "\Hom_A(T{,} Q)\otimes_B \alpha_Q", outer sep=0.75ex, "\simeq"']     \arrow[d, "\Hom_A(T{,} Q)\otimes_B H_{Q{,} DA}", outer sep=0.75ex, "\simeq"'] & \Hom_A(T, Q)\otimes_B DQ  \arrow[r, "\psi_{T, Q}", "\simeq"'] & \Hom_A(DQ, DT)\otimes_B DQ \arrow[dd, "\chi_{DT}^r"] \\
			\Hom_A(T, Q)\otimes_B \Hom_{R(A)}(HQ, H(DA)) \arrow[d, "\Hom_A(T{,}Q)\otimes_B \alpha_T", "\simeq"'] \\
			\Hom_A(T, Q)\otimes_B \Hom_{R(A)}(HQ, DT)\arrow[rr, "\chi_{DT}"] & &  DT
		\end{tikzcd}.
	\end{equation*}
Here, the $H_{Q, DA}$ is the induced isomorphism of $H$ by fully faithfulness on $\mathcal{F}(\Cssim)$, $\alpha_M$ is the isomorphism  $\Hom_A(M, DA) \rightarrow DM$, given by Tensor-Hom adjunction, for any $M\in A\m\cap R\proj$, while $\psi_{T, Q}$ is the canonical isomorphism given by $Q, T\in A\m\cap R\proj$ (see \citep[Proposition 2.2]{CRUZ2022410}).
	
	These isomorphisms also yield that 
	\begin{align}
		\Tor_i^B(\Hom_A(T, Q), \Hom_{R(A)}(\Hom_A(T, Q), DT) )\simeq \Tor_i^B(\Hom_A(T, Q), DQ), \ \forall i\in \mathbb{N}.
	\end{align}
The result now follows from Theorems \ref{dualmoduleMuellerpartone}, \ref{moduleMuellerpartone} and \ref{moduleMuellerparttwo}.
\end{proof}

\section{Going from bigger covers to a smaller covers}\label{Going from bigger covers to a smaller covers}

In this section, our aim is to explore whether the quality of a split quasi-hereditary cover is preserved by truncation by split heredity ideals and by Schur functors between split highest weight categories.

	\subsection{Truncation of split quasi-hereditary covers}\label{Truncation of split quasi-hereditary covers}

\begin{Theorem}\label{truncationcover}
	Let $(A, \{\Delta(\lambda)_{\lambda\in \Lambda}\})$ be a split quasi-hereditary Noetherian $R$-algebra. Assume that $(A, P)$ is an $i$-$\mathcal{F}(\Stsim)$ cover of $\End_A(P)^{op}$ for some integer $i\geq 0$. Let $J$ be a split heredity ideal of $A$. Then, $(A/J, P/JP)$ is an $i$-$\mathcal{F}(\Stsim_{A/J})$ cover of $\End_{A/J}(P/JP)^{op}$, where $\mathcal{F}(\Stsim_{A/J})=\mathcal{F}(\Stsim)\cap A/J\m$.
\end{Theorem}
\begin{proof}Denote by $B$ the endomorphism algebra $\End_A(P)^{op}$.
	The map $A\twoheadrightarrow A/J$ induces the fully faithful functor $A/J\m\rightarrow A\m$. Hence, $\End_{A/J}(P/JP)^{op}\simeq \End_A(P/JP)^{op}$. We wish to express $\End_{A/J}(P/JP)^{op}$ as a quotient of $B$. To see this, consider the exact sequence of $(A, B)$-bimodules
	\begin{align}
		0\rightarrow JP\rightarrow P\rightarrow P/JP\rightarrow 0. \label{eqfc108}
	\end{align}Applying $\Hom_A(P, -)$ yields the exact sequence 
	\begin{align}
		0\rightarrow \Hom_A(P, JP)\rightarrow B\rightarrow \Hom_A(P, P/JP)\rightarrow 0, \label{eqfc109}
	\end{align}
	while applying $\Hom_A(-, P/JP)$ yields the exact sequence
	\begin{align}
		0\rightarrow \End_A(P/JP)\rightarrow \Hom_A(P, P/JP)\rightarrow \Hom_A(JP, P/JP). \label{eqfc110}
	\end{align} Thanks to $J=J^2$ we have $\Hom_A(JP, X)=0$ for every $X\in A/J\m$. Combining (\ref{eqfc110}) with (\ref{eqfc109}) we obtain the exact sequence
	\begin{align}
		0\rightarrow \Hom_A(P, JP)\rightarrow B\rightarrow \End_{A/J}(P/JP)\rightarrow 0.
	\end{align}Since (\ref{eqfc108}) is exact as $(A, B)$-bimodules the latter is exact as $(B, B)$-bimodules. Denote by $B_J$ the endomorphism algebra $\End_{A/J}(P/JP)^{op}$. By the previous argument, the functor $B_J\m\rightarrow B\m$ is fully faithful. Denote by $G_{P/JP}$ the functor $$\Hom_{B_J}(\Hom_{A/J}(P/JP, A/J), -)=\Hom_B(\Hom_{A/J}(P/JP, A/J), -)\colon B_J\m\rightarrow A/J\m$$ and $F_{P/JP}=\Hom_{A/J}(P/JP, -)=\Hom_A(P/JP, -)\colon A/J\m\rightarrow B_J\m$. 
	
	To assert that the truncated cover is a $0$-$\mathcal{F}(\Stsim)$ cover, it is enough to compare the restrictions of the functors $F_P$ and $G_P\circ F_P$ to $\mathcal{F}(\Stsim)\cap A/J\m$ with the restriction of the functors $F_{P/JP}$ and $G_{P/JP}\circ F_{P/JP}$ to $\mathcal{F}(\Stsim_{A/J})$, respectively. For each $X\in A/J\m$, applying $\Hom_A(-, X)$ to (\ref{eqfc108}) instead of $\Hom_A(-, P/JP)$ yields that $F_{P/JP}X\simeq F_PX$. By applying $\Hom_B(-, F_PX)$ to $0\rightarrow F_PJ\rightarrow F_PA\rightarrow F_P(A/J)\rightarrow 0$ we obtain the exact sequence $0\rightarrow G_{P/JP}F_{P/JP} X\rightarrow G_PF_PX\rightarrow \Hom_B(F_PJ, F_PX)$. Fixing $X\in \mathcal{F}(\Stsim_{A/J})$ we obtain that  $\Hom_B(F_PJ, F_PX)\simeq \Hom_A(J, X)=0$ since $(A, P)$ is a $0$-$\mathcal{F}(\Stsim)$ cover of $B$. These isomorphisms are functorial, so if we denote by $ \eta^J$ the unit of the adjunction $F_{P/JP}\dashv G_{P/JP}$, then $\eta^J_X$ is an isomorphism for every $X\in \mathcal{F}(\Stsim_{A/J})$. This shows that $(A/J, P/JP)$ is a $0$-$\mathcal{F}(\Stsim_{A/J})$ cover of $B_J$.
	
	Our aim now is to compute $\R^jG_{P/JP}(F_{P/JP}X)$ for $j\leq i$ and every $X\in \mathcal{F}(\Stsim_{A/J})$. Hence, fix an arbitrary $X\in \mathcal{F}(\Stsim_{A/J})$. Applying $\Hom_B(-, F_PX)$ to (\ref{eqfc109}) we obtain $\Ext_B^1(B_J, F_PX)=0$ and $\Ext_B^l(B_J, F_PX)\simeq \Ext_B^{l-1}(F_PJP, F_PX)$ for every $l>1$.
	Observe that $JP\simeq J\otimes_A P$ as left $A$-modules since $P\in A\proj$. Moreover, $JP\in \add_A J$, and thus it is projective as left $A$-module. Thus, $\Ext_B^l(B_J, F_PX)\simeq \Ext_A^{l-1}(JP, X)=0$ for every $0<l-1\leq i$. Hence, $\Ext_B^l(B_J, F_PX)=0$ for every $1\leq l\leq i+1$.
	Let $\cdots\rightarrow P_i\rightarrow \cdots \rightarrow P_1\rightarrow P_0\rightarrow F_{P/JP}A/J\rightarrow 0$ be a projective $B_J$-resolution of $F_{P/JP}A/J$. Denote by $\Omega^{j+1}$ the kernel of $P_j\rightarrow P_{j-1}$, with $P_{-1}=\Omega^0=F_{P/JP}A/J$. Note that $\Ext_B^l(P_j, F_PX)=0$ for $1\leq l\leq i+1$. Taking into account that $B_J\m\rightarrow B\m$ is a fully faithful functor, applying $\Hom_B(-, F_PX)$ and $\Hom_{B_J}(-, F_PX)$ to the $B_J$ projective resolution of $F_PA/J$ yields \begin{align}
		\Ext_B^l(\Omega^j, F_PX)\simeq \Ext_B^{l+1}(\Omega^{j-1}, F_PX), \quad \Ext_{B_J}^s(\Omega^j, F_PX)\simeq \Ext_{B_J}^{s+1}(\Omega^{j-1}, F_PX),
	\end{align} for $1\leq l\leq i$,  $s, \ j\geq 1$ and the commutative diagram \begin{equation}
		\begin{tikzcd}
			\Hom_B(P_j, F_{P/JP}X)\arrow[r] \arrow[d, equal]& \Hom_B(\Omega^{j+1}, F_{P/JP}X) \arrow[r] \arrow[d, equal] & \Ext_B^1(\Omega^j, F_{P/JP}X)\arrow[r] & 0\\
			\Hom_{B_J}(P_j, F_{P/JP}X)\arrow[r] & \Hom_{B_J}(\Omega^{j+1}, F_{P/JP}X) \arrow[r] & \Ext_{B_J}^1(\Omega^j, F_{P/JP}X)\arrow[r] & 0
		\end{tikzcd}.
	\end{equation}  By the commutative diagram, $\Ext_B^1(\Omega^j, F_PX)$ is zero if and only if $\Ext_{B_J}^1(\Omega^j, F_PX)$ is zero. By assumption and the previous discussion, for each $1\leq l\leq i$, \begin{multline}
		0=\Ext_B^l(F_PA/J, F_PX)=\Ext^l_B(\Omega^0, F_PX)\simeq \Ext_B^1(\Omega^{l-1}, F_PX)\\=\Ext_{B_J}^1(\Omega^{l-1}, F_{P/JP}X)\simeq \Ext_{B_J}^l(F_{P/JP}A/J, F_{P/JP}X). 
	\end{multline}This concludes the proof.
\end{proof}

\begin{Remark}
	The module $P/JP$ might not be injective even if $P$ is.
\end{Remark}

\begin{Remark}
	It follows from the proof of Theorem \ref{truncationcover} that if $(A, P)$ is a cover of $B$ such that $(A/J, P/JP)$ is a $0$-$\mathcal{F}(\Stsim_{A/J})$ cover of $B_J$, then $(A, P)$ is a $(-1)$-$\mathcal{F}(\Stsim)$ cover of $B$.
\end{Remark}

This gives another reason to be interested in zero faithful split quasi-hereditary covers. These are exactly the covers for which double centralizer properties occur in every step of the split heredity chain. In particular, this gives another perspective on why zero faithful split quasi-hereditary covers possess so much nicer properties compared with $(-1)$-faithful split quasi-hereditary covers.

	\subsection{Relative codominant dimension with respect to a module in the image of a Schur functor preserving the highest weight structure}
Many split quasi-hereditary algebras can be written in the form $eAe$, for some bigger quasi-hereditary algebra $A$ in such a way that the quasi-hereditary structure of $eAe$ is inherited from the quasi-hereditary structure of $A$ under the Schur functor $\Hom_{A}(Ae, -)\colon A\m\rightarrow eAe\m$.  This is the case for Schur algebras $S_K(n, d)$ when $n<d$ (recall Theorem \ref{eAeqh}). In general, the Schur functor does not preserve dominant or codominant dimension. In this subsection, we give a relation between $eP\rcodomdim_{eAe}eT$ and the codominant dimension of a characteristic tilting module $T$ of $A$, when $P$ is a faithful projective-injective $A$-module.

\begin{Theorem}Let $k$ be a field and $(A, \{\Delta(\lambda)_{\lambda\in \Lambda}\})$ be a split quasi-hereditary algebra over $k$.
	Assume that there exists an idempotent $e$ of $A$ such that both $e$ and $A$ satisfy  the conditions of Theorem \ref{eAeqh}.
	Suppose that $P$ is a projective-injective faithful module. Let $M\in \mathcal{F}(\Cs)$. 
	If $\codomdim_{A} M \geq i$, then $eP\rcodomdim_{eAe} eM\geq i$ \label{lowervaluesrelawirthrespmodule} for $i\in \{1, 2\}$.
\end{Theorem}
\begin{proof}
	Denote by $B=\End_A(P)^{op}$ and by $C=\End_{eAe}(eP)^{op}$. Since $P$ is a (partial) tilting module the map given by multiplication by $e$, $B\rightarrow C$ is surjective according to Theorem \ref{eAeqh}. Thus, $C$ is a quotient of $B$. In particular, $C\m$ is a full subcategory of $B\m$. Again by Theorem \ref{eAeqh}, the map $\Hom_A(P, M)\rightarrow \Hom_{eAe}(eP, eM)$ is a surjective left $B$-homomorphism. Denote such a map by $\varphi_M$.  We can consider the following commutative diagram
	\begin{equation*}
		\begin{tikzcd}[scale cd=0.90, column sep=small, ]
			e\cdot (P\otimes_B \Hom_A(P, M))\arrow[d, "e\delta_{DM}"]  \arrow[r, equal] &	(eP)\otimes_B \Hom_A(P, M) \arrow[r, "(e\cdot P)\otimes_B \varphi_M", outer sep=0.75ex] & 	(eP)\otimes_B \Hom_{eAe}(eP, eM) \arrow[r, equal] &(eP)\otimes_C \Hom_{eAe}(eP, eM) \arrow[d, "\chi_{eM}"] 
			\\
			eM \arrow[rrr, equal] & & & eM
		\end{tikzcd}
	\end{equation*} with the composition of the upper rows being surjective (see also Remark \ref{counitchianddelta}).
	In fact, thanks to the $C\m$ being a full subcategory of $B\m$ we have the isomorphisms\begin{align}
		D((eP)\otimes_C \Hom_{eAe}(eP, eM))&\simeq \Hom_C(\Hom_{eAe}(eP, eM), D(eP))=\Hom_B(\Hom_{eAe}(eP, eM), D(eP))\nonumber\\&\simeq D((eP)\otimes_B \Hom_{eAe}(eP, eM)).
	\end{align}
	Since $\domdim_{A^{op}} DM=\codomdim_{A} M \geq 1 (\text{resp.} \ 2)$ if and only if $\delta_{DM}$ is surjective (resp. isomorphism) we obtain that $e\delta_{DM}$ is surjective if $i=1$ and bijective if $i=2$. So, if $i=1$ it follows that $\chi_{eM}$ is surjective, by the commutative diagram. Assume that $i=2$. Then, $(e\cdot P)\otimes_B \varphi_M$ must be injective, and so it is an isomorphism. This implies that $\chi_{eM}$ is also an isomorphism.
\end{proof}

For larger values of relative dominant dimension the most natural approach to consider is to see when the exact sequence giving the value of dominant dimension under the Schur functor $eA\otimes_A -$ gives information about the relative dominant dimension of $eAe$ with respect to $eP$. As we know, we can focus only in what happens over finite-dimensional algebras over a field.

\begin{Prop}
	 Let $(A, \{\Delta(\lambda)_{\lambda\in \Lambda}\})$ be a split quasi-hereditary algebra over a field $k$, $e$ a cosaturated idempotent of $A$ and $P$ a projective-injective $A$-module.  Suppose that there exists an exact sequence
	\begin{align}
		\delta\colon 0\rightarrow A\rightarrow P_0\rightarrow \cdots\rightarrow P_{n-1}, \quad \text{with } P_i\in \add_A P.
	\end{align}
	Then, the chain complex  $\Hom_{eAe}(e\delta, eP)$ is exact if and only if $P\in \add D (eA)$. In particular, if the chain complex $\Hom_{eAe}(e\delta, eP)$ is exact, then $eP$ is a projective-injective $eAe$-module.
\end{Prop}
\begin{proof}
	Assume that $P\in \add_A D(eA)$. Then, $eP\in \add_{eAe} D(eAe)$, that is, $eP$ is injective over $eAe$. It is clear that the functor $\Hom_{eAe}(-, eP)$ is exact. 
	
	Conversely, suppose that $e\delta$ remains exact under $\Hom_{eAe}(-, eP)$. 
	Let $X_0$ be the cokernel of $A\rightarrow P_0$.
	Consider the commutative diagram
	\begin{equation}
		\begin{tikzcd}[column sep=20pt,row sep=15pt]
			\Hom_A(P_1, P)\arrow[rr] \arrow[ddd, twoheadrightarrow] \arrow[rd, twoheadrightarrow] & & \Hom_A(P_0, P)\arrow[r, twoheadrightarrow] \arrow[ddd, twoheadrightarrow]& \Hom_A(A, P) \arrow[ddd, twoheadrightarrow]\\
			& \Hom_A(X_0, P)  \arrow[ur, hookrightarrow] \arrow[d] & & \\
			& \Hom_{eAe}(eX_0, eP) \arrow[dr, hookrightarrow] & &\\
			\Hom_{eAe}(eP_1, eP)\arrow[ur] \arrow[rr] & & \Hom_{eAe}(eP_0, eP) \arrow[r, twoheadrightarrow] & \Hom_{eAe}(eA, eP) 
		\end{tikzcd}\label{eqcocover58}
	\end{equation}The vertical maps are surjective maps due to Theorem \ref{eAeqh}.
	By assumption, the bottom row of (\ref{eqcocover58}) is exact. Hence, the lower triangle is a epi-mono factorization. Therefore, $\Hom_A(X_0, P)\rightarrow \Hom_{eAe}(eX_0, eP)$ is surjective. By Snake Lemma, we obtain that the map $\Hom_A(A, P)\rightarrow \Hom_{eAe}(eA, eP)$ is in addition to being surjective an injective map.
	Since $eA$ has a filtration by standard modules over $eAe$, $\Ext_{eAe}^{i>0}(eA, eP)$. By Lemma 2.10 of \cite{zbMATH06409569} for every $M\in A\m$,\begin{align}
		\Hom_A(M, P)\simeq \Hom_A(M, \Hom_{eAe}(eA, eP))\simeq \Hom_{eAe}(eM, eP).
	\end{align}By Theorem 3.10 of \cite{MR3123754}, this means that there exists an exact sequence \begin{align*}
	0\rightarrow P\rightarrow \Hom_{eAe}(eA, D(eAe))\simeq D(eA).
\end{align*} Since $P$ is injective, this exact sequence splits and we obtain that $P\in \add_A D(eA)$.
\end{proof}

For Schur algebras, this is only true in case $V^{\otimes d}$ is projective-injective module since it is a partial tilting module.
We can however give a lower bound to the relative dominant dimension with respect to $V^{\otimes d}$ based on its injective dimension. 

\begin{Cor}
	Let $k$ be a field and $A$ a finite-dimensional $k$-algebra. Let $Q\in A\m$ with $Q\in Q^\perp$. Suppose that $M\in {}^\perp Q$ and assume that there exists an $A$-exact sequence 
	\begin{align}
		0\rightarrow M\rightarrow Q_1\rightarrow \cdots \rightarrow Q_n,
	\end{align} with $Q_i\in \add Q$. Then, $Q\rdomdim_A M\geq n-\injdim_A Q$.
\end{Cor}
\begin{proof}Assume that $n>\injdim_A Q$, otherwise there is nothing to prove.
	Denote by $X_i$ the cokernel of $Q_{i-1}\rightarrow Q_i$ where by convention we consider $Q_0:=M$. By dimension shifting, \begin{align}
		\Ext_{A}^{i>0}(X_{n-\injdim_A Q}, Q)\simeq \Ext_{A}^{i+1>0}(X_{n-\injdim_A Q+1}, Q)\simeq \Ext_{A}^{i+\injdim_A Q>0}(X_n, Q)=0.
	\end{align}
	So, the exact sequence $ M\hookrightarrow Q_1\rightarrow \cdots \rightarrow Q_{n-\injdim_A Q}$ satisfies the assumptions of  Proposition \ref{cokernelrolereldomimodule}.
\end{proof}
%
%
%
%
Another option is to consider the homology over $\End_{eAe}(eP)$ by viewing it as a quotient of $\End_A(P)$.
\begin{Remark}Even for a field $K$, 
	the surjective map $\psi\colon KS_d\twoheadrightarrow \End_{S_K(n, d)}(V^{\otimes d})$ may not be a homological epimorphism if $n<d$. Indeed, by Proposition 2.2(a) of \citep{zbMATH05124581}, $\psi$ is a homological epimorphism if and only if $\ker \psi$ is an idempotent ideal and $\Tor_{i>0}^{KS_d}(\ker \psi, KS_d/\ker \psi)=0$. Fix $n=2$, $d=3$ and $K$ a field of characteristic three. Then, $\ker \psi$ is the ideal generated by $a:= e +(132)+(123)-(12)-(13)-(23)$. As $a^2=0$, $\ker \psi$ is not an idempotent ideal.\label{homologicalepipsi}
\end{Remark}

As it turns out, we do not need such assumption on the map $\End_A(P)^{op}\rightarrow \End_{eAe}(eP)^{op}$ to give lower bounds of codominant dimension with respect to $eP$ using the codominant dimension with respect to $P$. We can use, instead, the techniques of truncation of covers. This techniques are only fruitful for values of Hemmer-Nakano dimension greater than or equal to zero but this poses no problem in our situation since the lower cases can be treated using Theorem \ref{lowervaluesrelawirthrespmodule}. 

\begin{Theorem}\label{lowerboundrelativedominantdimensionV}
	Let  $(A, \{\Delta(\lambda)_{\lambda\in \Lambda}\})$ be a split quasi-hereditary algebra over a field $k$.
	Let $e$ be a cosaturated idempotent of $A$. If $A$ has positive dominant dimension with faithful projective-injective $Af$, then $$eAf\rcodomdim_{eAe}eT\geq  \codomdim_A T,$$ where $T$ is the characteristic tilting module of $A$.
\end{Theorem}
\begin{proof}We can assume without loss of generality that $A$ is a basic algebra and $Af$ is also basic.
	If $\codomdim_A T=1$, then the result follows from Theorem \ref{lowervaluesrelawirthrespmodule}. Assume that $\codomdim_A T\geq 2$. 
	
	By Theorem \ref{eAeqh}, $eT$ is the characteristic tilting module of $eAe$. Hence, the endomorphism algebra $\End_{eAe}(eT)^{op}$ is the Ringel dual of $eAe$ which we denote by $R(eAe)$. Also, by Theorem \ref{eAeqh} there exists an exact sequence $0\rightarrow X\rightarrow R(A)\rightarrow R(eAe)\rightarrow 0$ where $X$ is an ideal of the Ringel dual of $A$. More precisely, $X$ is the set of all endomorphisms $g\in \End_A(T)$ satisfying $eg=0$. Fix $P=Af$.  We claim that $X\Hom_A(T, P)$ is the kernel of the surjective map $\Hom_A(T, P)\rightarrow \Hom_{eAe}(eT, eP)$. Denote this surjection by $\psi$. Let $g\in X$ and $l\in \Hom_A(T, P)$ then $e(lg)=(el)(eg)=0$. So, it is clear that $X\Hom_A(T, P)\subset \ker \psi$. Now, let $l\in \Hom_A(T, P)$ such that $el=0$, that is, $l\in \ker \psi$. By assumption, we can write $i\circ \pi=\id_P$, where $\pi\in \Hom_A(T, P)$. So, $e(i\circ l)=ei\circ el=0$. This means that $i\circ l\in X$. Now $l=\pi\circ i\circ l=(i\circ l)\cdot \pi\in X\Hom_A(T, P)$. Now, a $k$-basis of $\End_A(T)$ can be constructed using its filtration by modules $\Hom_A(\St(\nu), \Cs(\nu))$, $\nu\in \L$ and the liftings of $\St(\l)\hookrightarrow T(\l)\twoheadrightarrow \Cs(\l)$ along these filtrations (see \citep[Proposition 5.5]{appendix} and \citep{Koenig2001}). In particular, these maps factor through $T(\l)$, $\l\in \L$. By assumption, $eS(\l)=0$ if and only if $\l<\mu$ for a fixed $\mu\in \L$, and so $eT(\l)=0$ if and only if $\l<\mu$. Analogously, $\End_{eAe}(eT)$ has a $k$-basis of the maps factoring through $eT(\l)\neq 0$. So, $X$ has a basis whose maps $T\rightarrow T$ factor through $T(\l)$, $\l<\mu$. Let $g_\l$ denote the idempotent $T\twoheadrightarrow T(\l)\hookrightarrow T$ and $g_e=\sum_{\l<\mu} g_\l$. Then, we showed that $X=R(A) g_e R(A)$. In particular, $X$ has a filtration by split heredity ideals of quotients of $R(A)$.  
	
	As $\codomdim_A T\geq 2$, Theorem \ref{ringeldualitycovers} implies that $(R(A), \Hom_A(T, P))$ is a $(\codomdim_A T-2)$-$\mathcal{F}(\St_{R(A)})$ cover of $\End_A(P)^{op}$. By induction on the filtration of $X$ by split heredity ideals and using Theorem \ref{truncationcover}, we obtain that
	\begin{align}
		(\End_{eAe}(eT), \Hom_{eAe}(eT, eP))\simeq (R(A)/X, \Hom_A(T, P)/X\Hom_A(T, P))
	\end{align}is a $(\codomdim_A T-2)$-$\mathcal{F}(\St_{R(eAe)})$ cover of $\End_{R(A)/X}(\Hom_A(T, P)/X\Hom_A(T, P))^{op}$ which is isomorphic to  $\End_{R(eAe)}(\Hom_{eAe}(eT, eP))^{op}\simeq \End_{eAe}(eP)^{op}$. By Corollary \ref{codominantdimensioncontrollfinitedimensionalcase}, we obtain the inequality $
	eP\rcodomdim_{eAe}eT\geq\codomdim_A T$.
\end{proof}

\section{Applications}

In this section, we will use the technology on relative dominant dimension with respect to a partial tilting module including \cite{CRUZ2022410,p2paper} to construct a split quasi-hereditary cover of certain quotients of Iwahori-Hecke algebras. This gives a new point of view to the Schur functors from module categories of $q$-Schur algebras studied in detail in \cite{p2paper}. This technology together with the results developed in \cite{p2paper} will allow us to give a new proof for the Ringel self-duality of the blocks of the BGG category $\mathcal{O}$ in Theorem \ref{RingelselfdualityofO}.

\subsection{Generalized Schur algebras in the sense of Donkin}\label{Generalized Schur algebras in the sense of Donkin}

In this subsection, we continue \cite{p2paper, CRUZ2022410} studying $q$-Schur algebras and Iwahori-Hecke algebras over commutative Noetherian rings using relative dominant dimensions.
Let $R$ be a commutative Noetherian ring with an invertible element $u$. Fix natural numbers $n, d$ and $q=u^{-2}$. The study of Iwahori-Hecke algebras can be traced back to \cite{zbMATH03218961}, and there are several equivalent ways to define them. The \textbf{Iwahori-Hecke algebra} $H_{R, q}(d)$ is the $R$-algebra with basis $\{T_\sigma\colon \sigma\in S_d \}$ (see \citep{zbMATH04193959}) satisfying the relations
\begin{align}
	T_\sigma T_s =\begin{cases}
		T_{\sigma s}, \quad &\text{ if } l(\sigma s)=l(\sigma) +1\\
		(u-u^{-1})T_\sigma +T_{\sigma s}, \quad &\text{ if } l(\sigma s)=l(\sigma)-1,
	\end{cases}\label{eqex24}
\end{align}
where $s\in S:=\{(1,2), (2,3), \cdots, (d-1, d) \}$ is a set of transpositions and $l(\sigma)$, $\sigma\in S_d$, is the minimum number of simple transpositions belonging to $S$ needed to write $\sigma$. In particular,  the elements $T_s$, $s\in S$, generate $H_{R, q}(d)$ as algebra. By (\ref{eqex24}), $H_{R, 1}(d)$ is exactly the group algebra of the symmetric group $RS_d$. Just like in case $q=1$ the Iwahori-Hecke algebras over general $q$ admit a base change property, we just have to replace the role of $\mathbb{Z}$ by the Laurent polynomial ring $\mathbb{Z}[X, X^{-1}]$:
\begin{align}
	H_{R, q}(d)\simeq R\otimes_{\mathbb{Z}[X, X^{-1}]} H_{\mathbb{Z}[X, X^{-1}], X^{-2}}(d). \label{heckechangering}
\end{align}
The ring $R$ is an $\mathbb{Z}[X, X^{-1}]$-algebra by defining the map $\mathbb{Z}[X, X^{-1}]\rightarrow R$ which sends $z\in \mathbb{Z}$ to $z1_R$ and $X$ to $u\in R$. 
Let $V$ be a free module of rank $n$ over $R$ and $\{e_1, \ldots, e_n\}$ a basis of $V$. So, the $R$-module $V^{\otimes d}$ is free with basis  $\{ e_i:=e_{i_1}\otimes\cdots \otimes e_{i_d}\ | \ i\in I(n, d)\}$, where $I(n, d)$ denotes the set of maps $i\colon \{1, \ldots, d\}\rightarrow \{1, \ldots, n\}$ and $i_a:=i(a)$, $a\in \{1, \ldots, d\}$. The symmetric group on $d$ letters $S_d$ acts on $I(n, d)$ by place permutation and so $V^{\otimes d}$ also admits a right $S_d$-action by place permutation. We can also regard $V^{\otimes d}$ as a right $H_{R, q}(d)$-module by considering the following deformation of place permutation, for every $i\in I(n, d)$, $s=(t, t+1)\in S$
\begin{align*}
	e_{i_1}\otimes\cdots \otimes e_{i_d}\cdot T_s = \begin{cases}
		e_{i_1}\otimes\cdots \otimes e_{i_d} \cdot s \quad &\text{ if } i_t<i_{t+1}\\
		u e_{i_1}\otimes\cdots \otimes e_{i_d} \quad &\text{ if } i_t=i_{t+1}\\
		(u-u^{-1}) e_{i_1}\otimes\cdots \otimes e_{i_d} + e_{i_1}\otimes\cdots \otimes e_{i_d}\cdot s \quad &\text{ if } i_t>i_{t+1}
	\end{cases}, \quad   1\leq t< d.\nonumber
\end{align*}
In particular, this action can be extended to an homomorphism of $R$-algebras $\psi\colon H_{R, q}(d)\rightarrow \End_R(V^{\otimes d})^{op}$. The subalgebra $\End_{H_{R, q}(d)}(V^{\otimes d})$ of the endomorphism algebra $\End_R(V^{\otimes d})$, which we denote by $S_{R, q}(n, d)$, is known as the \textbf{$q$-Schur algebra}. Many of their main properties can be found in \cite{zbMATH00014844, zbMATH04168928, MR1707336} and in \cite{zbMATH05080041} for $q=1$. Again, fixing $q=1$, we obtain the \textbf{classical Schur algebra} $S_R(n, d)=\End_{RS_d}(V^{\otimes d})$. In particular, $V^{\otimes d}$ inherits a left module structure over $S_{R, q}(n, d)$ given by the endomorphism algebra action.
By construction, the actions of $H_{R, q}(d)$ and $S_{R, q}(d)$ on $V^{\otimes d}$ commute and so the homomorphism of $R$-algebras $\psi\colon H_{R, q}(d)\rightarrow \End_{S_{R}(n, d)}(V^{\otimes d})^{op}$ is well defined. Quantum Schur--Weyl duality says that $\psi$ is a surjective map (see for example \citep[Theorem 6.2]{zbMATH01294605}). If $n\geq d$, $V^{\otimes d}$ is a faithful module over $H_{R, q}(d)$ and in such a case $\psi$ is actually an isomorphism. So this means that there are double centralizer properties:
\begin{align}
	&S_{R, q}(n, d)=\End_{H_{R, q}(d)}(V^{\otimes d}),  &H_{R, q}(d)\simeq \End_{S_{R, q}(n, d)}(V^{\otimes d})^{op}, \qquad &\text{if } n\geq d;\label{eq89}\\
	&S_{R, q}(n, d)=\End_{H_{R, q}(d)/\ker \psi}(V^{\otimes d}),  &H_{R, q}(d)/\ker \psi\simeq \End_{S_{R, q}(n, d)}(V^{\otimes d})^{op}, \qquad &\text{if } n< d;
\end{align}The double centralizer property in (\ref{eq89}) is actually a consequence of $(S_{R, q}(n, d), V^{\otimes d})$ being  a cover of $H_{R, q}(d)$ (see for example \citep[Theorem 7.20]{CRUZ2022410}, \citep[Proposition 2.4.4]{p2paper}). But, if $n<d$, then $V^{\otimes d}$ is not necessarily projective, and so the former pair cannot be a cover anymore in general. So we might wonder if the case $n<d$ can also be explained using cover theory.

Our aim now is to unify these double centralizer properties and their cohomological  higher versions  using relative dominant dimension and cover theory generalising the approach taken in \cite{Koenig2001}. 
In our treatment, we will take care of the quantum case and the classical case simultaneously in both cases having the integral setup in mind.

\subsubsection{Background on $q$-Schur algebras}
A \textbf{weight} of an element $i\in I(n, d)$ is the composition $\omega(i)=(\omega_1, \ldots, \omega_n)$ of $d$ in at most $n$ parts where each $\omega_j$ is the number $|\{l\in \{1, \ldots, d\}\colon i_l=j \}|$, $j=1, \ldots, n$. We denote by $\L(n, d)$ the set of all weights associated with $I(n, d)$.

The action of $S_d$ in $I(n, d)$ can be extended diagonally to a right action of $S_d$ in $I(n, d)\times I(n, d)$. 
 Each $S_d$-orbit of $I(n, d)\times I(n, d)$ has a representative $(i, j)$ satisfying $(i_1, j_1)\leq \cdots \leq (i_d, j_d)$, where the elements in $I(n, d)\times I(n, d)$ are ordered according to the lexicographical order.
Denoting by $\{e_1^*, \ldots, e_n^*\}$ the dual basis of $\{e_1, \ldots, e_n\}$ over $R$ we obtain that $V^{\otimes^d}\otimes_{H_{R, q}(d)} DV^{\otimes^d}$ is a free $R$-module with basis 	$$\{	e_{i_1}\otimes\cdots \otimes e_{i_d}\otimes_{H_{R, q}(d)} e_{j_1}^*\otimes \cdots \otimes e_{j_d}^*\colon  i, j\in I(n, d), \ (i_1, j_1)\leq \cdots \leq (i_d, j_d) \}$$
(see \citep[Proposition 7.16]{CRUZ2022410}). This fact can be used to deduce that $q$-Schur algebras admit a base change property:
 $$S_{R, q}(n, d)\simeq R\otimes_{\mathbb{Z}[X, X^{-1}]} S_{\mathbb{Z}[X, X^{-1}], X^{-2}}(n, d).$$
It can also be used to construct a basis of $S_{R, q}(n, d)$.  In particular, for each $\omega(i)\in \L(n, d)$, the image of $e_{i_1}\otimes\cdots \otimes e_{i_d}\otimes_{H_{R, q}(d)} e_{i_1}^*\otimes \cdots \otimes e_{i_d}^*$ in $D(V^{\otimes^d}\otimes_{H_{R, q}(d)} DV^{\otimes^d})\simeq S_{R, q}(n, d)$, which we denote by $\xi_{\omega(i)}$, is an idempotent of $S_{R, q}(n, d)$.

The $q$-Schur algebra admits a split quasi-hereditary structure whose standard modules are called the $q$-Weyl modules, and are indexed by the partitions of $d$ in at most $n$ parts ordered by the dominance order (see for example \citep[Theorem 11.5.2]{zbMATH04193959} and \citep[Theorem 7.2.2]{p2paper}).  We denote the standard modules of $S_{R, q}(n, d)$ by $\St_{S_{R, q}(n, d)}(\l)$, $\l\in \L^+(n, d)$, or just by $\St(\l)$ when there are no ambiguities. By $\L^+(n, d)$ we mean the set of all partitions of $d$ in at most $n$ parts. It admits a cellular structure with the standard modules being the cell modules (ordered with the opposite of the dominance order). Furthermore, for each $\mi\in \MaxSpec R$, the $q_\mi$-Schur algebra $S_{R(\mi), q_\mi}(n, d)$ (with $n\geq d$) belongs to the class $\mathcal{A}$ defined in \cite{zbMATH05871076}, where $q_\mi$ denotes the image of $q$ in $R/\mi$. As a consequence, in the integral setup, the following result holds.

\begin{Theorem}\label{dominantdimensionquantumschur}
	Let $R$ be a commutative Noetherian ring with invertible element $u\in R$. Put $q=u^{-2}$ and assume that $n\geq d$. Let $T$ be a characteristic tilting module of $S_{R, q}(n, d)$. Then, the following assertions hold.
	\begin{enumerate}[(a)]
		\item $(S_{R, q}(n, d), V^{\otimes d})$ is a relative gendo-symmetric $R$-algebra;
		\item $\domdim (S_{R, q}(n, d), R)=2 \domdim_{(S_{R, q}(n, d), R)} T$;
		\item 	$\domdim_{(S_{R, q}(n, d), R)} T=\inf \{ s\in \mathbb{N} \ | \ 1+q+\cdots+ q^s\notin R^\times, \ s<d \}$.
	\end{enumerate}
\end{Theorem}
\begin{proof}
	See \citep[Theorem 7.20]{CRUZ2022410} and \citep[Corollary 7.2.4]{p2paper}. 
\end{proof}

\begin{Remark}\label{domcodomSchur}
Let $R$ be a commutative Noetherian ring with invertible element $u\in R$ and fix $q=u^{-2}$.	For any natural numbers $n, d$ and any module $M\in \add_{S_{R, q}(n, d)}T$, we have  $M\rdomdim_{(S_{R, q}(n, d), R)} T= M\rcodomdim_{(S_{R, q}(n, d), R)}T$. Indeed, this follows from applying Theorem \ref{changeofringrelativedomdimrelativetomodule} (and its dual) with Proposition \ref{dominantandcodominantoftiltingcomparison} to the $q_\mi$-Schur algebras $S_{R(\mi), q_\mi}(n, d)$, $\mi\in \MaxSpec R$.
\end{Remark}

\paragraph{Covers of Iwahori-Hecke algebras}
It follows by Theorem \ref{dominantdimensionquantumschur} that $(S_{R, q}(n, d), V^{\otimes d})$ is a split quasi-hereditary cover of $H_{R, q}(d)$. The quality of the Schur functor $$F_{V^{\otimes d}}=\Hom_{S_{R, q}(n, d)}(V^{\otimes d}, -)\colon S_{R, q}(n, d)\m\rightarrow H_{R, q}(d)\m$$ was completely determined for all local commutative rings $R$ in \citep[Subsection 7.2.1]{p2paper}. The field case is due to \citep{zbMATH07050778}. In the general case, the quality depends on whether the ground ring is partial divisible or not. For more details, we refer to \cite{p2paper}.

\subsubsection{Covers of quotients of Iwahori-Hecke algebras}

To be in the conditions of Theorem \ref{lowerboundrelativedominantdimensionV} to deduce a lower bound for $V^{\otimes d}\rdomdim_{S_{R, q}(n, d)} T$, for some characteristic tilting module $T$ we need the following result.
Although most statements of the following result hold over arbitrary commutative rings, for our purposes it is enough to state the classical versions.

\begin{Theorem}		 Let $K$ be a field. Put $q=u^{-2}$ for some $u\in K$. Let $n, d$ be natural numbers so that $d>n$. Define $$\L(d, d)^n:=\{\beta\in \L(d, d)\colon \beta_{n+1}=\cdots=\beta_d=0 \}.$$
  Define the idempotent $\displaystyle f=\sum_{\beta\in \L(d, d)^n} \xi_\beta\in S_{K, q}(n, d)$.\label{schuralgebraspassageNton} Then, the following assertions hold.
  \begin{enumerate}[(a)]
  	\item The $K$-algebra $fS_{K, q}(d, d)f$ is isomorphic to $S_{K, q}(n, d)$;
  	\item $f(K^d)^{\otimes d}\simeq (K^n)^{\otimes d}$ as $S_{K, q}(n, d)$-modules;
  	\item The idempotent $f$ satisfies the hypothesis of Theorem \ref{eAeqh};
  	\item If $R$ is a commutative Noetherian ring with invertible element $u\in R$, then $V^{\otimes d}=(R^n)^{\otimes d}$ is a partial tilting module of $S_{R, q}(n, d)$, where $q=u^{-2}$;
  	\item The (partial) tilting indecomposable modules of $S_K(n, d)$ are the image of the (partial) tilting indecomposable modules of $S_K(d, d)$ under the Schur functor $${\Hom_{S_{K, q}(d, d)}(S_{K, q}(d, d)f, -)\colon S_{K, q}(d, d)\m \rightarrow S_{K, q}(n, d)\m}.$$
  \end{enumerate}
\end{Theorem}
\begin{proof}Consider the injective map $\Upsilon\colon \L(n, d)\rightarrow \L(d, d)$, given by $\alpha\mapsto (\alpha_1, \cdots, \alpha_n, 0, \cdots, 0)$. The image of $\Upsilon$ is exactly $\L(d, d)^n$. 
	For the case $q=1$ in (a) see for example \citep[6.5]{zbMATH05080041}, the argument also holds for arbitrary commutative rings since for any $i, j\in I(d, d)$ and the weight $\l=(1, \ldots, 1)$,
		\begin{align}
		f\xi_{i, j}\xi_\l=\begin{cases}
			\xi_{i, j}, \text{ if } i\in I(n, d) \text{ and } \omega(j)=\l, \\
			0, \text{ otherwise}.
		\end{cases}, \quad \xi_{i, j}(e_1\otimes\cdots\otimes e_d)=e_{i_1}\otimes\cdots\otimes e_{i_d},
	\end{align}where $\omega(j)$ denotes the weight of $j$.
For the general case in (a), see \citep[2.2(1)]{MR1707336}. For (b) see \citep[4.7]{MR1707336}. For (c), (d), (e) with $q=1$ and $R=K$ see  \citep[3.9, 4.2]{zbMATH00681964}. For (c) in the general case, see \citep[4.7, Proposition A3.11(i)]{MR1707336}.
By Theorem \ref{eAeqh}, (e) holds. Since $(K^d)^{\otimes d}$ is a projective-injective of $S_{K, q}(d, d)$,  Theorem \ref{eAeqh} and (b) yields that $(K^n)^{\otimes d}$ is a partial tilting module of $S_{K, q}(n, d)$. 
\mbox{By Proposition \ref{standardscotiltingsreductiontofields}, (d) follows.}
\end{proof}

\begin{Theorem}\label{domdimTcasensmallerd}
	Let $R$ be a commutative Noetherian ring with invertible element $u\in R$ and $n, d$ be natural numbers. Put $q=u^{-2}$.  Let $T$ be a characteristic tilting module of $S_R(n, d)$. Then,
	$$V^{\otimes d}\rdomdim_{(S_{R, q}(n, d), R)} T\geq  \inf \{ s\in \mathbb{N} \ | \ 1+q+\cdots+ q^s\notin R^\times, \ s<d \}\geq 1.$$
\end{Theorem}
\begin{proof}
	By Theorem \ref{changeofringrelativedomdimrelativetomodule} and Propositions \ref{standardscotiltingsreductiontofields} and \ref{qhproperties}, 
	\begin{align}
		V^{\otimes d}\rdomdim_{(S_{R, q}(n, d), R)} T= \inf\{(R(\mi)^n)^{\otimes d}\rdomdim_{(S_{R(\mi), q_\mi}(n, d)} T(\mi)| \mi\in \MaxSpec R \},
	\end{align}where $q_\mi$ is the image of $q$ in $R(\mi)$. By Theorems \ref{lowerboundrelativedominantdimensionV} and \ref{schuralgebraspassageNton}, \begin{align}
	V^{\otimes d}\rdomdim_{(S_{R, q}(n, d), R)} T&\geq \inf\{(R(\mi)^d)^{\otimes d}\rdomdim_{(S_{R(\mi), q_\mi}(d, d)} T_\mi| \mi\in \MaxSpec R \}\\&= \inf\{\domdim_{(S_{R(\mi), q_\mi}(d, d)} T_\mi| \mi\in \MaxSpec R \} \\ & = \inf \{\domdim_{S_{R(\mi), q_\mi}(d, d)} T(\mi)| \mi\in \MaxSpec R \},
\end{align}where $T_m$ is the characteristic tilting module of $(S_{R, q}(n, d), R)$ and $\add T(\mi)=\add T_\mi$, $\mi\in \MaxSpec R$.  By Theorem 6.13 of \cite{CRUZ2022410} and Corollary 7.1.3 of \cite{p2paper}, it follows that  
	\begin{align}
		V^{\otimes d}\rdomdim_{(S_{R, q}(n, d), R)} T\geq  	\inf \{ s\in \mathbb{N} \ | \ 1+q+\cdots+ q^s\notin R^\times, \ s<d \}. \tag*{\qedhere}
	\end{align}
\end{proof}
This inequality is sharp in general since this becomes an equality in case $n\geq d$ (see \citep[Corollary 7.1.3]{p2paper}).  
We are now ready to prove one of our main results. 
\begin{Theorem}\label{maintheoremforRingeldualSchur}
		Let $R$ be a commutative Noetherian regular ring with invertible element $u\in R$ and $n, d$ be natural numbers. Put $q=u^{-2}$. Let $T$ be a characteristic tilting module of $S_{R,q}(n,d)$.  Denote by $R(S_{R, q}(n, d))$ the Ringel dual of the $q$-Schur algebra $S_{R, q}(n, d)$ (there are no restrictions on the natural numbers $n$ and $d$) $\End_{S_{R,q}(n,d)}(T)^{op}$. 
		
		Then, $(R(S_{R, q}(n, d)), \Hom_{S_{R, q}(n, d)}(T, V^{\otimes d}))$ is a {$(V^{\otimes d}\rdomdim_{(S_{R, q}(n, d), R)} T-2)$-$\mathcal{F}(\Stsim_{R(S_{R, q}(n, d))})$} split quasi-hereditary cover of $\End_{S_{R, q}(n, d)}(V^{\otimes d})^{op}$.
\end{Theorem}
\begin{proof}
If ${V^{\otimes d}\rdomdim_{(S_{R, q}(n, d), R)} T}\geq 2$ the result	follows from Theorems \ref{domdimTcasensmallerd}, \ref{ringeldualitycovers} and Proposition \ref{dominantandcodominantoftiltingcomparison}.
Assume now that  $V^{\otimes d}\rdomdim_{(S_{R, q}(n, d), R)} T=1$.  Theorem \ref{domdimTcasensmallerd} implies that $1+q$ is not invertible and $d>1$.
So we need to verify, in particular, the case $1+q=0$ and $d>1$. 

Let $L$ be a field with an element $u\in L$ satisfying $u^{-2}=-1$.  Consider the Laurent polynomial ring $R:=L[X, X^{-1}]$. The invertible elements of $L[X, X^{-1}]$ are of the form $lX^z$, $z\in \mathbb{Z}$ and $l\in L$ and $1+X^{-2}\neq 0$. So,  $\inf \{ s\in \mathbb{N} \ | \ 1+q+\cdots+ q^s\notin L[X, X^{-1}]^\times, \ s<d \}=1$ with $q=X^{-2}$.  Let $Q$ be the quotient field of $L[X, X^{-1}]$. Then, $\inf \{ s\in \mathbb{N} \ | \ 1+q+\cdots+ q^s\neq 0\in Q, \ s<d \}\geq 2$. Therefore,
${(Q^n)^{\otimes d}\rdomdim_{S_{Q, q}(n, d), Q} Q\otimes_{L[X, X^{-1}]} T\geq 2}$ and  $V^{\otimes d}\rdomdim_{(S_{R, q}(n, d), R)} T\geq 1$ for a characteristic tilting module $T$ of $S_{R, q}(n, d)$. By Corollary \ref{deformationringeldualfunctor}, we obtain that $(R(S_{R, q}(n, d)), \Hom_{S_{R, q}(n, d)}(T, V^{\otimes d}))$ is a $0$-$\mathcal{F}(\Stsim_{R(S_{R, q}(n, d))})$ split quasi-hereditary cover of $\End_{S_{R, q}(n, d)}(V^{\otimes d})^{op}$. 
Observe that $L$ is the residue field $L[X, X^{-1}]/(X-u^{-1})$ and the image of $X$ in this residue field is $u^{-1}$. Applying \linebreak ${L[X, X^{-1}]/(X-u^{-1})\otimes_{L[X, X^{-1}]}} -$ to the previous cover we obtain by Theorem 5.0.7 of \citep{p2paper} that $(R(S_{L, -1}(n, d)), \Hom_{S_{L, -1}(n, d)}(T, V^{\otimes d}))$ is a $(-1)$-$\mathcal{F}(\Stsim_{R(S_{L, -1}(n, d))})$ split quasi-hereditary cover of \linebreak$\End_{S_{L, -1}(n, d)}(V^{\otimes d})^{op}$.  By Proposition \ref{faithfulcoverresiduefield}, we conclude that the result holds also for the cases satisfying $V^{\otimes d}\rdomdim_{(S_{R, q}(n, d), R)} T=1$.
\end{proof}

The special case $q=1$ puts the Schur--Weyl duality between $S_R(n,d)$ and $RS_d$ without restrictions on the parameters $n$ and $d$ into our context. Moreover, this generalizes the results of Hemmer and Nakano in \citep{hemmer_nakano_2004} and \citep[Theorem 3.9]{zbMATH05871076} on the Schur algebra ($S_R(n, d)$ with parameters $n\geq d$). In general, this formulation helps us to put the quantum Schur--Weyl duality between $q$-Schur algebras and Iwahori-Hecke algebras on $V^{\otimes d}$ completely in the setup of cover theory generalising the results of \citep{zbMATH07050778, zbMATH02182639}.

In fact, if $n\geq d$ and $k$ is a field with $u\in k$, then by  (\citep[Proposition 3.7]{zbMATH00549737} for $q=1$, \citep[Proposition 4.1.4, Proposition 4.1.5]{MR1707336}) the Ringel dual of $S_{k, q}(n, d)$ is the opposite algebra of $S_{k, q}(n, d)$ and we can identify the projective relative injective module $\Hom_{S_{R, q}(n, d)}(T, V^{\otimes d})$ with $DV^{\otimes d}$. In this case, quantum Schur--Weyl duality gives $\End_{S_{R, q}(n, d)}(V^{\otimes d})\simeq H_{R, q}(d)^{op}$. So, for $n\geq d$, Theorem \ref{maintheoremforRingeldualSchur} is translated to $(S_{R, q}(n, d)^{op}, DV^{\otimes d})$ being  a split quasi-hereditary cover of $H_{R, q}(d)^{op}$. Therefore, for $n\geq d$, this statement is nothing new and since $S_R(n, d)$ is relative gendo-symmetric the quality of this cover coincides with the cover $(S_{R, q}(n, d), V^{\otimes d})$. The novelty lies in the case $n<d$. 

On the other hand, this formulation gives us better insights into older results present in the literature and it explains, in some sense, why different conventions about the Schur functor can be considered.
For simplicity, let us assume that $q=1$.

 Let $H\colon S_R(n, d)\m\rightarrow R(S_R(n, d))$ be the equivalence of categories given by Ringel self-duality in \citep[Proposition 3.7]{zbMATH00549737}. Denote by $F'$ the Schur functor associated with the cover  \linebreak$(R(S_R(n, d)), \Hom_{S_R(n, d)}(T, V^{\otimes d}))$ and by $F$ the Schur functor associated with the cover $(S_{R, q}(n, d), V^{\otimes d})$.
Then, $F\St(\l)=\xi_{(1, \ldots, d), (1, \ldots, d)}\St(\l)$ and by \citep[Lemma 1.6.12]{zbMATH00971625}, \begin{align}
	F'H\St(\l)&\simeq F'\Hom_{S_R(n, d)}(T, \Cs(\l'))\simeq F\Cs(\l')= \xi_{(1, \ldots, d), (1, \ldots, d)} D\St(\l') ^{\iota} \\&\simeq D(\xi_{(1, \ldots, d), (1, \ldots, d)} \St(\l'))^{\xi_{(1, \ldots, d), (1, \ldots, d)}\iota\xi_{(1, \ldots, d), (1, \ldots, d)}}\simeq \sgn\otimes_K \xi_{(1, \ldots, d), (1, \ldots, d)}\St(\l).
\end{align}Here $\sgn$ is the free module $R$ with the $S_d$-action $\sigma\cdot 1_R=\sgn(\sigma)1_R$ and $M^\iota$ is the right module $M$ with right action $m\cdot a=\iota(a)m$, $m\in M$ and $a\in S_R(n, d)$. The same notation is used for modules over $RS_d$. Thus,
there exists a commutative diagram
\begin{equation*}
	\begin{tikzcd}
		S_R(n, d)\m \arrow[rr, "F"] \arrow[d, "H"]&&RS_d\m \arrow[d, "\sgn\otimes_R -"]\\
		R(S_R(n, d))\m \arrow[rr, "F'"] && RS_d\m
	\end{tikzcd}.
\end{equation*}
This means that, the cover $(R(S_R(n, d)), \Hom_{S_R(n, d)}(T, V^{\otimes d}))$ is equivalent to $(S_R(n, d), V^{\otimes d})$ if $n\geq d$.

In case $n<d$, the Ringel dual of $S_R(n, d)$ is no longer, in general, a Schur algebra; it is instead a generalized Schur algebra in the sense of Donkin.
The construction of the Ringel dual of $S_R(n, d)$ is as follows: let $U_{\mathbb{Z}}$ be the Konstant $\mathbb{Z}$-form of the enveloping algebra of the semi-simple complex Lie algebra $\mathfrak{sl}_d(\mathbb{C})$. That is, $U_{\mathbb{Z}}$ is the subring of the enveloping algebra of  $\mathfrak{sl}_d(\mathbb{C})$ generated by the elements
$$\frac{e_{i, j}^m}{m!}, \quad 1\leq i\neq j\leq d, \ m\geq 0,$$
where $e_{i, j}$, $1\leq i\neq  j\leq d$ denote the generators of the enveloping algebra of $\mathfrak{sl}_d(\mathbb{C})$. Then, the Ringel dual of $S_{\mathbb{Z}}(n, d)$ is the free Noetherian $\mathbb{Z}$-algebra $U_{\mathbb{Z}}/I_{\mathbb{Z}}$, where $I_{\mathbb{Z}}$ is the largest ideal of $U_{\mathbb{Z}}$ so that the simple modules of $\mathbb{Q}\otimes_{\mathbb{Z}} U_{\mathbb{Z}}/I_{\mathbb{Z}}$ are isomorphic to the Weyl modules indexed by the weights belonging to $\L^+(n, d)$ (see \citep[3.1]{MR866778} and \citep[Proposition 3.11]{zbMATH00549737}). For an arbitrary commutative ring $R$, $R\otimes_{\mathbb{Z}}U_{\mathbb{Z}}/I_{\mathbb{Z}}$ is the Ringel dual of $S_R(n, d)$ known as \textbf{generalized Schur algebra} associated with $\mathfrak{sl}_d$ and  the set $\L^+(n, d)$.
Since the Ringel dual of the Schur algebra is a quotient of $R\otimes_{\mathbb{Z}} U_\mathbb{Z}$, Theorem \ref{maintheoremforRingeldualSchur} suggests why Schur--Weyl duality between $S_R(n, d)$ and $RS_d$ can be deduced by studying the action of the Konstant $\mathbb{Z}$-form on $V^{\otimes d}$.

\begin{Remark}
	From Theorem \ref{maintheoremforRingeldualSchur}, it follows that the basic algebra of $\End_{S_{R, q}(n, d)}(V^{\otimes d})^{op}(\mi)$, $A_\mi$, has a cellular structure with cell modules of $H\Hom_{S_{R, q}(n, d)}(V^{\otimes d}, \Cs(\l))(\mi)$, $\l\in \L$, for every $\mi\in \MaxSpec R$, (see \citep[Theorem 7.5]{appendix}) where $H\colon \End_{S_{R, q}(n, d)}(V^{\otimes d})^{op}(\mi)\m\rightarrow A_\mi\m$ is an equivalence of categories. If $q=1$, using the fact that all Schur algebras can be recovered using change of rings from the integral Schur algebra $S_{\mathbb{Z}}(n, d)$ \citep[Proposition 7.6]{appendix} implies that $\End_{S_{R}(n, d)}(V^{\otimes d})^{op}_\mi$ is Morita equivalent to a Noetherian algebra which has a cellular structure with cell modules labelled by $\l\in \L$ for all $\mi\in \MaxSpec R$.
\end{Remark}

Let $K$ be a field.  The existence of the quasi-hereditary cover described in Theorem \ref{maintheoremforRingeldualSchur} allows to study the multiplicities of simple $\End_{S_K(n, d)}(V^{\otimes d})$-modules through the multiplicities of simple $R(S_K(n, d))$-modules. In particular, this explains the background for the techniques used in \citep{zbMATH00681964} to determine decomposition numbers in the symmetric group. For example, \citep[4.5]{zbMATH00681964} can be deduced using the Schur functor constructed in Theorem \ref{maintheoremforRingeldualSchur}, the Ringel duality functor and BGG reciprocity.

If $R$ is a field, the value of the cover in Theorem \ref{maintheoremforRingeldualSchur} is optimal. But, as we saw even for the case $n\geq d$ the situation can be improved in some cases. We leave as open question to determine the relative dominant dimension $(V^{\otimes d}\rdomdim_{(S_{R, q}(n, d), R)} T$, when $n<d$ and in particular the optimal value of the cover in Theorem \ref{maintheoremforRingeldualSchur}.

\subsection{Relative dominant dimension as a tool for Ringel self-duality}

We are now going to use the results of Section \ref{Ringel duality} together with integral representation theory to reprove Ringel self-duality for two classical cases.
The method is particularly effective when applied to the Bernstein-Gelfand-Gelfand category $\mathcal{O}$. 
Such situations highlight the strength of working with split quasi-hereditary covers whose exact categories of modules admit a filtration by summands of direct sums of standard modules with large Hemmer-Nakano dimension.
\subsubsection{BGG category $\mathcal{O}$} \label{BGG category O}
We consider now the Bernstein-Gelfand-Gelfand category $\mathcal{O}$ introduced in \citep{zbMATH03549206}. 
In \cite{p2paper} using an integral version of the BGG category $\mathcal{O}$ introduced in \cite{zbMATH03747378} the author introduced split quasi-hereditary algebras whose module categories are deformations of the blocks of the category $\mathcal{O}$. Our aim now is to see that these objects are Ringel self-dual. We will continue using the same notation as in \cite{p2paper}. 

\paragraph{Some terminology}We assume throughout this subsubsection that $R$ is a commutative regular Noetherian local ring which is a $Q$-algebra with maximal ideal $\mi$, unless stated otherwise.
Let $\mathfrak{g}$ be a complex semi-simple Lie algebra with Cartan subalgebra $\mathfrak{h}$. The Lie algebra $\mathfrak{g}$ admits a Cartan decomposition $\mathfrak{g}=\mathfrak{n}^+\oplus \mathfrak{h}\oplus \mathfrak{n}^-$. The Lie algebra $\mathfrak{b}:= \mathfrak{n}^+\oplus \mathfrak{h}$ is called a Borel subalgebra of $\mathfrak{g}$. Let $\Phi\subset \mathfrak{h}^*$ be the set of non-zero roots associated with the Cartan decomposition, and let $\Pi$ be a set of simple roots, in particular, it is a basis of the root system $\Phi$. Here, $\mathfrak{h}^*$ denotes the dual of $\mathfrak{h}$. For each $\alpha\in \Phi$, we denote by $\alpha^{\vee}$ the coroot associated with $\alpha$ with respect to the symmetric bilinear form on the real span of $\Phi$ induced by the Killing form of $\mathfrak{g}$. Denote by $\Phi^\vee$ the set of all coroots. 

Using a Chevalley basis of $\mathfrak{g}$ we can consider an integral version $\mathfrak{g}_{\mathbb{Z}}$ of $\mathfrak{g}$ so that $\mathbb{C}\otimes_{\mathbb{Z}} \mathfrak{g}_{\mathbb{Z}}\simeq \mathfrak{g}$ as Lie algebras. In the same way, there are integral versions for $\mathfrak{h}$, $\mathfrak{n}^+$, $\mathfrak{n}^-$, etc. These Lie algebras are also free of finite rank over $\mathbb{Z}$.   We set $\mathfrak{g}_R:=R\otimes_\mathbb{Z} \mathfrak{g}_{\mathbb{Z}}$ and similarly we define $\mathfrak{h}_R$, $\mathfrak{h}_R^*$, etc. We denote by $U(\mathfrak{g}_R)$ (resp. $U(\mathfrak{n}_R^+))$ the enveloping algebra of $\mathfrak{g}_R$ (resp. $\mathfrak{n}_R^+$).  For each $M\in U(\mathfrak{g}_R)\m$ and each $\l\in \mathfrak{h}_R^*$, the \textbf{weight space} $M_{\l}$ is the $R$-module $\{m\in M\colon h\cdot m=\l(h)m, \forall h\in \mathfrak{h}_R \}$. The element $\l\in \mathfrak{h}_R^*$ is called a \textbf{weight} of $M$ if $M_\l\neq 0$. Denote by $W$ the \textbf{Weyl group} generated by the reflections, $s_\alpha$, $\alpha\in \Phi$, associated with the root system $\Phi$. In particular, $W$ acts on $\mathfrak{h}^*$ and consequently it acts on $\mathfrak{h}_R^*$.

By the set of integral weights $\L_R$ we mean the set $\{\l\in \mathfrak{h}_R^*\colon \langle \l, \alpha^\vee\rangle_R \in \mathbb{Z}, \ \forall \alpha \in \Phi \}$. Here, $\langle -, -\rangle_R\colon \mathfrak{h}_R^*\times \Phi^\vee\rightarrow R$ denotes the extension of $\Phi\times \Phi^\vee \rightarrow \mathbb{Z}$ to $\mathfrak{h}_R^*\times \Phi^\vee\rightarrow R$.

\begin{Def}\label{BGGcategoryO}
	The \textbf{BGG category $\mathcal{O}$} of a semi-simple Lie algebra $\mathfrak{g}_K$ over a splitting field of characteristic zero $K$ is the full subcategory of $U(\mathfrak{g}_K)\M$ whose modules $M$ satisfy the following conditions:
	\begin{enumerate}[(i)]
		\item $M\in U(\mathfrak{g}_K)\m$;
		\item $M$ is semi-simple over $\mathfrak{h}$, that is, $M=\bigoplus_{\l\in \mathfrak{h}_K^*} M_\l$;
		\item $M$ is locally $\mathfrak{n}_K^+$-finite, that is, for each $m\in M$ the subspace $U(\mathfrak{n}_K^+)m$ of $M$ is finite-dimensional.
	\end{enumerate}
\end{Def} 

\paragraph{Some facts about the BGG category $\mathcal{O}$}In \cite{zbMATH03549206, zbMATH05309234, zbMATH03663362}, many results and properties of the BGG category $\mathcal{O}$ can be found. This category contains all finite-dimensional modules over $U(\mathfrak{g}_K)$ but also nice modules like the Verma modules $\St(\mu)$, $\mu\in \mathfrak{h}_K^*$.  These are indecomposable modules. The simple modules of $\mathcal{O}$  are exactly the unique simple quotients of the Verma modules. Denote by $L(\l)$ the top of $\St(\l)$. This category has enough projectives and it can be brought to the study of finite-dimensional algebras using its blocks.

In fact, the BGG category $\mathcal{O}$ decomposes into a direct sum of full subcategories $\mathcal{O}_{W_\l\cdot \l}$, indexed by antidominant weights $\l\in \mathfrak{h}_K^*$, consisting of modules whose composition factors are the simple modules $L(\mu)$ with $\mu \in W_\l\cdot \l$.  The subcategories $\mathcal{O}_{W_\l \cdot \l}$ are known as \textbf{blocks} of $\mathcal{O}$. Here, $W_\l\subset W$ is the Weyl group of the rooting system making $\l$ an integral weight while $W_\l\cdot \l$ denotes the orbit of $\l$ under the dot action of the Weyl group $W_\l$. By an \textbf{antidominant weight} $\l\in \mathfrak{h}_K^*$, we mean a weight $\l$ satisfying $\langle \l+\rho, \alpha^\vee\rangle_K \notin \mathbb{Z}_{>0}$ for all $\alpha\in \Phi^+:=\mathbb{Z}^+_0\varPi\cap \Phi$, where $\rho$ is the half-sum of all positive roots. By a \textbf{dominant weight} we mean a weight $\l$ satisfying $\langle \l+\rho, \alpha^\vee\rangle_K \notin \mathbb{Z}_{<0}$ for all $\alpha\in \Phi^+:=\mathbb{Z}^+_0\varPi\cap \Phi$. A crucial fact for our purposes is that the orbit of a weight which is both dominant and antidominant is singular, and so in such a case, the block $\mathcal{O}_{W_\l\cdot\l}$ is semi-simple.

\paragraph{A projective Noetherian algebra associated with a block}

The BGG category $\mathcal{O}$ admits the following integral version, introduced in \citep[1.4]{zbMATH03747378}.

\begin{Def}
		Let $\l\in \mathfrak{h}_R^*$. For each $\l\in \mathfrak{h}_R^*$, we denote by $[\l]$ the set of elements of $\mathfrak{h}_R^*$, $\mu$, that satisfy $\mu-\l\in \L_R$. We define $\mathcal{O}_{[\l], R}$ to be the full subcategory of $U(\mathfrak{g}_R)\M$  whose modules $M$ satisfy the following conditions:
		\begin{enumerate}[(i)]
			\item $M\in U(\mathfrak{g}_R)\m$;
			\item $M=\sum_{\mu\in [\l]} M_\mu$;
			\item $M$ is $U(\mathfrak{n}_R^+)$-locally finite, that is, $U(\mathfrak{n}_R^+)m\in R\m$ for every $m\in M$.
		\end{enumerate}
\end{Def}  Since $R$ is a $\mathbb{Q}$-algebra, the sum in (ii) is actually a direct sum (see \citep[1.4.4]{zbMATH03747378}).

A partial order can be defined on $[\l]$ by imposing $\mu_1\leq \mu_2$ if and only if $\mu_2-\mu_1\in \mathbb{Z}_{\geq 0}\Phi^+\subset \L_R$. 
The Verma modules can be defined over any commutative ring. For each $\mu\in [\l]$, let $R_\mu$ be the free $R$-module with rank one with action $h\cdot1_R=\mu(\l)$ for all $h\in \mathfrak{h}_R$. $R_\mu$ can be regarded as an $U(\mathfrak{b}_R)$-module by imposing that $1_R\otimes x_\alpha$ acts trivially on $R_\mu$ for any $\alpha\in \Phi^+$. The \textbf{Verma module} $\St(\mu)\in \mathcal{O}_{[\l], R}$ is the $U(\mathfrak{g}_R)$-module $U(\mathfrak{g}_R)\otimes_{U(\mathfrak{b}_R)} R_\mu$. It is a direct consequence of their definition that they behave well under change of rings (see \citep[Subsubsection 7.3.8]{p2paper})

Given $\l\in \mathfrak{h}_R^*$, the decomposition $\mathcal{O}_{[\l], R}$ into blocks makes use of $R$ being a local ring and the set $[\l]$ being decomposable into blocks as well. We call $\mathcal{D}\subset [\l]$ a block  of $[\l]$ if $\{\overline{\mu}\colon \mu\in \mathcal{D} \}$ is an orbit under the dot action of the Weyl group $W$ on $[\l]\subset \mathfrak{h}_R^*$. By $\overline{\mu}$ we mean the image of $\mu$ in $\mathfrak{h}_R^*/\mi\mathfrak{h}_R^*$. In particular, the Weyl group also acts on $\mathfrak{h}_{R(\mi)}^*\simeq \mathfrak{h}_R^*/\mi\mathfrak{h}_R^*$ and $\overline{w\cdot \mu}=w\cdot \overline{\mu}$ under the dot action for all $w\in W$, $\mu\in \mathfrak{h}_R^*$. The blocks of $[\l]$ were completely determined in \citep[1.8.2]{zbMATH03747378} and all of them can be written in the form $\mathcal{D}=W_{\overline{\mu}}\cdot \mu +\nu$ with $\mu\in \mathfrak{h}_R^*$, $\nu\in \mi\mathfrak{h}_R^*$.

Let $\mathcal{D}$ be a block of $[\l]$ for some $\l\in \mathfrak{h}_R^*$. For each $\mu\in \mathcal{D}$, there exists a projective object $P(\mu)$ in $\mathcal{O}_{[\l], R}\cap R\Proj$ that maps  epimorphically onto $\St(\mu)$ whose kernel admits a filtration by standard modules $\St(\omega)$ with $\mu<\omega$ (see \citep[Subsubsection 7.3.10]{p2paper}). Define $P_{\mathcal{D}}:=\bigoplus_{\mu\in \mathcal{D}} P(\mu)$.  We define the $R$-algebra $A_{\mathcal{D}}$ to be the endomorphism algebra $\End_{\mathcal{O}_{[\l], R}}\left( P_{\mathcal{D}}  \right)^{op} $.  It is a projective Noetherian $R$-algebra and moreover the module category $A_{\mathcal{D}}(\mi)\m$ is equivalent to a block of the BGG category $\mathcal{O}$ whose composition factors are of the form $L(\overline{\mu})$ with $\mu\in \mathcal{D}$. In particular, every block of the BGG category $\mathcal{O}$ can be obtained in this way.
These algebras possess several nice properties.

\begin{Theorem}\label{Oissplitqh}
	Let $R$ be a local regular commutative Noetherian ring which is a $\mathbb{Q}$-algebra. Let $\mathcal{D}$ be a block of $[\l]$ for some $\l\in \mathfrak{h}_R^*$. The algebra $A_{\mathcal{D}}$ is split quasi-hereditary with standard modules $\St_A(\mu):=\Hom_{\mathcal{O}_{[\l], R}}(P_{\mathcal{D}}, \St(\mu))$, $\mu\in \mathcal{D}$.
	The set $\mathcal{D}$ is a poset with the partial order $\mu_1<\mu_2$ if and only if $\mu_2-\mu_1\in \mathbb{N}\varPi$. \end{Theorem}
\begin{proof}
	See \citep[Theorem 7.3.38]{p2paper}.
\end{proof}
The algebra $(A_{\mathcal{D}}, \{\St_A(\mu)_{\mu\in \mathcal{D}} \})$ is actually a split quasi-hereditary cellular $R$-algebra (see \citep[Theorem 7.3.42]{p2paper}). 

We say that $\omega$ is an antidominant weight of $\mathcal{D}$ if $\overline{\omega}$ is an antidominant weight in $\mathfrak{h}^*_{R(\mi)}$.

\begin{Theorem}\label{integralstruktursatz}
	Let $R$ be a local regular commutative Noetherian ring which is a $\mathbb{Q}$-algebra with unique maximal ideal $\mi$. 
	Let $\mathcal{D}$ be a block of $[\l]$ for some $\l\in \mathfrak{h}_R^*$. Suppose that $\omega$ is the antidominant weight of $\mathcal{D}$. The following assertions hold.
	\begin{enumerate}[(a)]
		\item $(A_{\mathcal{D}}, P_A(\omega))$ is a relative gendo-symmetric $R$-algebra.
		\item \textbf{(Integral Struktursatz) }$(A_{\mathcal{D}}, P_A(\omega))$ is a split quasi-hereditary cover of the commutative $R$-algebra $C:=\End_{A_{\mathcal{D}}}(P_A(\omega))^{op}$.
		\item  If $|\mathcal{D}|=1$, then $\domdim{(A_{\mathcal{D}}, R)}=+\infty$. Otherwise, $\domdim{(A_{\mathcal{D}}, R)} =2$.
		\item If $T$ is a characteristic tilting module of $A_{\mathcal{D}}$, then $2\domdim_{(A_{\mathcal{D}}, R)}T=\domdim {(A_{\mathcal{D}}, R)} $.
	\end{enumerate}
\end{Theorem}
\begin{proof}
See \citep[Theorem 7.3.43, Theorem 7.3.44]{p2paper}.
\end{proof}

Here $C$ is a deformation of a subalgebra of a coinvariant algebra, in the sense that $C(\mi)$ is a subalgebra of a coinvariant algebra. We can consider an integral version of the Soergel's combinatorial functor $$\mathbb{V}_{\mathcal{D}}= \Hom_{A_{\mathcal{D}}}(P_A(\omega), -) \colon A_{\mathcal{D}}\m\rightarrow C\m.$$ This integral version has better properties than the classical one, and in particular, if $\mathcal{D}$ is nice enough, $\mathbb{V}_{\mathcal{D}}$ is actually fully faithful on standard modules. More precisely, we have the following:

\begin{Theorem}\label{HNforcategoryO}
	Fix $t$ a natural number. Let $K$ be a field of characteristic zero. Let $R$ be the localization of $K[X_1, \ldots, X_t]$ at the maximal ideal $(X_1, \ldots, X_t)$. Denote by $\mi$ the unique maximal ideal of $R$. Pick $\theta\in \mathfrak{h}_{R(\mi)}^*\simeq \mathfrak{h}_R^*/\mi \mathfrak{h}_R^*$ to be an antidominant weight which is not dominant. Define $\mu\in \mathfrak{h}_R^*$ to be a preimage of $\theta$ without coefficients in $\mi$ in its unique linear combination of simple roots. Fix $s$ to be a natural number satisfying $1\leq s\leq \rank_R \mathfrak{h}_R^*$ and $s\leq t$. Consider the block $\mathcal{D}=W_{\overline{\mu}}\cdot \mu +\frac{X_1}{1}\alpha_1+\cdots +\frac{X_s}{1}\alpha_s$, where $\alpha_i\in \varPi$ are distinct simple roots, $i=1, \ldots, s$ and by $\frac{f}{1}$ we mean the image of $f\in K[X_1, \ldots, X_t]$ in $R$.
	Then,  $\HN_{\mathbb{V}_{\mathcal{D}}} \mathcal{F}(\St_A)=s-1$.
\end{Theorem}
\begin{proof}
	See \citep[Theorem 7.3.45, Remark 7.3.47]{p2paper}.
\end{proof}

\paragraph{Ringel self-duality of BGG category $\mathcal{O}$}
It is well known that the blocks of classical BGG category $\mathcal{O}$ are Ringel self-dual. This goes back to the work of Soergel \citep[Corollary 2.3]{zbMATH01267662}. This fact was then reproved independently in \citep[Proposition 4]{MR1785327} using the Enright completion functor. Using the methodology introduced here we can establish a new proof of this fact without using the so-called semi-regular bimodule and without using Enright's completions.

\begin{Theorem}\label{RingelselfdualityofO}
	Let $R$ be a local regular commutative Noetherian ring which is a $\mathbb{Q}$-algebra with maximal ideal $\mi$.  Let $\mathcal{D}$ be a block of $[\l]$ for some $\l\in \mathfrak{h}_R^*$. The split quasi-hereditary $R$-algebra $A_{\mathcal{D}}$ is Ringel self-dual.
\end{Theorem}
\begin{proof}
	Let $\mu\in \mathcal{D}$. Both $\St(\mu)$ and $\Cs(\mu)$ belong to $A_{\mathcal{D}}/J\m$ where $J$ is an ideal admitting a filtration by split heredity ideals and such that $\St(\mu)$ is a projective $A_{\mathcal{D}}/J$-module. Further, since $\Cs(\mu)(\mi)$ is the dual of $\St(\mu)(\mi)$ its socle coincides with the top of $\St(\mu)(\mi)$. Denote by $f$ the non-zero $A_{\mathcal{D}}/J(\mi)$-homomorphism $\St(\mu)(\mi)\twoheadrightarrow \top \St(\mu)(\mi)\hookrightarrow \Cs(\mu)(\mi)$. As $\St(\mu)$ is a projective object in $A_{\mathcal{D}}/J\m$ there exists an $A_{\mathcal{D}}$-homomorphism $\overline{f}$ making the following diagram commutative:
	\begin{equation}
		\begin{tikzcd}
			\St(\mu)\arrow[r, twoheadrightarrow] \arrow[d, "\overline{f}"]& \St(\mu)(\mi)\arrow[d, "f"] \\
			\Cs(\mu)\arrow[r, twoheadrightarrow] & \Cs(\mu)(\mi)
		\end{tikzcd}.
	\end{equation}
	Consider the Schur functor $F=\Hom_{A_{\mathcal{D}}}(P_A(\omega), -)=\mathbb{V}_{\mathcal{D}}$, where $\omega$ is the antidominant weight. Applying $F$, we obtain the commutative diagram
	\begin{equation}
		\begin{tikzcd}
			F\St(\mu)\arrow[r, twoheadrightarrow] \arrow[d, "F\overline{f}"]& F\St(\mu)(\mi)\arrow[d, "Ff"] \\
			F\Cs(\mu)\arrow[r, twoheadrightarrow] & F\Cs(\mu)(\mi)
		\end{tikzcd}.\label{eqexparttwo46}
	\end{equation}
	Observe that for any $X\in A_{\mathcal{D}}\m$,  \begin{align}
		F(X(\mi))=\Hom_{A_{\mathcal{D}}}(P_A(\omega), X(\mi))\simeq \Hom_{A_{\mathcal{D}}(\mi)}(P_{A_{\mathcal{D}}(\mi)}(\overline{\omega}), X(\mi))
	\end{align}(see for example \citep[Lemma 2.1.1]{p2paper}) and $Ff$ is isomorphic to the map $\Hom_{A_{\mathcal{D}}(\mi)}(P_{A(\mi)}(\overline{\omega}), f)$ which is non-zero since $\top P_{A(\mi)} (\overline{\omega})$ is the image of $f$. Moreover, $F\St(\mu)(\mi)\simeq R(\mi)$ and $F\Cs(\mu)(\mi)\simeq R(\mi)$. Hence, $Ff$ is an isomorphism. Applying $R(\mi)\otimes_R -$ to the diagram (\ref{eqexparttwo46}) we obtain that $F\overline{f}(\mi)$ is an isomorphism. Since both $F\St(\mu), F\Cs(\mu)\in R\proj$, Nakayama's Lemma yields that $F\overline{f}$ is an isomorphism. This shows that \begin{align}
		F\St(\mu)\simeq F\Cs(\mu), \quad \forall \mu\in \mathcal{D}.\label{eqexemplosparttwo48}
	\end{align}
	The results in Theorem \ref{HNforcategoryO} hold if we replace the complex numbers by any field of characteristic zero. Fix, for a moment $R=K[X_1, X_2]_{(X_1, X_2)}$ and $\mathcal{D}$ to be the block $W_{\overline{\mu}}\mu +\frac{X_1}{1}\alpha_1+\frac{X_2}{1}\alpha_2$, where $\alpha_1, \alpha_2$ are distinct simple roots (so we are excluding the case $\mathfrak{g}=\mathfrak{sl}_2)$, where $\mu\in \mathfrak{h}_R^*$ is a preimage of an antidominant weight in $ \mathfrak{h}_{R(\mi)}^*$ which is not dominant without coefficients in $\mi$ in its unique linear combination of simple roots. Hence, we are excluding the simple blocks which are trivially Ringel self-dual. By Theorem \ref{HNforcategoryO}, $(A_{\mathcal{D}}, P_A(\omega))$ is a $1$-$\mathcal{F}(\Stsim)$ cover of $C$. 
	
	Let $T$ be a characteristic tilting module of $A_{\mathcal{D}}$.
	We claim that $(R(A_{\mathcal{D}}), \Hom_{A_{\mathcal{D}}} (T, P(\omega))$ is a $1$-$\mathcal{F}(\Stsim_{R})$ cover of $C$, where $R(A_{\mathcal{D}})$ denotes the Ringel dual of $A_{\mathcal{D}}$. In the proof of Theorem 7.3.45 of \cite{p2paper} (replacing $\mathbb{C}$ by $K$), we observe that, for any prime ideal $\pri$ of $R$ with height at most one, $Q(R/\pri)\otimes_R A_{\mathcal{D}}$ is semi-simple because the weights in $Q(R/\pri)\otimes_R \mathcal{D}$ are both dominant and antidominant. Here, $Q(R/\pri)$ denotes the quotient field of $R/\pri$. 
	Therefore, the Ringel dual of $Q(R/\pri)\otimes_R A_{\mathcal{D}}$, $Q(R/\pri)\otimes_R R(A_{\mathcal{D}})$, according to Propositions \ref{qhproperties} and \ref{standardscotiltingsreductiontofields}, is semi-simple for any prime ideal $\pri$ of $R$ with height at most one. Therefore,
	\begin{align}
		Q(R/\pri)\otimes_R P(\omega)\rcodomdim_{Q(R/\pri)\otimes_R A_{\mathcal{D}}} Q(R/\pri)\otimes_R T=+ \infty
	\end{align}
	and $(Q(R/\pri)\otimes_R R(A_{\mathcal{D}}), Q(R/\pri)\otimes_R\Hom_{A_{\mathcal{D}}} (T, P(\omega))$ is a $+ \infty-$faithful split quasi-hereditary cover of \linebreak$Q(R/\pri)\otimes_R C$ for every prime ideal $\pri$ of $R$ with height at most one. 
	By Theorem \ref{integralstruktursatz} and Proposition \ref{dominantandcodominantoftiltingcomparison}, $P(\omega)\rcodomdim_{(A_{\mathcal{D}}, R )} T=1$. By Corollary \ref{deformationringeldualfunctor}, $(R(A_{\mathcal{D}}), \Hom_{A_{\mathcal{D}}} (T, P(\omega))$ is a $0$-$\mathcal{F}(\Stsim_{R})$ cover of $C$. So, we can apply Theorem \ref{improvingcoverwithspectrum} to obtain 
	that $(R(A_{\mathcal{D}}), \Hom_{A_{\mathcal{D}}} (T, P(\omega)))$ is a $1$-$\mathcal{F}(\Stsim_{R})$ cover of $C$. 
	Now, rewriting (\ref{eqexemplosparttwo48}) we obtain
	\begin{align}
		F\St(\theta)\simeq F\Cs(\theta)=\Hom_{A_{\mathcal{D}}}(P(\omega), \Cs(\theta))\simeq \Hom_{R(A_\mathcal{D})}(\Hom_{A_{\mathcal{D}}}(T, P(\omega)), \Hom_{A_{\mathcal{D}}}(T, \Cs(\theta)) ).
	\end{align}
	By Corollary \ref{equivalenceofoneuniqueness}, $A_{\mathcal{D}}$ is Ringel self-dual. That is, there exists an equivalence of categories \linebreak $H\colon A_{\mathcal{D}}\m\rightarrow R(A_{\mathcal{D}})\m$ preserving the highest weight structure. Applying $R(\mi)\otimes_R -$ to $H$ we obtain that $A_{\mathcal{D}}(\mi)$ is Ringel self-dual. That is, the blocks of category $\mathcal{O}$ over a field of characteristic zero are Ringel self-dual. We excluded the case $\mathfrak{g}=\mathfrak{sl}_2$. But the non-simple blocks of the category $\mathcal{O}$ associated with $\mathfrak{sl}_2$ are Morita equivalent to the Auslander algebra of $K[x]/(x^2)$ which is Ringel self-dual. 
	
	Return to the general case of $R$ being an arbitrary local regular commutative Noetherian ring which is a $\mathbb{Q}$-algebra and $\mathcal{D}$ an arbitrary block. Since $R(\mi)$ is a field of characteristic zero, $A_{\mathcal{D}}(\mi)$ is Ringel self-dual. By \citep[Lemma 7.4]{appendix}, $A_{\mathcal{D}}$ is Ringel self-dual.
\end{proof}

\subsubsection{Schur algebras} \label{RdSchur algebras}

By \citep[Proposition 3.7]{zbMATH00549737}, the Schur algebras $S_{\mathbb{Z}}(d, d)$ are Ringel self-dual, and therefore the Schur algebra $S_R(n, d)$ is also Ringel self-dual for an arbitrary commutative Noetherian ring $R$.

The approach presented in Theorem \ref{RingelselfdualityofO} can also be applied to Schur algebras. However, we have to exclude the case of characteristic two for similar reasons why we excluded the case $\mathfrak{sl}_2$ for the general cases in the category $\mathcal{O}$.

\begin{Theorem}Assume that $n\geq d$ are natural numbers.
	The Schur algebra $S_{\mathbb{Z}[\frac{1}{2}]}(n, d)$ is Ringel self-dual.
\end{Theorem}
\begin{proof}
	The quotient field of $\mathbb{Z}[\frac{1}{2}]$ is $\mathbb{Q}$. So, for $S_{\mathbb{Z}[\frac{1}{2}]}(n, d)$ conditions (i) and (ii) of Corollary \ref{Ringelselfdualcodomdim} hold.  Condition (iii) follows by Proposition \ref{dominantandcodominantoftiltingcomparison} and Theorem \ref{dominantdimensionquantumschur} with $q=1$. Now we focus on the image of the dual of Weyl modules under the Schur functor $F=F_{V^{\otimes d}}$. Fix $R=\mathbb{Z}[\frac{1}{2}]$. We can see that, for any $\l\in \L^+(n, d)$
	\begin{align}
		F\Cs(\l)&\simeq \xi_{(1, \ldots, d), (1, \ldots, d)}\Cs(\l)\simeq \xi_{(1, \ldots, d), (1, \ldots, d)} D\St(\l)^\iota \simeq D(\xi_{(1, \ldots, d), (1, \ldots, d)}\St(\l))^\iota \\&\simeq D\theta(\l)^\iota\simeq \sgn\otimes_R \theta(\l').
	\end{align}The last isomorphism is \citep[Lemma 1.6.12]{zbMATH00971625} and $\l'$ is the conjugate partition of $\l$. Here $\sgn$ is the free $R$-module with rank one with the action the sign of the permutation $\sigma\cdot 1_R=\sgn(\sigma)$, $\sigma\in S_d$. Moreover, $RS_d$ acts on $\sgn\otimes_R M$ through the diagonal action. Hence, $\sgn\otimes_R \otimes_R M\simeq M$ for any $M\in RS_d\m$. Hence, $\sgn\otimes_R -\colon RS_d\m\rightarrow RS_d\m$ is an isomorphism of categories. Therefore, \begin{align}
		F\Cs(\l)\simeq \sgn\otimes_R \theta(\l')\simeq \sgn\otimes_R F\St(\l'), \ \forall \l\in \L^+(n, d),
	\end{align}and $\mathcal{F}(F\Cs)\simeq \mathcal{F}(F\St)$. By Corollary \ref{Ringelselfdualcodomdim}, the result follows.
\end{proof}

\section*{Acknowledgments}
Most of the results of this paper are contained in the author's PhD thesis \citep{thesis}, financially supported by \textit{Studienstiftung des Deutschen Volkes}. The author would like to thank Steffen Koenig for all the conversations on these topics, his comments and suggestions towards improving this manuscript.
The author thanks Volodymyr Mazorchuk for bringing the work \cite{MR1785327} to the author's attention. The author thanks Karin Erdmann for interesting discussions about this work.

\begingroup
\linespread{0.5}
{\footnotesize
\bibliographystyle{alphaurl}
\bibliography{bibarticle} }
\endgroup

\Address
\end{document}